\documentclass[preprint,12pt]{elsarticle}

\usepackage{amsmath,amssymb,amsthm,enumerate} 
\usepackage{microtype}

\usepackage{xcolor}
\usepackage{marginnote}
\usepackage{bm}

\numberwithin{equation}{section} 

\newtoggle{lmargin}
\newcommand{\alternatingtodo}[2][]{%
\iftoggle{lmargin}%
{%
\todo[#1]{#2}
\togglefalse{lmargin}
}{%
{%
\let\marginpar\marginnote%
\reversemarginpar%
\todo[#1]{#2}%
}%
\toggletrue{lmargin}%
}%
}%
\usepackage[colorinlistoftodos,textwidth=1\marginparwidth]{todonotes}

\theoremstyle{theorem}
\newtheorem{theorem}{Theorem}[section]
\newtheorem{proposition}[theorem]{Proposition}
\newtheorem{lemma}[theorem]{Lemma}
\newtheorem{corollary}[theorem]{Corollary}

\newtheorem{deflem}[theorem]{Definition-Lemma}
\theoremstyle{definition}
\newtheorem{definition}[theorem]{Definition}
\newtheorem{remark}[theorem]{Remark}
\newtheorem{example}[theorem]{Example}
\newtheorem*{sumprob}{Summability Problem}

\definecolor{e-mail}{rgb}{0,.40,.80}
\definecolor{reference}{rgb}{.20,.60,.22}
\definecolor{citation}{rgb}{0,.40,.80}
\usepackage[pdfstartview=FitH, colorlinks, linkcolor=reference, citecolor=citation, urlcolor=e-mail]{hyperref}

\usepackage{ifpdf}
\ifpdf
  \usepackage{hyperref}
\fi

\newcommand{\Cq}{\mathcal{C}} 
\newcommand{\Ca}{\mathcal{C}} 
\newcommand{\bk}{\mathcal{K}} 
\newcommand{\bkl}{\bk_{\Lambda}} 
\newcommand{\bkq}{\bk_{q}} 
\newcommand{\bke}{{\bk}} 
\newcommand{\ec}{\mathcal{E}} 
\newcommand{\idec}{{o}} 
\newcommand{\ecl}{\ec_{\Lambda}} 
\newcommand{\orbitl}{{\ecl/\tau}} 
\newcommand{\ecq}{\ec_{q}}
\newcommand{\orbitq}{{\ecq/\tau}} 
\newcommand{\orbita}{{\ec/\tau}} 
\newcommand{\C}{\mathbb{C}} 
\newcommand{\Z}{\mathbb{Z}} 
\newcommand{\wpl}{\wp_{\Lambda}}
\newcommand{\zl}{\zeta_{\Lambda}}
\newcommand{\zq}{\zeta_{q}}
\newcommand{\zqd}[1]{\zq^{({#1})}} 
\newcommand{\za}[2]{\zeta^{\mathcal{U}}_{( {#1} , {#2} )}}
\DeclareMathOperator{\ores}{Ores}
\DeclareMathOperator{\pano}{PanOres}

\journal{}

\begin{document}

\begin{frontmatter}

\title{Panorbital residues and elliptic summability}
\author{Carlos E. Arreche} \ead{arreche@utdallas.edu} \author{Matthew W. Babbitt} \ead{mbabbitt2003@gmail.com}
\address{Department of Mathematical Sciences, The University of Texas at Dallas}


\begin{abstract}

 For $\tau$ the translation automorphism defined by a non-torsion point in an elliptic curve, we consider the elliptic summability problem of deciding whether a given elliptic function $f$ is of the form $f=\tau(g)-g$ for some elliptic function $g$. We introduce two new panorbital residues and show that they, together with the orbital residues introduced in 2018 by Dreyfus, Hardouin, Roques, and Singer, comprise a complete obstruction to the elliptic summability problem. The underlying elliptic curve can be described in any of the usual ways: as a complex torus, as a Tate curve, or as a one-dimensional abelian variety. We develop the necessary results from scratch intrinsically within each setting; in the last two of them, we also work in arbitrary characteristic. We include several basic concrete examples of computation of orbital and panorbital residues for some summable and non-summable functions in each setting. We conclude by applying the technology of orbital and panorbital residues to obtain several new results of independent interest.

\end{abstract}


\end{frontmatter}

\tableofcontents

\section{Introduction} 
\vspace{-.04in}
For an elliptic curve $\ec$ defined over a field $\Ca$ and a $\Ca$-point $s$ of $\ec$, addition by $s$ under the elliptic group law $\oplus$ defines an automorphism of $\ec$, which in turn defines a $\Ca$-linear automorphism $\tau$ of the function field $\bk$ of $\ec$. This $\tau$ is of infinite order precisely when $s$ is a non-torsion point. Our main goal is to develop a complete set of $\Ca$-linear obstructions to the following. \vspace{-.02in}
\begin{sumprob}
    Given $f\in \bk$, decide whether there exists $g\in\bk$ such that $f=\tau(g)-g$. \vspace{-.02in}
\end{sumprob} 
\noindent If the answer is affirmative, we say $f$ is \emph{(elliptically) summable}.

The basic summability problem has an interesting (though classical and well-understood) solution when $s$ is a torsion point, say of order $N$. Indeed, in this case $\bk$ is a cyclic Galois extension of degree $N$ over the fixed field $\bk^\tau=:\bk'$, which is identified with the function field of another elliptic curve $\ec'$ such that the inclusion $\bk'\hookrightarrow \bk$ corresponds to an isogeny $\ec\twoheadrightarrow\ec'$ of order $N$ whose kernel is precisely $\Z s=\{\idec,s,\dots,(N-1)s\}$, where we denote by $\idec$ the identity for the elliptic group law. Therefore we can identify the points of $\ec'$ with the set $\orbita$ of $\Z s$-orbits of points in $\ec$. Since the Galois group $\mathrm{Gal}(\bk/\bk')=\langle \tau\rangle$, it follows from Hilbert's Theorem~90 that a given $f\in\bk$ is summable if and only if the Galois-theoretic trace $\mathrm{Tr}_{\bk/\bk'}(f)=\sum_{n=0}^{N-1}\tau^n(f)=0$. Thus when $s$ is a torsion point the Galois-theoretic trace $\mathrm{Tr}_{\bk/\bk'}(f)$ is a complete obstruction to summability.

However, if $s$ is a non-torsion point then just about every aspect of the classical situation above breaks down. The fixed field $\bk^\tau=\bk'$ is just $\Ca$, over which $\bk$ is not a finite Galois extension. The quotient map $\ec\twoheadrightarrow \ec':=\orbita$ is only a homomorphism of abstract abelian groups, and one cannot na\"ively think of the set of orbits $\orbita$ as a reasonable variety or manifold. We cannot appeal to Hilbert's Theorem~90, and in any case the ``trace'' $\mathrm{Tr}_{\bk/\bk'}(f)=\sum_{n\in\Z}\tau^n(f)$ is utterly meaningless (but see \eqref{eq:intro-ores} and \eqref{eq:intro-pano-1} below). Thus when $s$ is a non-torsion point, there is no classical source of obstructions to summability.

But the summability problem is most interesting and fruitful \emph{precisely} when $s$ is a non-torsion point. 
We view this as a fundamental problem in the analytic, algebro-geometric, and arithmetic study of elliptic curves and elliptic functions at large --- interesting in its own right and also having many potential applications to other domains. Let us briefly illustrate the potential range of its interactions with other areas of mathematics by discussing explicitly two recent examples from the literature: elliptic hypergeometric functions (mathematical physics), and walks in the quarter plane (combinatorics).

The elliptic hypergeometric functions were introduced in \cite{Spiridonov2008} as generalizations of the classical Euler-Gauss hypergeometric functions $\phantom{.}_2F_1(a,b,c;z)$. One begins with $q,s\in\C^\times$ such that $|q|,|s|<1$ and $s^\Z\cap q^\Z=\{1\}$, so that $s$ represents a non-torsion point on the Tate elliptic curve $\ecq=\C^\times/q^\Z$. Like their classical counterparts, the elliptic hypergeometric functions $F_{\boldsymbol{\varepsilon}}(z)$ depend on parameters $\boldsymbol{\varepsilon}=(\varepsilon_1,\dots,\varepsilon_7)\in(\C^\times)^7$, which may be chosen freely but for the balancing condition $\bigl(\prod_{j=1}^6\varepsilon_j)\varepsilon_7^2=q^2s$. The role of the classical hypergeometric differential equation is played by a certain second-order scalar linear difference equation satisfied by $F_{\boldsymbol{\varepsilon}}(z)$ with respect to the automorphism $\tau(f(z))=f(sz)$. The coefficients of this difference equation belong to the field $\bkq$ of (multiplicatively) $q$-periodic functions on $\C^\times$, which is identified with the function field of $\ecq$, so that the automorphism $\tau$ corresponds to addition by the point $\check{s}:=s\cdot q^\Z\in\ecq$ with respect to the elliptic group law as discussed above. In \cite{Arreche2021} were given a pair of criteria guaranteeing that the non-trivial solutions to certain second order difference equations, such as the one satisfied by $F_{\boldsymbol{\varepsilon}}(z)$, do not satisfy any non-trivial differential equations over the base field ($\bkq$, in this case). One of these criteria demands the non-summability of a certain (infinite) set of elements of $\bkq$ extracted from the difference equation satisfied by $F_{\boldsymbol{\varepsilon}}(z)$. Using these criteria it was possible to show in \cite{Arreche2021} that, for ``sufficently generic'' parameters $\boldsymbol{\varepsilon}\in(\C^\times)^7$, all of these summability problems have negative answers. Thus the summability problem  played a fundamental role in the proof in \cite{Arreche2021} of the \emph{hypertranscendence} of ``most'' elliptic hypergeometric functions. 

With a subset $\mathcal{D}\subseteq \{-1,0,1\}^2-\{(0,0)\}$ of permissible cardinal directions is associated the generating series $F_\mathcal{D}(x,y,t)=\sum_{i,j,k}a_{i,j,k}x^iy^jt^k$, where $a_{i,j,k}$ counts how many sequences of $k$ steps from $\mathcal{D}$ define a lattice walk beginning at the origin that ends at the point $(i,j)$ and remains confined to the first quadrant. The $a_{i,j,k}$ satisfy a recursion that is encoded into a functional equation for $F_\mathcal{D}(x,y,t)$. As explained in \cite{Dreyfus2018,Hardouin2021,dreyfus-hardouin:2021,dreyfus:2023} and the references therein, for certain choices of $\mathcal{D}$, certain specializations of this functional equation realize the corresponding specializations of $F_\mathcal{D}(x,y,t)$ as solutions to difference equations of the form $\tau(F)-F=f$. In some of the most interesting cases, the automorphism $\tau$ is interpreted as above as induced by the addition of a certain non-torsion point $s$ of an elliptic curve $\ec$ defined over\footnote{
The variable $t$ is sometimes specialized to a convenient complex number as in \cite{Dreyfus2018}, sometimes allowed to vary to obtain an elliptic pencil as in \cite{Hardouin2021}, and sometimes kept as a formal parameter that can be differentiated as in \cite{dreyfus-hardouin:2021,dreyfus:2023}.} $\mathbb{Q}(t)$, and the coefficient $f$ belongs to the function field $\bk$ of $\ec$ as above. The first systematic study of the properties of the summability problem  was undertaken in \cite[Appendix~B]{Dreyfus2018} for the utilitarian purpose of elucidating the differential nature of $F_\mathcal{D}(x,y,t)$, a program which has been successfully continued in \cite{Dreyfus2018,Hardouin2021,dreyfus-hardouin:2021,dreyfus:2023} and by many other authors over a long time using a wide variety of algebraic, analytic, and combinatorial techniques. We again refer to the above references for more details on this interesting story, to which we are ill-suited to do justice.

Working under the assumption that the field $\Ca$ is algebraically closed of \underline{characteristic zero}, the authors of \cite{Dreyfus2018} develop a \underline{partial} obstruction to the summability problem, in the form of \emph{orbital residues}. As they explain, this terminology is suggestive but potentially misleading. A given $f\in\bk$ does not have any residues to speak of. But, choosing once and for all a suitable differential form $\varpi$, one can speak of the residues of $f\varpi$ and abusively think of $\mathrm{res}(f\varpi,\alpha)$ as ``the residue of $f$ at the point $\alpha$''. Similarly, in order to speak of the orbital residues of $f$, one must first make some choices. Thus in \cite{Dreyfus2018} is introduced the notion of a set $\mathcal{U}$ of $\tau$-compatible local uniformizers (see the Definition~\ref{def:uniformizers} as well as a discussion of this important notion in Remark~\ref{rem:outer-uniformizers}). Having chosen such a $\mathcal{U}$, one can define local principal part coefficients $c^\mathcal{U}_k(f,\alpha)$ of $f$ at each $\alpha$ in $\ec$ for $k\geq 1$, all but finitely many of which are non-zero. The \emph{orbital residue} of $f$ at $\tau$-orbit $\omega\in\orbita$ of order $k$ (relative to $\mathcal{U}$) is \vspace{-.04in}\begin{equation}\label{eq:intro-ores}\ores^\mathcal{U}(f,\omega,k):=\sum_{\alpha\in\omega}c_k^\mathcal{U}(f,\alpha). \vspace{-.04in}\end{equation} It is easy to see that if $f$ is summable then every orbital residue of $f$ vanishes (because they can be shown to be additive $\tau$-invariant). On the other hand, the vanishing of all the orbital residues of $f$ is not sufficient\footnote{One should be careful not to read too quickly the misprint in \cite[Prop.~A.4]{Hardouin2021}, which is missing a necessary hypothesis from the \cite[Lem.~3.7]{Hardouin2021} cited in its proof. This hypothesis is satisfied in every application of \cite[Prop.~A.4]{Hardouin2021} made within \cite{Hardouin2021}.} to guarantee the summability of $f$. In order to simplify the exposition, we still assume as in \cite{Dreyfus2018,Hardouin2021} that $\Ca$ is algebraically closed. But, in contrast, we work over a field $\Ca$ of \underline{arbitrary characteristic}.

The flagship contribution of the present work is the introduction of two additional $\Ca$-linear obstructions to the summability problem, which we call \emph{panorbital residues} because they amalgamate the local data from all the orbits. To define them, one must add to the choice of $\mathcal{U}$ another choice $\mathcal{R}$, consisting of one representative $\alpha_\omega$ in $\ec$ for each orbit $\omega\in\orbita$. We call the pair of (sets of) choices $\mathbf{P}=(\mathcal{R},\mathcal{U})$ a $\tau$-\emph{pinning} of $\ec$ (see Definition~\ref{def:pinninga}). Both panorbital residues of $f$ (relative to $\mathbf{P}$) are given explicitly in Definition~\ref{def:pano-a} as certain linear combinations of the same $c^\mathcal{U}_k(f,\alpha)$ used to define the orbital residues of $f$ in \eqref{eq:intro-ores}. One of our new obstructions admits the following concrete interpretation. As we show in Proposition~\ref{prop:diff-pano}, there exists a unique $\tau$-invariant meromorphic differential form $\varpi$ on $\ec$ such that the panorbital residue of order $1$ of $f$ is given by the ``twisted'' sum of classical residues \vspace{-.04in}\begin{equation}\label{eq:intro-pano-1}
    \pano^\mathbf{P}(f,1)=\sum_{\omega\in\orbita}\sum_{n\in\Z}n\cdot\mathrm{res}\bigl(f\varpi,\alpha_\omega\oplus ns\bigr).\vspace{-.04in}
\end{equation} Thus one can harmlessly redefine $\pano^\mathbf{P}(f,1)$ by \eqref{eq:intro-pano-1}, relative to \emph{any} non-zero invariant differential $\varpi$, in place of the official Definition~\ref{def:pano-a} whenever that is convenient (as it often is!). This new obstruction \eqref{eq:intro-pano-1} is also clearly additive, and the more subtle fact that it is also $\tau$-invariant is seen from the vanishing of the sum of all the residues of $f\varpi$ by the Residue Theorem. In contrast, we do not have a similarly compelling and meaningful description of the last obstruction $\pano^\mathbf{P}(f,0)$, which is more technical, difficult to work with, and (to us) still somewhat mysterious. And yet with the definitions that we have we are able to prove our main result, to the effect that the orbital residues of \cite{Dreyfus2018}, together our panorbital residues, comprise a complete obstruction to the summability problem.

\begin{theorem}\label{thm:main}
    The elliptic function $f\in\bk$ is elliptically summable if and only if all its orbital residues and both of its panorbital residues vanish.
\end{theorem}

In Theorem~\ref{thm:dhrs-additive} below we state succinctly and precisely what is proved in \cite{Dreyfus2018} concerning an $f\in\bk$ all of whose orbital residues vanish, and in Remark~\ref{rem:dhrs-additive} we carefully explain how their results and methods compare with ours. So let us here give only a brief and informal summary. When all the orbital residues of $f$ vanish, then $f$ is ``nearly summable'' in the sense that $f=\sigma(g)-g+h$ for some $g,h\in\bk$ such that $h$ has at worst two poles of order $1$ at some $\alpha$ and $\alpha\oplus s$ in $\ec$. By the Riemann-Roch Theorem~\ref{thm:rr}, there is a two-dimensional $\Ca$-vector space of choices for what such an $h$ could be. Our $\pano^\mathbf{P}(f,1)$ and $\pano^\mathbf{P}(f,0)$ can be thought of as the coordinates of the vector space of choices for such an $h$. Indeed, if all (pan)orbital residues of $f$ vanish except possibly for $\pano^\mathbf{P}(f,0)$, then $f$ is ``very nearly summable'' in the sense that there exists (a unique) $g\in\bk$ such that $f=\tau(g)-g+\pano^\mathbf{P}(f,0)$. It is a non-trivial technical challenge to define both panorbital residues (relative to $\mathbf{P}$) uniformly for every $f\in\bk$, regardless of which other (pan)orbital residues of $f$ vanish or not.

We have aimed to write our exposition in a way that is as general as possible while remaining as accessible as possible to whomever might need to utilize (pan)orbital residues for whatever purpose. In service of this goal, we treat from scratch the basic definitions and results concerning (pan)orbital residues and summability intrinsically within each of the three common settings within which elliptic curves are treated, which allows us to minimize the technical prerequisites needed within each setting.

In Section~\ref{sec:weierstrass}, we work over $\Ca=\C$ and consider the elliptic curve $\ecl=\C/\Lambda$ for a lattice $\Lambda\subset \C$, whose function field is identified with the field $\bkl$ of $\Lambda$-periodic meromorphic functions on $\C$. In Section~\ref{sec:tate} we assume that $\Ca$ is complete with respect to a non-archimedean absolute value $| \ \ |$ (but we also allow $\Ca=\C$), and consider the Tate elliptic curve $\ecq=\Cq^\times/q^\Z$ for $q\in\Cq^\times$ such that $|q|<1$, whose function field is identified with the field $\bkq$ of (multiplicatively) $q$-periodic meromorphic functions on $\Cq^\times$. In Section~\ref{sec:algebraic} we treat the most general setting where $\ec$ is an elliptic curve over $\Ca$ in the sense of algebraic geometry (a connected smooth projective curve of genus $1$ with a distinguished $\idec\in\ec(\Ca)$). 

Although these three sections are logically independent of one another, we suggest that the best way to read them is side by side by side. This is because an earlier section is comparatively more conceptual and less technical than a later one, whereas a later section contextualizes the results and constructions of an earlier one and makes more apparent which earlier aspects were essential versus incidental. Our (sometimes onerous) notation is chosen carefully so as to make as salient as possible the analogies that bind together all three settings. 
Each of these sections concludes with several basic examples where we showcase the computation of (pan)orbital residues, in both summable and non-summable cases. Although they are designed to be ``the same'' examples across all three settings, they show how the formal yoga of (pan)orbital residues that applies essentially uniformly across all three settings devolves into rather case-specific sorts of computations in practice ``on the ground''.

In the final Section~\ref{sec:applications} we apply our panorbital residues together with the orbital residues of \cite{Dreyfus2018} to prove several new results. Included among these is a new simple formula for the dimension of the subspace $\mathcal{S}(D)$ consisting of the summable elements in the Riemann-Roch space $\mathcal{L}(D)$ for an effective divisor $D$ in Theorem~\ref{thm:summable-rr} (which we think of as a kind of analogue of the Riemann-Roch Theorem~\ref{thm:rr} for divisors of positive degree). We are also able to prove several new characterizations of various important properties of some small-order systems of linear difference equations over $\ec$ in \S\ref{sec:integrability}. In particular, we obtain in Theorem~\ref{thm:additive-integrability} a surprising characterization of the differential integrability of the fundamental difference equation $\tau(y)-y=f$ in an unknown $y$, which has an amusing twist in positive characteristic and which already in characteristic zero sharpens a fundamental result from \cite{Dreyfus2018} (see Remark~\ref{rem:dhrs-additive}). We obtain a similarly surprising characterization of the differential integrability of the fundamental difference equation $\tau(y)=ay$ with $a\in\bk^\times$, but only in characteristic zero (Theorem~\ref{thm:dlog-integrability-0}), or under a suitable non-integrality hypothesis on the $j$-invariant of $\ec$ in arbitrary characteristic (Corollary~\ref{cor:dlog-integrability-0}). Finally, in Theorem~\ref{thm:constant-gauge-equivalent} we characterize when the difference equation $\tau(y)=ay$ is gauge equivalent to a constant difference equation $\tau(y)=cy$ based only on the divisor of $a$. This is a strictly stronger property than the differential integrability of $\tau(y)=ay$, as we first learned from \cite{hardouin-roques:2023}. Although this last Theorem~\ref{thm:constant-gauge-equivalent} makes no use of (pan)obital residues, their heavy influence in both its statement and its proof is impossible to miss.

\section{Weierstrass setting}\label{sec:weierstrass}

\subsection{Reminders and notation}

Given a lattice $\Lambda\subset \C$, we denote by $\bkl$ the field of $\Lambda$-periodic meromorphic functions on $\C$, which is identified with the field of meromorphic functions on the elliptic curve $\ecl:=\mathbb{C}/\Lambda$. We denote the projection map $\C\rightarrow\ecl:t\mapsto \check{t}:=t+\Lambda$. As usual, the Weierstrass $\zeta$-function and the Weierstrass $\wp$-function are the meromorphic functions on $\C$ given by
\begin{equation}\label{eq:zeta-wp-defs}
\begin{gathered}
\zl(z)=\frac{1}{z}+\sum_{\lambda\in\Lambda-\{0\}}\left(\frac{1}{z-\lambda}+\frac{1}{\lambda}+\frac{z}{\lambda^2}\right);\qquad\text{and}\\ 
\wpl(z)=\frac{1}{z^2}+\sum_{\lambda\in\Lambda-\{0\}} \left(\frac{1}{(z-\lambda)^2}-\frac{1}{\lambda^2}\right)=-\zl'(z).
\end{gathered} \end{equation}
It is well-known that $\wpl(z)$ is $\Lambda$-periodic, and that $\bkl=\C(\wpl,\wpl')$ is a quadratic extension of its subfield $\C(\wpl)$, due to the Weierstrass relation \begin{equation}\label{eq:weierstrass-de} \wpl'(z)^2=4\wpl(z)^3-g_{2,\Lambda}\wpl(z)-g_{3,\Lambda}\end{equation} for certain constants $g_{2,\Lambda},g_{3,\Lambda}\in\C$ that depend only on $\Lambda$ and satisfy $g_{2,\Lambda}^3-27g_{3,\Lambda}^2\neq 0$. Although \mbox{$\zl(z)\notin \bkl$} is not $\Lambda$-periodic, each difference \mbox{$\zl(z+\lambda)-\zl(z)=\eta_\Lambda(\lambda)\in\C$} is a constant that depends only on $\lambda\in\Lambda$ but not on $z$. This immediately implies the following basic result.
\begin{lemma}\label{lem:l-zeta-difference}
    For each $t\in\C$, the difference $\zl(z-t)-\zl(z)\in\bkl$ is \mbox{$\Lambda$-periodic}. Moreover, if $t\notin\Lambda$, then $\zl(z-t)-\zl(z)\in\bkl$ has simple poles at $\check{t}$ and at $\check{0}$ on $\ecl$, and is holomorphic elsewhere in $\C$.
\end{lemma}

Given any meromorphic function $f(z)$ on $\C$, we write $c_k(f,\alpha)\in\C$ for the coefficients occurring in the principal part of the Laurent series expansion of $f(z)\in\C((z-\alpha))$ at $\alpha\in\C$: \begin{equation}\label{eq:laurent-coeffs-l}f(z)=\sum_{k\geq 1}\frac{c_k(f,\alpha)}{(z-\alpha)^k}+h(z), \qquad\text{with} \quad h(z)\in\C[[z-\alpha]].\end{equation} We define the $c_k(f,\alpha)$ in this way for every $k\in\mathbb{N}$ and every $\alpha\in\C$, even though all but finitely many of them are zero for each given $\alpha\in\C$. We note that if $f(z)\in\bkl$ then the $c_k(f,\alpha)$ depend only on the point $\check{\alpha}\in\ecl$, which we occasionally emphasize by writing $c_k(f,\check{\alpha})$ in this case.

As in the introduction, we denote by $\tau$ the automorphism of $\bkl$ given by $\tau(f)(z)=f(z+s)$ for a fixed $s\in\C-\bigcup_{n\geq 1}\frac{1}{n}\Lambda$, or equivalently, such that $\check{s}$ is a non-torsion point of $\ecl$. The set of $\tau$-\emph{orbits} in $\ecl$ is \[\orbitl:=\ecl/\mathbb{Z}\check{s}.\]
The following definition is a complex analytic reinterpretation of that of \cite[Def.~B.7]{Dreyfus2018} (cf.~\cite[Rem.~B.10]{Dreyfus2018} and \cite[Rem.~A.8]{Hardouin2021}).

\begin{definition}[Orbital residues - Weierstrass version]\label{def:oresl} The \emph{orbital residue} of $f(z)\in\bkl$ at the orbit $\omega\in\orbitl$ of order $k\in\mathbb{N}$ is \[\ores(f,\omega,k):=\sum_{\check{\alpha}\in\omega}c_k(f,\check{\alpha}).\]
\end{definition}

We emphasize that the finite sum above is being taken over the points $\check{\alpha}\in\omega\subset\ecl$, rather than over the set of pre-images of these points in $\C$, which would render the definition meaningless.

\subsection{Pinnings and their resulting zeta-expansions}

The remaining notions and results in this section are all stated relative to the following arbitrary set of choices. 

\begin{definition}[$\tau$-pinnings - Weierstrass version]\label{def:pinningl}
    A $\tau$-\emph{pinning} of $\ecl$ is a choice $\mathcal{R}=\{\alpha_\omega \in\C\ | \ \omega\in\orbitl\}$ of representatives with $\check{\alpha}_\omega\in\omega$ for each $\omega\in\orbitl$.
\end{definition}

The following is a variation on the well-known expansion of $\Lambda$-periodic functions in terms of shifted Weierstrass $\zeta$-functions and their derivatives, adapted to our present purposes.

\begin{proposition}\label{prop:zetaExp-l}
    Let $\mathcal{R}=\{\alpha_\omega \ | \ \omega\in\orbitl\}$ be any $\tau$-pinning of $\ecl$, and let $f(z)\in\bkl$. There exists a unique constant $c_0^\mathcal{R}(f)\in\mathbb{C}$ such that
    \begin{equation} \label{eq:zetaExp-l}
        f(z) = c_0^{\mathcal{R}}(f)+\sum_{\omega\in\orbitl}\sum_{n\in\mathbb{Z}}\sum_{k\geq 1}\frac{(-1)^{k-1}c_k(f,\alpha_\omega+ns)}{(k-1)!}\zl^{(k-1)}(z-\alpha_\omega-ns)
    \end{equation} as meromorphic functions on $\C$.
\end{proposition}

\begin{proof}
    By construction, the set $\{\alpha_\omega+ns \ | \ \omega\in\orbitl, \ n\in\mathbb{Z}\}$ contains precisely one coset representative in $\C$ of each point in $\ecl$. The rest of the proof proceeds exactly as in \cite[pp.~449-450]{whittakerCourseModernAnalysis2006}, which we briefly recall. Each summand in the right-hand side of \eqref{eq:zetaExp-l} with $k\geq 2$ is $\Lambda$-periodic by \eqref{eq:zeta-wp-defs}. Comparing local expansions, one could immediately reduce to the case where $f(z)$ has only simple poles. The remaining portion of the sum \mbox{$\sum_{\omega,n}c_1(f,\alpha_\omega+ns)\zl(z-\alpha_\omega-ns)$} is seen to be $\Lambda$-periodic by the transformation law of $\zl(z)$ under $\Lambda$-shifts, together with the fact that $\sum_{\omega,n}c_1(f,\alpha_\omega+ns)=0$ by Cauchy's Residue Theorem applied to the differential form $f(z)dz$ over any fundamental parallelogram whose boundary avoids the poles of $f(z)$. It follows that the difference $c_0^\mathcal{R}(f)$ between $f(z)$ and the rightmost sum in \eqref{eq:zetaExp-l} is entire and $\Lambda$-periodic, and therefore it must be constant.
\end{proof}

\subsection{Panorbital residues and summability}

\begin{definition}[Panorbital residues - Weierstrass version] \label{def:pano-l} For a given \mbox{$\tau$-pinning} $\mathcal{R}=\{\alpha_\omega \ | \ \omega\in\orbitl\}$ of $\ecl$, the \emph{panorbital residues} of $f(z)\in\bkl$ of orders $0$ and $1$ relative to $\mathcal{R}$ are respectively
\begin{equation}\label{eq:pano-l-def}
    \begin{gathered}
        \pano^\mathcal{R}(f,0):=c_0^\mathcal{R}(f);\vphantom{\sum_A}\qquad \text{and}\\
    \pano^\mathcal{R}(f,1):=\sum_{\omega\in\orbitl}\sum_{n\in\mathbb{Z}} n\cdot c_1(f,\alpha_\omega+ns).
    \end{gathered}
\end{equation}
\end{definition}

\begin{theorem}\label{thm:mainl} Let $\mathcal{R}=\{\alpha_\omega \ | \ \omega\in\orbitl\}$ be any $\tau$-pinning of $\ecl$, and let $f(z)\in\bkl$. Then $f$ is elliptically summable if and only if $\pano^\mathcal{R}(f,0)=0$, $\pano^\mathcal{R}(f,1)=0$, and $\ores(f,\omega,k)=0$ for every $\omega\in\orbitl$ and $k\geq 1$.
\end{theorem}

\begin{proof}
    First suppose $f(z)=g(z+s)-g(z)$ for some $g(z)\in\bkl$. From the fact that $c_k(\tau(g),\alpha)=c_k(g,\alpha+s)$ for every $\alpha\in\C$ and $k\geq 1$ follows the relation $\ores(\tau(g),\omega,k)=\ores(g,\omega,k)$ for every $\omega\in\orbitl$ and $k\geq 1$, and therefore every $\ores(f,\omega,k)=0$. Moreover, applying $\tau$ to both sides of the zeta expansion \eqref{eq:zetaExp-l} for $g(z)$, we find that similarly $c^\mathcal{R}_0(\tau(g))=c^\mathcal{R}_0(g)$, and therefore $\pano^\mathcal{R}(f,0)=0$. Finally, we see that \begin{multline*}\pano^\mathcal{R}(\tau(g),1) =\sum_{\omega\in\orbitl}\sum_{n\in\mathbb{Z}}n\cdot c_1\bigl(g,\alpha_\omega+(n+1)s\bigr) \\
    =\pano^\mathcal{R}(g,1)-\sum_{\omega\in\orbitl}\sum_{n\in\mathbb{Z}}c_1(g,\alpha_\omega+ns)=\pano^\mathcal{R}(g,1),\end{multline*} by Cauchy's Residue Theorem applied to the differential form $g(z)dz$ over an appropriate fundamental parallelogram in $\C$. Thus $\pano^\mathcal{R}(f,1)=0$.

    To prove the converse, let us now suppose that every orbital residue and both panorbital residues of $f(z)$ are zero. To simplify notation, let us define \begin{equation}\label{eq:phil}\varphi_{\omega,n,k}^\mathcal{R}(z):=\begin{cases} \frac{(-1)^{k-1}}{(k-1)!}\zl^{(k-1)}(z-\alpha_\omega-ns)\in\bkl &\text{for} \ k\geq 2; \\ \zl(z-\alpha_\omega-ns)-\zl(z-ns)\in\bkl & \text{for} \ k=1.\end{cases}\end{equation}for $\omega\in\orbitl$ and  $n\in\mathbb{Z}$. Let us consider 
    \begin{equation}\tilde{f}(z):=f(z)+\sum_{\omega\in\orbitl}\sum_{n\in\mathbb{Z}}\sum_{k\geq 1}c_k(f,\alpha_\omega+ns)\cdot\bigl(\varphi_{\omega,n,k}^\mathcal{R}(z+ns)-\varphi_{\omega,n,k}^\mathcal{R}(z)\bigr).\label{eq:ftildel-def}
    \end{equation}
    Since $\varphi(z+ns)-\varphi(z)$ is summable for any $\varphi(z)\in\bkl$ and $n\in\mathbb{Z}$, $f(z)$ is summable if and only if $\tilde{f}(z)\in\bkl$ is summable. Moreover, since each $\varphi_{\omega,n,k}^{\mathcal{R}}(z+ns) = \varphi_{\omega,0,k}^{\mathcal{R}}(z)$, writing $f(z)$ as in \eqref{eq:zetaExp-l} yields 
    \begin{align}
    \tilde{f}(z)&=0+\sum_{\omega\in\orbitl}\bigg(\sum_{k\geq 1}\ores(f,\omega,k)\cdot\varphi_{\omega,0,k}^\mathcal{R}(z)  +\sum_{n\in\mathbb{Z}}c_1(f,\alpha_\omega+ns)\cdot\zl(z-ns)\bigg) \notag\\
    &=0+\sum_{\omega\in\orbitl}\sum_{n\in\mathbb{Z}}c_1(f,\alpha_\omega+ns)\cdot\zl(z-ns),\label{eq:ftildel-short}\end{align} where the leftmost $0$'s in \eqref{eq:ftildel-short} arise from our assumptions that $c_0^\mathcal{R}(f)$ and the $\ores(f,\omega,k)$ all vanish. Next let us again simplify notation by defining \begin{gather*}\tilde{c}(n):=c_1\bigl(\tilde{f},ns\bigr)=\sum_{\omega\in\orbitl}c_1(f,\alpha_\omega+ns)\in\C,\end{gather*}
    so that $\tilde{f}(z)=\sum_{n\in\mathbb{Z}}\tilde{c}(n)\zl(z-ns)$ in \eqref{eq:ftildel-short} and $\sum_{n\in\mathbb{Z}}n\tilde{c}(n)=\pano^\mathcal{R}(f,1)$. Let us also define the auxiliary $\Lambda$-periodic functions
    \begin{equation}\label{eq:psil-def}
        \psi_j(z):=
        \begin{cases}
            \zl(z-js)-\zl(z-s)\in\bkl & \text{for} \ j\geq 2;\\
            \zl(z-js)-\zl(z)\in\bkl & \text{for} \ j\leq -1.
        \end{cases}
    \end{equation} We now consider the further reduction of $f(z)$ by summable elements of $\bkl$: 
    \begin{equation}
        \bar{f}(z):=\tilde{f}(z)+\sum_{n\leq -1}\sum_{j=n}^{-1}\tilde{c}(n)\cdot\bigl(\psi_j(z-s)-\psi_j(z)\bigr)+\sum_{n\geq 2}\sum_{j=2}^{n}\tilde{c}(n)\cdot\bigl(\psi_j(z+s)-\psi_j(z)\bigr).\label{eq:fbarl-def}
    \end{equation}
    Once again we see that $f(z)$ is summable if and only if $
    \bar{f}(z)$ is summable. Moreover, it follows from the definitions~\eqref{eq:psil-def} that, for $n\geq 2$,
    \begin{equation}\label{eq:psil-comp1}
        \sum_{j=2}^n\bigl(\psi_j(z+s)-\psi_j(z)\bigr)=-\zl(z-ns)+n\zl(z-s)-(n-1)\zl(z);\end{equation} and similarly for $n\leq -1$ that
        \begin{equation}\label{eq:psil-comp2}\sum\limits_{j=n}^{-1}(\psi_j(z-s)-\psi_j(z)) = -\zl(z-ns)+n\zl(z-s)-(n-1)\zl(z).
    \end{equation}
    Thus, writing the $\tilde{f}(z)$ as in \eqref{eq:ftildel-short} in \eqref{eq:fbarl-def}, the computations \eqref{eq:psil-comp1} and \eqref{eq:psil-comp2} yield that \vspace{-.07in}\begin{multline*}
        \bar{f}(z)=\tilde{c}(0)\zl(z)+{\overbrace{\tilde{c}(1)\zl(z)-\tilde{c}(1)\zl(z)}^0}+\tilde{c}(1)\zl(z-s)\\ + \sum_{n\in\mathbb{Z}-\{0,1\}}\Bigl(n\tilde{c}(n)\bigl(\zl(z-s)-\zl(z)\bigr)+\tilde{c}(n)\zl(z)\Bigr)\\
        =\pano^\mathcal{R}(f,1)\cdot\bigl(\zl(z-s)-\zl(z)\bigr)+\left(\sum_{\omega\in\orbitl}\ores(f,\omega,1)\right)\cdot\zl(z)=0.
    \end{multline*}
    This concludes the proof that $f(z)$ is summable.\end{proof}

\begin{remark}\label{rem:reduced-form-l}
    Whether or not any orbital or panorbital residues of $f(z)\in\bkl$ vanish, defining $\tilde{f}(z)$ in terms of $f(z)$ as in \eqref{eq:ftildel-def}, and subsequently $\bar{f}(z)$ in terms of $\tilde{f}(z)$ as in \eqref{eq:fbarl-def}, results in the \emph{reduced form} for $f(z)$ (cf.~\cite[Lem.~B.14]{Dreyfus2018}):
    \begin{multline*}\bar{f}(z)=\pano^\mathcal{R}(f,0)+\pano^\mathcal{R}(f,1)\cdot\bigl(\zl(z-s)-\zl(z)\bigr)\\+\sum_{\omega\in\orbitl}\sum_{k\geq 1}\ores(f,\omega,k)\cdot\varphi^\mathcal{R}_{\omega,0,k},\end{multline*} relative to a choice of $\tau$-pinning $\mathcal{R}$ as in Definition~\ref{def:pinningl} and with the $\varphi^\mathcal{R}_{\omega,0,k}\in\bkl$ defined as in \eqref{eq:phil}, and such that $\bar{f}(z)-f(z)$ is summable.
\end{remark}

\begin{remark} \label{rem:pano-pinning-l}
    For any other choice $\mathcal{R}'=\{\alpha_\omega' \ | \ \omega\in\orbitl\}$ of $\tau$-pinning, there exist unique $\lambda_\omega\in\Lambda$ and $n_\omega\in\mathbb{Z}$ such that $\alpha_\omega'=\alpha_\omega+\lambda_\omega+n_\omega s$ for each $\omega\in\orbitl$. We see immediately from the definition \eqref{eq:pano-l-def} that \begin{equation}\label{eq:pinning-comparison-l1} \pano^{\mathcal{R}'}(f,1)=\pano^\mathcal{R}(f,1)-\sum_{\omega\in\orbitl}n_\omega\ores(f,\omega,1).\end{equation}
    Similarly, comparing the rightmost sums in the $\zeta$-expansions \eqref{eq:zetaExp-l} relative to $\mathcal{R}$ and to $\mathcal{R}'$ we find that \begin{equation}\label{eq:pinning-comparison-l0} \pano^{\mathcal{R}'}(f,0)=\pano^\mathcal{R}(f,0)- \sum_{\omega\in\orbitl} \eta_\Lambda(\lambda_\omega)\ores(f,\omega,1).\end{equation}
    Thus in particular the panorbital residues of $f(z)\in\bkl$ are independent of the choice of $\tau$-pinning of $\ecl$ whenever all the (first-order) orbital residues of $f(z)$ vanish.
\end{remark}

\subsection{Basic examples}\label{subsec:basic-examples-l}

To conclude this section, we illustrate the above general results with some basic examples. We choose once and for all an arbitrary $\tau$-pinning $\mathcal{R}=\{\alpha_\omega\}_{\omega\in\orbitl}$ of $\ecl$ as in Definition~\ref{def:pinningl}.

\begin{example}\label{ex:wp-l}
    Consider first $\wpl(z),\wpl'(z)\in\bkl$ as in \eqref{eq:zeta-wp-defs}. It is easy to see directly that neither $\wpl(z)$ nor $\wpl'(z)$ is summable. We see that both $\pano^\mathcal{R}(\wpl,1)=0$ and $\pano^\mathcal{R}(\wpl',1)=0$, since $c_1(\wpl,\alpha)=0$ and $c_1(\wpl',\alpha)=0$ for every $\alpha\in\C$. The expansions \eqref{eq:zetaExp-l} in this case are given by $\wpl(z)=0+\zl'(z-\alpha_{\Z \check{s}}-ns)$ and $\wpl'(z)=0+\zl''(z-\alpha_{\Z \check{s}}-ns)$, relative to whichever $\alpha_{\Z \check{s}}\in\Lambda+\Z s$ is chosen in $\mathcal{R}$ so that \mbox{$0=\alpha_{\Z \check{s}}+\lambda+ns$}. Thus $\pano^\mathcal{R}(\wpl,0)=0$ and $\pano^\mathcal{R}(\wpl',0)=0$. Finally, we see that every other orbital residue of these functions vanishes except for $\ores(\wpl,\Z \check{s},2)=1$ and $\ores(\wpl',\Z\check{s},3)=-2$, confirming the conclusion of Theorem~\ref{thm:mainl}. A similar computation shows more generally that for $k\geq 2$ each panorbital residue of $\zl^{(k-1)}(z)$ also vanishes, as does each of its orbital residues except for $\ores\bigl(\zl^{(k-1)},\Z\check{s},k)=(-1)^{k-1}(k-1)!$.\end{example}

 \begin{example}\label{ex:zeta-difference-l}Consider $f_t(z):=\zl(z-t)-\zl(z)\in\bkl$ for $t\in\C$ (see Lemma~\ref{lem:l-zeta-difference}). It is not difficult to see directly that $f_t(z)$ cannot be summable unless it is $0$. Let us further simplify notation by writing $\alpha_t,\alpha_0\in\C$ for the representatives chosen in $\mathcal{R}$ for the orbits $\omega_t:=\check{t}\oplus\Z\check{s}$ and $\omega_0:=\check{0}\oplus\Z \check{s}$ in $\ecl$. Then there exist unique $\lambda_t,\lambda_0\in\Lambda$ and $n_t,n_0\in\Z$ such that $t=\alpha_t+\lambda_t+n_ts$ and $0=\alpha_0+\lambda_0+n_0s$. We compute directly that $\pano^\mathcal{R}(f_t,1)=n_t-n_0$. The expansion \eqref{eq:zetaExp-l} in this case is \[f_t(z)=\pano^\mathcal{R}(f_t,0)+\zl(z-\alpha_t-n_ts)-\zl(z-\alpha_0-n_0s),\] and it follows that $\pano^\mathcal{R}(f_t,0)=-\eta_\Lambda(\lambda_t)+\eta_\Lambda(\lambda_0)=\eta_\Lambda(\lambda_0-\lambda_t)$. Moreover $\ores(f_t,\omega_0,1)=0$ if and only if $t\in\Lambda+\Z s$ (i.e., if $\check{t}\in\Z\check{s}$); otherwise, $\ores(f_t,\omega_t,1)=1=-\ores(f_t,\omega_0,1)$. But if $t\in\Lambda+\Z s$ and $n_t\neq n_0$ then $\pano^\mathcal{R}(f_t,1)\neq 0$. Thus the only way to have all orbital residues and the first-order panorbital residue of $f_t(z)$ vanish is to have $\check{t}=\check{0}\in\ecl$. But in this case $f_t(z)=\eta_\Lambda(\lambda_0-\lambda_t)\in\C$ is summable if and only if it is $0$. This again confirms the conclusion of Theorem~\ref{thm:mainl}.\end{example}

\begin{example} \label{ex:psi-sum-l}For $n\in\Z$, consider the elements
    \[\Psi_n(z):=-\zl(z-ns)+n\zl(z-s)-(n-1)\zl(z)\in\bkl.\] Then $\ores(\Psi_n,\Z\check{s},1)=-1+n-(n-1)=0$, and every other orbital residue of $\Psi_n(z)$ also vanishes, since $\Psi_n(z)$ is holomorphic away from $\Lambda\cup (s+\Lambda)\cup (ns+\Lambda)$ and only has at worst first-order poles at these places. For $\alpha_0=-\lambda_0-n_0s\in\mathcal{R}$ as in Example~\ref{ex:zeta-difference-l}, so that $js=\alpha_0+\lambda_0+(j+n_0)s$, we have \[
    \pano^\mathcal{R}(\Psi_n,1)=(n+n_0)\cdot(-1)+(1+n_0)\cdot n+n_0\cdot(1-n)=0.\] The expansion \eqref{eq:zetaExp-l} in this case is \[\Psi_n(z)=c^\mathcal{R}_0(\Psi_n)+(-1)\zl(z-ns+\lambda_0)+n\zl(z-s+\lambda_0)+(1-n)\zl(z+\lambda_0),\] whence \[\pano^\mathcal{R}(\Psi_n,0)=c_0^\mathcal{R}(\Psi_n)=\eta_\Lambda(\lambda_0)-n\eta_\Lambda(\lambda_0)+(n-1)\eta_\Lambda(\lambda_0)=0.\] Therefore by Theorem~\ref{thm:mainl} each $\Psi_n(z)$ is summable. This agrees with the computations \eqref{eq:psil-comp1} and \eqref{eq:psil-comp2} in case $n\neq0,1$, and $\Psi_0(z)=\Psi_1(z)=0$.
\end{example}

\section{Tate setting}
\label{sec:tate}
\subsection{Reminders and notation}
\label{subsec:tateNotation}
In this section we denote by $\Cq$ an algebraically closed field complete with respect to a non-Archimedean absolute value $| \ \ |$. Everything we do here is also valid in the Archimedean case $\Cq=\C$. Given $q\in\Cq$ such that $0<|q|<1$, we denote by $\bkq$ the field of (multiplicatively) $q$-periodic meromorphic functions $f(z)$ on $\Cq^\times$, i.e., such that $f(qz)=f(z)$, which is identified with the field of meromorphic functions on the elliptic curve $\ecq:=\Cq/q^\mathbb{Z}$. We again denote the projection map $\Cq^\times\rightarrow\ecq:t\mapsto \check{t}:=t\cdot q^\Z$. We consider as in \cite[\S2]{Roquette} and \cite[\S5.1]{RigidAnalyticGeometry} (cf.~\cite[Prop.~V.3.2]{silvermanAdvanced}) the \emph{basic theta function} \begin{equation} \label{eq:q-theta-def}
    \theta_q(z)=\prod_{m\geq 0}(1-q^mz^{-1})\cdot\prod_{m\geq 1}(1-q^mz),
\end{equation} which is holomorphic on $\Cq^\times$ and satisfies $\theta_q(qz)=-(qz)^{-1}\theta_q(z)$. Denoting by $\delta=z\frac{d}{dz}$ the Euler derivation, we consider similarly as in \cite[\S3]{Roquette} the ``zeta function'' \begin{equation}\label{eq:q-zeta-def}
    \zeta_q(z):=-\frac{\delta\theta_q(z)}{\theta_q(z)}=\frac{1}{1-z}+\sum_{m\geq 1} \left(\frac{q^mz}{1-q^mz}-\frac{q^mz^{-1}}{1-q^mz^{-1}}\right)
\end{equation} which is meromorphic on $\Cq^\times$ and satisfies $\zeta_q(qz)=\zeta_q(z)+1$. This immediately implies the following analogue of Lemma~\ref{lem:l-zeta-difference}.

\begin{lemma}\label{lem:q-zeta-difference}
    For each $t\in\Cq^\times$, the difference $\zq(t^{-1}z)-\zq(z)\in\bkq$ is \mbox{$q$-periodic}. Moreover, if $t\notin q^\Z$, then $\zq(t^{-1}z)-\zq(z)$ has simple poles at $\check{t}$ and at $\check{1}$ on $\ecq$, and is holomorphic elsewhere in $\Cq^\times$.
\end{lemma}

It is known (see e.g. \cite[\S3]{Roquette}, \cite[\S15.1]{LangElliptic}, or \cite[Thm.~5.1.10]{RigidAnalyticGeometry}) that $\bkq=\Cq(\wp_q,\wp_q')$ is a separable quadratic extension of its subfield $\Cq(\wp_q)$, where
\begin{equation}\label{eq:q-wp-def}\begin{aligned}
    \wp_q(z)&=\delta\zeta_q(z)=\sum_{m\in\Z}\frac{q^mz}{(1-q^mz)^2};\qquad\text{and} \\
    \wp_q'(z)&=\frac{1}{2}\bigl(\delta\wp_q(z)-\wp_q(z)\bigr)=\sum_{m\in\Z}\frac{q^{2m}z^2}{(1-q^mz)^3}.
\end{aligned}
\end{equation} The normalizations $\tilde{\wp}_q(z)=\wp_q(z)+-2s_{1,q}$ and $\tilde{\wp}_q'(z)=\wp_q'(z)+s_{1,q}$, where the constant $s_{1,q}=\sum_{m\geq 1}\frac{q^m}{(1-q^m)}\in\Cq$, yield the following analogue of the Weierstrass equation \eqref{eq:weierstrass-de}, which is valid in arbitrary characteristic:
\begin{equation} \label{eq:q-weierstrass-de}
\tilde{\wp}_q'(z)^2+\tilde{\wp}_q(z)\tilde{\wp}_q'(z)=\tilde{\wp}_q(z)^3+b_q\wp_q(z)+c_q.
\end{equation} The constants $b_q,c_q\in\Cq$ depend only on $q$ and satisfy the non-singularity condition $b_q^2-c_q-64b_q^3+72b_qc_q-432c_q^2\neq 0$.

Given any meromorphic function $f(z)$ on $\Cq^\times$, we write $c_k(f,\alpha)\in\Cq$ for the coefficients occurring in the principal part of the Laurent series expansion of $f(z)$ at $\alpha\in\Cq^\times$:
\begin{equation}\label{eq:laurent-coeffs-q}f(z)=\sum_{k\geq 1}\frac{c_k(f,\alpha)}{\bigl(1-\alpha^{-1}z\bigr)^k}+h(z),\qquad\text{with} \quad h(z)\in \Cq\bigl[\bigl[1-\alpha^{-1}z\bigr]\bigr].\end{equation}
We again define the $c_k(f,\alpha)$ in this way for every $k\in\mathbb{N}$ and every $\alpha\in\Cq^\times$, even though all but finitely many of them are zero for each given $\alpha\in\Cq^\times$. We again note that if $f(z)\in\bkq$ then the $c_k(f,\alpha)$ depend only on the point $\check{\alpha}\in\ecq$, which we occasionally emphasize by writing $c_k(f,\check{\alpha})$ in this case.

As in the introduction, we denote by $\tau$ the automorphism of $\bkq$ given by $\tau(f)(z)=f(sz)$ for a fixed $s\in\Cq^\times$ such that $s$ and $q$ are multiplicatively independent, or equivalently, such that $\check{s}$ is a non-torsion point of $\ecq$. The set of $\tau$-\emph{orbits} in $\ecq$ is \[\orbitq:=\ecq/\Z\check{s}.\]
The following definition is a rigid analytic reinterpretation of \cite[Def.~B.7]{Dreyfus2018}.

\begin{definition}[Orbital residues - Tate version]\label{def:oresq} The \emph{orbital residue} of $f(z)\in\bkq$ at the orbit $\omega\in\orbitq$ of order $k\in\mathbb{N}$ is \[\ores(f,\omega,k):=\sum_{\check{\alpha}\in\omega}c_k(f,\check{\alpha}).\]
\end{definition}

We again emphasize that the finite sum above is being taken over the points $\check{\alpha}\in\omega\subset\ecq$, rather than meaninglessly over the set of pre-images of these \mbox{points in~$\Cq^\times$.}

\subsection{Pinnings and their resulting zeta-expansions}

The following definition is the rigid analytic version of Definition~\ref{def:pinningl}.

\begin{definition}[$\tau$-pinnings - Tate version]\label{def:pinningq}
    A $\tau$-\emph{pinning} of $\ecq$ is a choice $\mathcal{R}=\{\alpha_\omega \in\Cq^\times\ | \ \omega\in\orbitq\}$ of representatives with $\check{\alpha}_\omega\in\omega$ for each $\omega\in\orbitq$.
\end{definition}

The following special functions $\zqd{k}(z) \in\bkq$ will serve as a sort of replacement for the higher derivatives $\frac{(-1)^k}{k!}\frac{d^k}{dz^k}\zl(z)$ that were utilized in the Weierstrass setting. We recall that $\delta=z\frac{d}{dz}$ denotes the Euler derivation.

\begin{definition}\label{def:q-zeta-fake-derivatives}
    We set $\zqd{0}(z):=\zq(z)$ as in \eqref{eq:q-zeta-def}, and, for $k\geq 1$, \[\zqd{k}(z):=\sum_{m\in\Z}\frac{q^{km}z^k}{\bigl(1-q^mz\bigr)^{k+1}}=\frac{1}{k}\Bigl(\delta\zqd{k-1}(z)-(k-1)\zqd{k-1}(z)\Bigr).\]
\end{definition}

Of course, $\zqd{1}(z)=\wp_q(z)$ and $\zqd{2}(z)=\wp_q'(z)$ as in \eqref{eq:q-wp-def}, and evidently every $\zqd{k}(z)\in\bkq$ for $k\geq 1$. Presumably the following rigid analytic version of Proposition~\ref{prop:zetaExp-l} will not be surprising to the experts, but we are not aware of a similar result being mentioned explicitly in the literature.

\begin{proposition}\label{prop:zetaExp-q}
    Let $\mathcal{R}=\{\alpha_\omega \ | \ \omega\in\orbitq\}$ be any $\tau$-pinning of $\ecq$, and let $f(z)\in\bkq$. There exists a unique constant $c_0^\mathcal{R}(f)\in\Cq$ such that
    \begin{equation} \label{eq:zetaExp-q}
        f(z) = c_0^{\mathcal{R}}(f)+\sum_{\omega\in\orbitq}\sum_{n\in\mathbb{Z}}\sum_{k\geq 1}c_k(f,s^n\alpha_\omega)\sum_{i=1}^k\binom{k-1}{i-1}\zqd{i-1}\bigl(s^{-n}\alpha_\omega^{-1}z\bigr)
    \end{equation} as meromorphic functions on $\Cq^\times$.
\end{proposition}

\begin{proof}
     First we prove the preliminary claim that, for $\ell,k\in\mathbb{N}$, \begin{equation}\label{eq:q-zeta-coeffs} c_\ell\left(\sum_{i=1}^k\binom{k-1}{i-1}\zqd{i-1},1\right)=\begin{cases}1 & \text{if} \ \ell=k; \\ 0 & \text{otherwise}.
    \end{cases}\end{equation} The claim is obvious in case $\ell>k$, since each $\zqd{i-1}(z)$ has a pole of order $i$ at $1$ for $i\geq 1$. It follows from the elementary computation \[\frac{z^{i-1}}{(1-z)^i}=\sum_{\ell=1}^{i}(-1)^{i-\ell}\binom{i-1}{\ell-1}\frac{1}{(1-z)^\ell}\] that $c_\ell\bigl(\zqd{i-1},1\bigr)=(-1)^{i-\ell}\binom{i-1}{\ell-1}$, where as usual $\binom{i-1}{\ell-1}=0$ if $\ell>i$. Thus \begin{equation}\label{eq:combo-identity}c_\ell\left(\sum_{i=1}^k\binom{k-1}{i-1}\zqd{i-1},1\right)=\sum_{i=\ell}^k(-1)^{i-\ell}\binom{k-1}{i-1}\binom{i-1}{\ell-1},\end{equation} concluding the proof of the preliminary claim  (cf.~\cite[Lem.~4.27]{Loehr}). 
    
    By construction, the set $\{s^n\alpha_\omega \ | \ \omega\in\orbitq, \ n\in\Z\}$ contains precisely one coset representative in $\Cq^\times$ of each point in $\ecq$. Each summand in the right-hand side of \eqref{eq:zetaExp-q} with $i\geq 2$ is $q$-periodic by \eqref{eq:q-wp-def} and Definition~\ref{def:q-zeta-fake-derivatives}. Comparing local expansions, one could immediately reduce to the case where $f(z)$ has only simple poles. The remaining portion of the sum \mbox{$\sum_{\omega,n,k}c_k\bigl(f,s^n\alpha_\omega\bigr)\zq\bigl(s^{-n}\alpha_\omega^{-1}z\bigr)$} is seen to be $q$-periodic by the transformation law $\zq(qz)=\zq(z)+1$, together with the fact that $\sum_{\omega,n,k}c_k(f,s^n\alpha_\omega)=0$ by the Residue Theorem \cite[Thm.~2.3.3]{RigidAnalyticGeometry} applied to the differential form $f(z)\frac{dz}{z}$ over any fundamental domain $\bigl\{\alpha\in\Cq^\times \ \big| \ r|q|\leq |\alpha| \leq r\bigr\}$ with $r$ in the valuation group $|\Cq^\times|$ such that $f(z)$ has no poles of absolute value $r$. It follows that the difference $c_0^\mathcal{R}(f)$ between $f(z)$ and the rightmost sum in \eqref{eq:zetaExp-q} is $q$-periodic and holomorphic on $\Cq^\times$, and therefore it must be constant.
\end{proof}

\subsection{Panorbital residues and summability}

\begin{definition}[Panorbital residues - Tate version] \label{def:pano-q} For a given $\tau$-pinning $\mathcal{R}=\{\alpha_\omega \ | \ \omega\in\orbitq\}$ of $\ecq$, the \emph{panorbital residues} of $f(z)\in\bkq$ of orders $0$ and $1$ relative to $\mathcal{R}$ are respectively
\begin{equation}\label{eq:pano-q-def}
    \begin{gathered}
        \pano^\mathcal{R}(f,0):=c_0^\mathcal{R}(f);\vphantom{\sum_A}\qquad \text{and}\\
    \pano^\mathcal{R}(f,1):=\sum_{\omega\in\orbitq}\sum_{n\in\mathbb{Z}} \sum_{k\geq 1} n\cdot c_k\bigl(f,s^n\alpha_\omega\bigr).
    \end{gathered}
\end{equation}
\end{definition}

\begin{remark}\label{rem:pano-wq-comparison} The analogous ``zeta-expansions,'' given in \eqref{eq:zetaExp-l} in the Weierstrass setting and in \eqref{eq:zetaExp-q} in the present Tate setting, provide a direct analogy between the order $0$ panorbital residues given in Definition~\ref{def:pano-l} and Definition~\ref{def:pano-q}. The difference between the corresponding order $1$ panorbital residues, however, may perhaps seem more stark due to the appearance of the higher-order $k\geq 2$ terms in the Tate setting, which are absent in the Weierstrass setting. However, we see that in the Weierstrass setting we could have defined equivalently \[\pano^\mathcal{R}(f,1)=\sum_{\omega\in\orbitl}\sum_{n\in\Z}n\cdot\mathrm{res}(f(z)dz,\alpha_\omega+ns),\] and that similarly in the Tate setting we have \[\pano^\mathcal{R}(f,1)=\sum_{\omega\in\orbitq}\sum_{n\in\Z}n\cdot\mathrm{res}\biggl(f(z)\frac{dz}{z},s^n\alpha_\omega\biggr),\] where in both cases $\mathrm{res}(f(z)\varpi,\alpha)$ refers to the residues of the corresponding meromorphic differential forms $f(z)\varpi$ at the corresponding points $\alpha$. We note that $\varpi=dz$ in the Weierstrass setting and $\varpi=dz/z$ in the Tate setting both descend to a nowhere-vanishing global invariant differential on the corresponding elliptic curve (cf.~Proposition~\ref{prop:diff-pano}).
\end{remark}

\begin{theorem}\label{thm:mainq} Let $\mathcal{R}=\{\alpha_\omega \ | \ \omega\in\orbitq\}$ be any $\tau$-pinning of $\ecq$, and let $f(z)\in\bkq$. Then $f$ is elliptically summable if and only if $\pano^\mathcal{R}(f,0)=0$, $\pano^\mathcal{R}(f,1)=0$, and $\ores(f,\omega,k)=0$ for every $\omega\in\orbitq$ and $k\geq 1$.
\end{theorem}

\begin{proof} 
   First suppose $f(z)=g(sz)-g(z)$, for some $g(z)\in\bkq$. From the fact that \mbox{$c_k(\tau(g),\alpha)=c_k(g,s\alpha)$} for every $\alpha\in\Cq^\times$ and $k\geq 1$ follows the relation \mbox{$\ores(\tau(g),\omega,k)=\ores(g,\omega,k)$} for every $\omega\in\orbitq$ and $k\geq 1$, and therefore every $\ores(f,\omega,k)=0$. Moreover, applying $\tau$ to both sides of the ``zeta expansion'' \eqref{eq:zetaExp-q} for $g(z)$, we find that similarly $c^\mathcal{R}_0(\tau(g))=c^\mathcal{R}_0(g)$, and therefore $\pano^\mathcal{R}(f,0)=0$. Finally, we see that \begin{multline*}\pano^\mathcal{R}(\tau(g),1) =\sum_{\omega\in\orbitq}\sum_{n\in\mathbb{Z}}\sum_{k\geq 1}n\cdot c_k\bigl(g,s^{n+1}\alpha_\omega\bigr) \\
    =\pano^\mathcal{R}(g,1)-\sum_{\omega\in\orbitq}\sum_{n\in\mathbb{Z}}\sum_{k\geq 1}c_k(g,s^n\alpha_\omega)=\pano^\mathcal{R}(g,1),\end{multline*} by the Residue Theorem \cite[Thm.~2.3.3]{RigidAnalyticGeometry} applied to the differential form $g(z)\frac{dz}{z}$ over an appropriate fundamental domain in $\Cq^\times$. Thus $\pano^\mathcal{R}(f,1)=0$.    

    To prove the converse, let us now suppose that every orbital residue and both panorbital residues of $f(z)$ are zero. To simplify notation, let us define \begin{equation}\label{eq:phiq}\varphi^\mathcal{R}_{\omega,n,k}(z):=\sum_{i=1}^k\binom{k-1}{i-1}\zqd{i-1}\bigl(s^{-n}\alpha_\omega^{-1}z\bigr) - \zq\bigl(s^{-n}z\bigr)\in\bkq\end{equation} for $\omega\in\orbitq$, $n\in\mathbb{Z}$, and $k\geq 1$. Let us consider 
    \begin{equation}\tilde{f}(z):=f(z)+\sum_{\omega\in\orbitq}\sum_{n\in\Z}\sum_{k\geq 1}c_k(f,s^{n}\alpha_\omega)\cdot\bigl(\varphi^\mathcal{R}_{\omega,n,k}(s^nz)-\varphi^\mathcal{R}_{\omega,n,k}(z)\bigr).\label{eq:ftildeq-def}
    \end{equation}
    Since $\varphi\bigl( s^nz\bigr)-\varphi(z)$ is summable for any $\varphi(z)\in\bkq$ and $n\in\mathbb{Z}$, $f(z)$ is summable if and only if $\tilde{f}(z)\in\bkq$ is summable. Moreover, since each $\varphi_{\omega,n,k}^{\mathcal{R}}\bigl(s^nz\bigr) = \varphi_{\omega,0,k}^{\mathcal{R}}(z)$, writing $f(z)$ as in \eqref{eq:zetaExp-q} yields 
    \begin{align}
    \tilde{f}(z)&=0+\sum_{\omega\in\orbitq}\sum_{k\geq 1}\bigg(\ores(f,\omega,k)\cdot\varphi_{\omega,0,k}^\mathcal{R}(z)  +\sum_{n\in\mathbb{Z}}c_k\bigl(f,s^n\alpha_\omega\bigr)\cdot\zq\bigl(s^{-n}z\bigr)\bigg) \notag\\
    &=0+\sum_{\omega\in\orbitq}\sum_{n\in\mathbb{Z}}\sum_{k\geq 1}c_k\bigl(f,s^n\alpha_\omega\bigr)\cdot\zq\bigl(s^{-n}z\big),\label{eq:ftildeq-short}\end{align} where the leftmost $0$'s in \eqref{eq:ftildeq-short} arise from our assumptions that $c_0^\mathcal{R}(f)$ and the $\ores(f,\omega,k)=0$ all vanish.
    
    Next let us again simplify notation by defining \begin{gather*}\tilde{c}(n):=c_1\bigl(\tilde{f},s^n\bigr)=\sum_{\omega\in\orbitq}\sum_{k\geq 1}c_k\bigl(f,s^n\alpha_\omega\bigr)\in\Cq,
      \end{gather*} so that $\tilde{f}(z)=\sum_{n\in\mathbb{Z}}\tilde{c}(n)\zq\bigl(s^{-n}z\bigr)$ in \eqref{eq:ftildeq-short} and $\sum_{n\in\mathbb{Z}}n\tilde{c}(n)=\pano^\mathcal{R}(f,1)$. Let us also define the auxiliary $q$-periodic functions
    \begin{equation}\label{eq:psiq-def}
        \psi_j(z):=
        \begin{cases}\zq\bigl(s^{-j}z\bigr)-\zq\bigl(s^{-1}z\bigr)\in\bkq & \text{for} \ j\geq 2;\\
            \zq\bigl(s^{-j}z\bigr)-\zq\bigl(z\bigr)\in\bkq & \text{for} \ j\leq -1.\end{cases}
    \end{equation} We now consider the further reduction of $f(z)$ by summable elements of $\bkq$: 
      \begin{equation}
        \bar{f}(z):=\tilde{f}(z)+\sum_{n\leq -1}\sum_{j=n}^{-1}\tilde{c}(n)\cdot\bigl(\psi_j(s^{-1}z)-\psi_j(z)\bigr)+\sum_{n\geq 2}\sum_{j=2}^n\tilde{c}(n)\cdot\bigl(\psi_j(sz)-\psi_j(z)\bigr).\label{eq:fbarq-def}
    \end{equation}
    Once again we see that $f(z)$ is summable if and only if $
    \bar{f}(z)$ is summable, and like before we find that \begin{equation}\label{eq:psiq-comp1}\sum_{j=2}^n\bigl(\psi_j(sz)-\psi_j(z)\bigr)=-\zq\bigl(s^{-n}z\bigr)+n\zq\bigl(s^{-1}z\bigr)-(n-1)\zq\bigl(z\bigr)\end{equation} for each $n\geq 2$, and again similarly that \begin{equation}\label{eq:psiq-comp2}\sum_{j=n}^{-1}\bigl(\psi_j(s^{-1}z)-\psi_j(z)\bigr)=-\zq\bigl(s^{-n}z\bigr)+n\zq\bigl(s^{-1}z\bigr)-(n-1)\zq\bigl(z\bigr)\end{equation} for each $n\leq -1$. It follows that $\bar{f}(z)$ is holomorphic on $\Cq^\times$ away from $q^\Z$ and $sq^\Z$.
    Thus, writing the $\tilde{f}(z)$ as in \eqref{eq:ftildeq-short} in \eqref{eq:fbarq-def}, the computations \eqref{eq:psiq-comp1} and \eqref{eq:psiq-comp2} yield that \vspace{-.07in} \begin{multline*}
        \bar{f}(z)=\tilde{c}(0)\zq(z)+{\overbrace{\tilde{c}(1)\zq(z)-\tilde{c}(1)\zq(z)}^0}+\tilde{c}(1)\zq\bigl(s^{-1}z\bigr)\\ + \sum_{n\in\mathbb{Z}-\{0,1\}}\biggl(n\tilde{c}(n)\Bigl(\zq\bigl(s^{-1}z\bigr)-\zq(z)\Bigr)+\tilde{c}(n)\zq(z)\biggr)\\
        =\pano^\mathcal{R}(f,1)\cdot\Bigl(\zq\bigl(s^{-1}z\bigr)-\zq(z)\Bigr)+\left(\sum_{\omega\in\orbitq}\sum_{k\geq 1}\ores(f,\omega,k)\right)\cdot\zq(z)=0.
    \end{multline*}
    This concludes the proof that $f(z)$ is summable.
\end{proof}

\begin{remark}\label{rem:reduced-form-q}
    Whether or not any orbital or panorbital residues of $f(z)\in\bkq$ vanish, defining $\tilde{f}(z)$ in terms of $f(z)$ as in \eqref{eq:ftildeq-def}, and subsequently $\bar{f}(z)$ in terms of $\tilde{f}(z)$ as in \eqref{eq:fbarq-def}, results in the \emph{reduced form} for $f(z)$ (cf.~\cite[Lem.~B.14]{Dreyfus2018}):
    \begin{multline*}\bar{f}(z)=\pano^\mathcal{R}(f,0)+\pano^\mathcal{R}(f,1)\cdot\bigl(\zq(z-s)-\zq(z)\bigr)\\+\sum_{\omega\in\orbitq}\sum_{k\geq 1}\ores(f,\omega,k)\cdot\varphi^\mathcal{R}_{\omega,0,k},\end{multline*} relative to a choice of $\tau$-pinning $\mathcal{R}$ as in Definition~\ref{def:pinningq} and with the $\varphi^\mathcal{R}_{\omega,0,k}\in\bkq$ defined as in \eqref{eq:phiq} (cf.~Remark~\ref{rem:reduced-form-l}), and such that $\bar{f}(z)-f(z)$ is summable.
\end{remark}

\begin{remark} \label{rem:pano-pinning-q} The definitions of $\pano^\mathcal{R}(f,1)$ and $\pano^\mathcal{R}(f,0)$ are direct analogues of those in Definition~\ref{def:pano-l} in the Weierstrass case. For any other choice $\mathcal{R}'=\{\alpha_\omega' \ | \ \omega\in\orbitq\}$ of $\tau$-pinning, there exist unique $m_\omega,n_\omega\in\mathbb{Z}$ such that $\alpha_\omega'=q^{m_\omega}s^{n_\omega}\alpha_\omega$ for each $\omega\in\orbitq$. We see immediately from the definition \eqref{eq:pano-q-def} that \begin{equation}\label{eq:pinning-comparison-q1} \pano^{\mathcal{R}'}(f,1)=\pano^\mathcal{R}(f,1)-\sum_{\omega\in\orbitq}\sum_{k\geq 1}n_\omega\ores(f,\omega,k).\end{equation}
   It similarly follows from the definition \eqref{eq:pano-q-def} that \begin{equation}\label{eq:pinning-comparison-q0} \pano^{\mathcal{R}'}(f,0)=\pano^\mathcal{R}(f,0)- \sum_{\omega\in\orbitq} \sum_{k\geq 1}m_\omega\ores(f,\omega,k).\end{equation}
    Thus in particular the panorbital residues of $f(z)\in\bkq$ are independent of the choice of $\tau$-pinning of $\ecq$ whenever all the orbital residues of $f(z)$ vanish.
\end{remark}

\subsection{Basic examples} \label{subsec:basic-examples-q}

To conclude this section, we illustrate the above general results with some basic examples. We choose once and for all an arbitrary $\tau$-pinning $\mathcal{R}=\{\alpha_\omega\}_{\omega\in\orbitq}$ of $\ecq$ as in Definition~\ref{def:pinningq}.

\begin{example}\label{ex:zqd-q} Consider first the $\zqd{k-1}(z)\in\bkq$ for $k\geq 2$ as in Definition~\ref{def:q-zeta-fake-derivatives}. It is easy to see directly that none of these elements is summable. We saw in the proof of Proposition~\ref{prop:zetaExp-q} that \[c_\ell\bigl(\zqd{k-1},1\bigr)=(-1)^{k-\ell}\binom{k-1}{\ell-1}=\ores\bigl(\zqd{k-1},\Z\check{s},\ell\bigr)\neq 0\] for $1\leq \ell\leq k$, already confirming the conclusion of Theorem~\ref{thm:mainq}. In particular, since $\wp_q(z)=\zqd{1}(z)$ and $\wp_q'(z)=\zqd{2}(z)$ as in \eqref{eq:q-wp-def}, we obtain \begin{gather*}\ores(\wp_q,\Z\check{s},1)=-1;\qquad \ores(\wp_q,\Z\check{s},2)=1;\\
\ores(\wp_q',\Z\check{s},1)=1;\qquad \ores(\wp_q',\Z\check{s},2)=-2;\qquad \ores(\wp_q',\Z\check{s},3)=1;\end{gather*} (cf.~Example~\ref{ex:wp-l}). Moreover, \[\pano^\mathcal{R}\bigl(\zqd{k-1},1)=\sum_{\ell= 1}^kn(-1)^{k-\ell}\binom{k-1}{\ell-1}=0\] for every $k\geq 2$, regardless of the value of $n\in\Z$ such that $\check{\alpha}_{\Z \check{s}}=\ominus n\check{s}\in\ecq$. The expansions \eqref{eq:zetaExp-q} in this case collapse:
 \begin{multline*}
    \zqd{k-1}(z)=c^\mathcal{R}_0\bigl(\zqd{k-1}\bigr)+\sum_{\ell= 1}^k(-1)^{k-\ell}\binom{k-1}{\ell-1}\sum_{i=1}^\ell\binom{\ell-1}{i-1}\zqd{i-1}\bigl(s^{-n}\alpha_{\Z\check{s}}^{-1}z\bigr)\\=\pano^\mathcal{R}\bigl(\zqd{k-1},0\bigr)+\zqd{k-1}\bigl(s^{-n}\alpha_{\Z\check{s}}^{-1}z\bigr)
\end{multline*} (cf.~the proof of Proposition~\ref{prop:zetaExp-q}), and therefore $\pano^\mathcal{R}\bigl(\zqd{k-1},0)=0$, since $s^{n}\alpha_{\Z \check{s}}\in q^\Z$ and $\zqd{k-1}(z)$ is $q$-periodic for $k\geq 2$. In particular and analogously as in Example~\ref{ex:wp-l}, we find that the non-summability of $\wp_q(z)$ and $\wp_q'(z)$ is witnessed by their orbital residues, but not by their panorbital residues.\end{example}

 \begin{example} \label{ex:zeta-difference-q} Consider $f_t(z):=\zq(t^{-1}z)-\zq(z)\in\bkq$ for $t\in\Cq^\times$ (see Lemma~\ref{lem:q-zeta-difference}). It is not difficult to see directly that $f_t(z)$ cannot be summable unless it is $0$. Let us further simplify notation by writing $\alpha_t,\alpha_1\in\Cq$ for the representatives chosen in $\mathcal{R}$ for the orbits $\omega_t:=\check{t}\oplus\Z\check{s}$ and $\omega_1:=\check{1}\oplus\Z \check{s}$ in $\ecq$. Then there exist unique $m_t,m_1\in\Z$ and $n_t,n_1\in\Z$ such that $t=\alpha_tq^{m_t}s^{n_t}$ and $1=\alpha_1q^{m_1}s^{n_1}$. We compute directly that $\pano^\mathcal{R}(f_t,1)=n_t-n_1$. The expansion \eqref{eq:zetaExp-l} in this case is \[f_t(z)=\pano^\mathcal{R}(f_t,0)+\zq\bigl(q^{m_t}s^{-n_t}\alpha_t^{-1}z\bigr)-\zq\bigl(a^{m_1}s^{-n_1}\alpha_1^{-1}z\bigr),\] and it follows that $\pano^\mathcal{R}(f_t,0)=m_1-m_t$. Moreover $\ores(f_t,\omega_1,1)=0$ if and only if $t\in q^\Z\cdot s^\Z$ (i.e., if $\check{t}\in\Z\check{s}$); otherwise, $\ores(f_t,\omega_t,1)=1=-\ores(f_t,\omega_1,1)$. But if we have $t\in q^\Z\cdot s^\Z$ and $n_t\neq n_0$, then we have $\pano^\mathcal{R}(f_t,1)\neq 0$. Thus the only way to have all orbital residues and the first-order panorbital residue of $f_t(z)$ vanish is to have $\check{t}=\check{1}\in\ecq$. But in this case $f_t(z)=m_1-m_t\in\Ca$ is summable if and only if it is $0$. This again confirms the conclusion of Theorem~\ref{thm:mainq}. Analogously as in Example~\ref{ex:zeta-difference-l}, the non-summability of $\zq(t^{-1}z)-\zq(z)$, in the most interesting case where $\check{t}\in\Z\check{s}$, is witnessed by its panorbital residues, but not by its orbital residues.\end{example}

 \begin{example}\label{ex:phi-q}
    For $\omega\in\orbita$, $n\in\Z$, and $k\geq 2$, consider the elements \[\varphi^\mathcal{R}_{\omega,n,k}(z)=\sum_{i=2}^k\binom{k-1}{i-1}\zqd{i-1}(s^{-n}\alpha_\omega^{-1}z)+\bigl(\zq(s^{-n}\alpha_\omega^{-1}z)-\zq(s^{-n}z)\bigr)\in\bkq\] as in \eqref{eq:phiq} (the case $k=1$ corresponds to $f_t(s^{-n}z)$ as in Example~\ref{ex:zeta-difference-q} with $t=\alpha_\omega$). Since $\varphi^\mathcal{R}_{\omega,n,k}(z)=\varphi^\mathcal{R}_{\omega,0,k}(s^{-n}z)$ and the orbital and panorbital residues are $\tau$-invariant, let us treat only the case $n=0$, for simplicity. Since the orbital residues are $\Cq$-linear, it follows from Example~\ref{ex:zqd-q}, \eqref{eq:combo-identity}, and Example~\eqref{ex:zeta-difference-q} (with $t=\alpha_\omega$), that $c_1\bigl(\varphi^\mathcal{R}_{\omega,0,k},1\bigr)=\ores\bigl(\varphi^\mathcal{R}_{\omega,0,k},\Z\check{s},1\bigr)=-1$; \[c_\ell\bigl(\varphi^\mathcal{R}_{\omega,0,k},\alpha_\omega\bigr)=\ores\bigl(\varphi^\mathcal{R}_{\omega,n,k},\omega,\ell)=
    \begin{cases}
        1 & \text{if} \ \ell=k;\\
        -1 & \text{if} \ \ell=1 \ \text{and} \ \omega=\Z\check{s};\\
        0 & \text{otherwise};
    \end{cases}\] and every other orbital residue of $\varphi^\mathcal{R}_{\omega,0,k}(z)$ vanishes. Moreover, since the panorbital residues are also $\Ca$-linear, both panorbital residues of $\varphi^\mathcal{R}_{\omega,0,k}(z)$ agree with those of $f_t(z)$ as computed in Example~\ref{ex:zeta-difference-q} with $t=\alpha_\omega$. 
\end{example}

 \begin{example} \label{ex:psi-sum-q}For $n\in\Z$, consider the elements
    \[\Psi_n(z):=-\zq(s^{-n}z)+n\zq(s^{-1}z)-(n-1)\zq(z)\in\bkq.\] Then $\ores(\Psi_n,\Z\check{s},1)=-1+n-(n-1)=0$, and every other orbital residue of $\Psi_n(z)$ also vanishes, since $\Psi_n(z)$ is holomorphic away from $q^\Z\cup sq^\Z\cup s^nq^\Z$ and only has at worst first-order poles at these places. For $\alpha_1=q^{-m_1}s^{-n_1}\in\mathcal{R}$ as in Example~\ref{ex:zeta-difference-q}, so that $s^j=\alpha_1q^{m_1}s^{j+n_1}$, we have \[
    \pano^\mathcal{R}(\Psi_n,1)=(n+n_1)\cdot(-1)+(1+n_1)\cdot n+n_1\cdot(1-n)=0.\] The expansion \eqref{eq:zetaExp-l} in this case is \[\Psi_n(z)=c^\mathcal{R}_0(\Psi_n)+(-1)\zl(q^{m_1}s^{-n}z)+n\zl(q^{m_1}s^{-1}z)+(1-n)\zl(q^{m_1}z),\] whence \[\pano^\mathcal{R}(\Psi_n,0)=c_0^\mathcal{R}(\Psi_n)=m_1-nm_1+(n-1)m_1=0.\] Therefore by Theorem~\ref{thm:mainq} each $\Psi_n(z)$ is summable. This agrees with the computations \eqref{eq:psiq-comp1} and \eqref{eq:psiq-comp2} in case $n\neq0,1$, and $\Psi_0(z)=\Psi_1(z)=0$.
\end{example}

\section{Algebraic setting}\label{sec:algebraic}

\subsection{Reminders and notation}

In this section we denote by $\Ca$ an algebraically closed field, and by $\ec$ an elliptic curve over $\Ca$, i.e., a connected smooth projective algebraic curve of genus $1$ over $\Ca$ with a distinguished $\Ca$-point $\idec\in\ec(\Ca)$, which is the identity element for an abelian group law on $\ec(\Ca)$ that we denote by $\oplus$. The field of rational functions on $\ec$ is denoted by $\bk=\Ca(\ec)$. The group of divisors $\mathrm{Div}(\ec)$ is the free abelian group on the set of $\Ca$-points of $\ec$: \begin{equation}\label{eq:divisor-def}\mathrm{Div}(\ec)=\left\{\sum_{\alpha\in\ec(\Ca)}m_\alpha[\alpha] \ \middle| \ \begin{gathered} m_\alpha\in\Z \ \text{for every} \ \alpha\in\ec(\Ca) \\ \text{and almost every}\ m_\alpha=0\end{gathered}\right\}. \end{equation} For a divisor $D=\sum_\alpha m_\alpha[\alpha]\in\mathrm{Div}(\ec)$, its degree is $\mathrm{deg}(D)=\sum_\alpha m_\alpha$, its support is $\mathrm{supp}(D)=\{\alpha\in\ec(\Ca) \ | \ m_\alpha\neq 0\}$, 
and we say that $D$ is effective if $m_\alpha\geq 0$ for every $\alpha\in\ec(\Ca)$, in which case we write $D\geq 0$. For a rational function $f\in\bk$, its divisor $\mathrm{div}(f)\in\mathrm{Div}(\ec)$ is given by \[\mathrm{div}(f)=\sum_{\alpha\in\ec(\Ca)}\nu_\alpha(f)[\alpha],\] where $\nu_\alpha(f)$ is the valuation of $f$ at $\alpha\in\ec(\Ca)$. For a divisor $D\in\mathrm{Div}(\ec)$, its Riemann-Roch space is the $\Ca$-vector space \[\mathcal{L}(D)=\{f\in\bk \ | \ \mathrm{div}(f)-D\geq 0\}.\] More explicitly, if $D=\sum_\alpha m_\alpha[\alpha]$ then $\mathcal{L}(D)$ consists of those $f\in\bk$ such that $f$ has a zero at $\alpha$ of order at least $m_\alpha$ whenever $m_\alpha>0$ and a pole at $\alpha$ of order at most $-m_\alpha$ whenever $m_\alpha<0$. We shall make repeated use of the following consequence of the Riemann-Roch Theorem \cite[Cor.~II.5.5(c)]{SilvermanIntro}.
\begin{theorem}[Riemann-Roch] \label{thm:rr}
    If $D\in\mathrm{Div}(\ec)$ has $\mathrm{deg}(D)\geq 1$ then $\mathrm{dim}_\Ca\bigl(\mathcal{L}(D)\bigr)=\mathrm{deg}(D)$.
\end{theorem}
As in \cite[Prop.~III.3.1]{SilvermanIntro}, one deduces from Theorem~\ref{thm:rr} that there exist elements $\wp\in\mathcal{L}(2[\idec])-\Ca$ and $\wp'\in\mathcal{L}(3[\idec])-\mathcal{L}(2[\idec])$ satisfying the Weierstrass equation \begin{equation}\label{eq:alg-weierstrass-eq}(\wp')^2+a_1\wp\wp'+a_3\wp'=\wp^3+a_2\wp^2+a_4\wp+a_6\end{equation} for some $a_1,a_2,a_3,a_4,a_6\in\Ca$ and such that $\bk=\Ca(\wp,\wp')$ is a separable quadratic extension of its subfield $\Ca(\wp)$. It is well-known that when the characteristic of $\Ca$ is different from $2$ and $3$ then $\wp$ and $\wp'$ can be further chosen such that the corresponding $a_1$, $a_3$, and $a_2$ in \eqref{eq:alg-weierstrass-eq} are all zero.

Relative to a given choice of local uniformizers \begin{equation}\label{eq:uniformizers}\mathcal{U}=\bigl\{u_\alpha \in\bk \ \big| \ \nu_\alpha(u_\alpha)=1 \ \text{for every} \ \alpha\in\ec(\Ca)\bigr\},\end{equation} for any given $f\in\bk$ we write $c_k^\mathcal{U}(f,\alpha)\in\Ca$ for the coefficients occurring in the non-positive part of the formal Laurent series expansion of $f$ in $\Ca((u_\alpha))$:
\begin{equation}\label{eq:laurent-coeffs-a}f=\sum_{k\geq 0}\frac{c_k^\mathcal{U}(f,\alpha)}{u_\alpha^k}+u_\alpha h,\qquad\text{with} \quad h\in\Ca[[u_\alpha]].\end{equation} As before, we define the $c_k^\mathcal{U}(f,\alpha)$ in this way for every $k\in\mathbb{N}$ and every $\alpha\in\ec(\Ca)$, even though all but finitely many of them are zero for each given $\alpha\in\ec(\Ca)$. Note that if $\nu_\alpha(f)\geq 0$ then the constant term $c_0^\mathcal{U}(f,\alpha)=f(\alpha)$ is independent of the choice of $\mathcal{U}$, but otherwise the value of $c_0^\mathcal{U}(f,\alpha)\in\Ca$ does depend on the choice of local uniformizer $u_\alpha\in\bk$.

As in the introduction, we denote by $\tau$ the automorphism of $\bk$ induced by the addition-by-$s$ map $\ec(\Ca)\rightarrow \ec(\Ca):\alpha\mapsto \alpha\oplus s$ for a fixed non-torsion point $s\in\ec(\Ca)$. Equivalently, $\tau$ is the unique $\Ca$-linear field automorphism of $\bk$ such that $\nu_{\alpha}(\tau(f))=\nu_{\alpha\oplus s}(f)$ for every $f\in\bk$ and $\alpha\in\ec(\Ca)$. The explicit formulas describing how $\tau$ transforms elements of $\bk=\Ca(\wp,\wp')$ as in \eqref{eq:alg-weierstrass-eq} are well known \cite[\S III.2]{SilvermanIntro}. The set of $\tau$-\emph{orbits} in $\ec$ is \[\orbita:=\ec/\Z s.\]

The following definitions were first introduced in \cite[Def.~B.6]{Dreyfus2018} and \cite[Def.~B.7]{Dreyfus2018}, respectively (with slightly different terminology).

\begin{definition}[$\tau$-compatible uniformizers] \label{def:uniformizers}
    A set $\mathcal{U}=\{u_\alpha\ | \ \alpha\in\ec(\Ca)\}$ of local uniformizers as in \eqref{eq:uniformizers} is $\tau$-\emph{compatible} if $\tau(u_{\alpha\oplus s})=u_\alpha$ for every $\alpha\in\ec(\Ca)$.
\end{definition}

\begin{definition}[Orbital residues - algebraic version] \label{def:oresa} For a given $\tau$-compatible set of local uniformizers $\mathcal{U}$, the \emph{orbital residue} of $f\in\bk$ at the orbit $\omega\in\orbita$ of order $k\in\mathbb{N}$ relative to $\mathcal{U}$ is \[\ores^\mathcal{U}(f,\omega,k):=\sum_{\alpha\in\omega}c_k^\mathcal{U}(f,\alpha).\]  
\end{definition}

\begin{remark}\label{rem:ores-uniformizers} In principle, it is a relatively simple matter to compute the effect of the choice of $\tau$-compatible set of local uniformizers $\mathcal{U}$ on the orbital residues in Definition~\ref{def:oresa}. Indeed, for another given choice $\mathcal{U}'=\{u_\alpha' \ | \ \alpha\in\ec(\Ca)\}$ of $\tau$-compatible local uniformizers, one can compute explicitly (as many terms as desired of) the unique formal power series \[h_\omega(u)=\sum_{j\geq 0}h_{\omega,j}u^j\in\Ca[[u]]\] with $h_{\omega,0}\neq 0$ such that $u_\alpha=u_\alpha'/h_\omega(u_\alpha')\in\Ca[[u_\alpha']]$ for every $\alpha\in\omega$ simultaneously. Writing $h_{\omega,j}^{(m)}$ for the coefficients occurring in $\bigl(h_{\omega}(u)\bigr)^m=\sum_{j\geq 0}h_{\omega,j}^{(m)}u^j$, a straightforward computation yields that, for each $f\in\bk$, \begin{gather}\label{eq:u-coeff-effect}
    c_k^{\mathcal{U}'}(f,\alpha)=\sum_{m\geq k}h_{\omega,m-k}^{(m)}\cdot c_m^\mathcal{U}(f,\alpha)\qquad\text{for every} \quad\alpha\in\omega\quad \text{and} \quad k\geq 0,
    \intertext{and therefore for each $\omega\in\orbita$ and $k\geq 1$ we have}
    \label{eq:u-ores-effect} \ores^{\mathcal{U}'}(f,\omega,k)=\sum_{m\geq k}h_{\omega,m-k}^{(m)}\cdot\ores^\mathcal{U}(f,\omega,m).
\end{gather}
\end{remark}

\begin{remark}\label{rem:outer-uniformizers}
It is not strictly necessary to restrict ourselves to taking local uniformizers $u_\alpha\in \bk$.

Indeed, in the Weierstrass setting the orbital residues were given in Definition~\ref{def:oresl} relative to the complex analytic uniformizers \mbox{$u_{\check{\alpha}}=(z-\alpha)\notin\bkl$,} for $\alpha\in\C$ any lift of $\check{\alpha}\in\ecl$. A particular choice of lift of each point of $\ecl$ was subtly folded into the choice of Weierstrass pinning $\mathcal{R}$ in Definition~\ref{def:pinningl}, and these choices did not affect the definitions of any of the (pan)orbital residues save for that of the degree $0$ panorbital residue according to Remark~\ref{rem:pano-pinning-l}. The possibility of using such complex analytic uniformizers for elliptic curves over $\C$ relative to a choice of Weierstrass uniformization was already suggested in \cite[Remark~B.10]{Dreyfus2018} and \cite[Remark~A.8(1)]{Hardouin2021}.

Similarly, in the Tate setting the orbital residues were expressed in \mbox{Definition~\ref{def:oresq}} relative to our choice of (complex/rigid) analytic uniformizers $u_{\check{\alpha}}=(1-\alpha^{-1}z)\notin\bkq$, where again the particular choices of lifts $\alpha\in\Cq^\times$ for each $\check{\alpha}\in\ecq$ that were implicit in the choice of pinning $\mathcal{R}$ in \mbox{Definition~\ref{def:pinningq}} affected nothing but the order $0$ panorbital residues according to Remark~\ref{rem:pano-pinning-q}. One could have made other choices, say $u_{\check{\alpha}}=1-\alpha z^{-1}$ or, as suggested in \cite[Remark~A.8(2)]{Hardouin2021}, $u_{\check{\alpha}}=\log(\alpha^{-1}z)$ on some sufficiently small neighborhood of $\alpha\in\Cq^\times$. The formulas \eqref{eq:u-coeff-effect} and \eqref{eq:u-ores-effect} capture the effect of making such alternative choices of uniformizers for the local ring at $\check{\alpha}\in\ecq$.

Defining the local coefficients in \eqref{eq:laurent-coeffs-l} and \eqref{eq:laurent-coeffs-q} in terms of the Laurent series expansions with respect to these more general local uniformizers in the Weierstrass and Tate settings is what allowed us to carry out our constructions so ``concretely'' in terms of the zeta functions $\zl(z)$ and $\zq(z)$ (in spite of their being only quasi-periodic). Another advantage in those settings was that we were able to easily read off the residues of the differential forms $f(z)dz$ in the Weierstrass setting and $f(z)dz/z$ in the Tate setting from their local formal expansions with respect to these non-elliptic uniformizers, which enabled the application of the relevant Residue Theorems at pivotal points in our proofs of Proposition~\ref{prop:zetaExp-l} and Theorem~\ref{thm:mainl} in the Weierstrass setting, and Proposition~\ref{prop:zetaExp-q} and Theorem~\ref{thm:mainq} in the Tate setting. A similar need to appeal to the Cauchy Residue Theorem arises in both \cite[Appendix~B]{Dreyfus2018} and \cite[Appendix~A]{Hardouin2021}, in the latter of which the authors introduce the notion of $n$-coherent uniformizers (available only in characteristic zero, or when $n$ is sufficiently small relative to the characteristic), which have the property that the $c_1^\mathcal{U}(f,\alpha)=\mathrm{res}(f\varpi,\alpha)$ for every $f\in\bk$ such that the valuation $\nu_\alpha(f)\geq -n$, for $\varpi$ a conveniently chosen invariant differential form. We develop an alternative framework in which we are still able to appeal the the algebraic Residue Theorem~\cite[Thm.~III.7.14.2]{Hartshorne}.

There are other useful local uniformizers that one could utilize outside of $\bk$ while still remaining in a purely algebraic setting. For example, one could use local \'etale coordinates as in \cite[\href{https://stacks.math.columbia.edu/tag/054L}{Lemma 054L}]{stacks-project}, or more generally uniformizers in the \'etale local rings at closed points of $\ec$ \cite[\href{https://stacks.math.columbia.edu/tag/0BSK}{Section 0BSK}]{stacks-project}, or even more generally uniformizers in the stalks at closed points relative to other locally-ringed Grothendieck topoi, with the obvious notion of $\tau$-compatibility (under base change with respect to addition by $s$) in Definition~\ref{def:uniformizers}, should such a need or whim ever arise. Although everything that follows remains true relative to this more general notion of $\tau$-compatible set of local uniformizers, we shall restrict ourselves to the more concrete notion in Definition~\ref{def:uniformizers}.
\end{remark}

\subsection{Pinnings and their resulting ancillary data}

In this purely algebraic setting there is no analogue of the Weierstrass zeta-function $\zl(z)$ in \eqref{eq:zeta-wp-defs}, nor of the ``zeta function'' $\zq(z)$ in \eqref{eq:q-zeta-def}, which we respectively utilized throughout \S\ref{sec:weierstrass} and \S\ref{sec:tate} for various purposes. But we have this next analogue of the special $t=s$ cases of Lemma~\ref{lem:l-zeta-difference} and Lemma~\ref{lem:q-zeta-difference}.

\begin{deflem}\label{deflem:a-zeta-difference} For $\mathcal{U}$ a $\tau$-compatible set of local uniformizers as in Definition~\ref{def:uniformizers}, there exists a unique $\za{1}{0} \in\mathcal{L}\bigl([s]+[\idec]\bigr)$ such that
\begin{enumerate}
    \item $c_1^\mathcal{U}\bigl(\za{1}{0},s\bigr)=1$;
    \item $c_1^\mathcal{U}\bigl(\za{1}{0},\idec\bigr)=-1$; and
    \item $c_0^\mathcal{U}\bigl(\za{1}{0},\idec\bigr)=0$.
\end{enumerate}
\end{deflem}

\begin{proof}
    By the Riemann-Roch Theorem~\ref{thm:rr}, any $\zeta\in\mathcal{L}\bigl([s]+[\idec]\bigr)-\Ca$ can be uniquely scaled to satisfy condition (1), and then uniquely modified by an additive constant to satisfy condition (3). The fact that this unique $\zeta$ already satisfies condition (2) follows from \cite[Lem.~B.15]{Dreyfus2018}. We briefly recall the relevant argument here, merely to emphasize that this fact does not rely on the standing assumption made in \cite[Appendix~B]{Dreyfus2018} that the ground field ($\Ca$, in our case) be of characteristic zero. Given a choice of non-zero global section $\varpi\in H^0(\ec,\Omega^1_{\ec/\Ca})$, there is a unique derivation $\delta$ on $\bk$ such that the differentials $du=\delta(u)\varpi$ for every $u\in\bk$, and this $\delta$ commutes with~$\tau$ (cf.~\cite[Prop.~III.5.1]{SilvermanIntro}). For any $\varphi\in\bk$ and any $\alpha\in\ec(\Ca)$ whatsoever, the residue of the differential form $\varphi\varpi$ at $\alpha$ is $\mathrm{res}(\varphi\varpi,\alpha)=c_1^\mathcal{U}\bigl(\varphi/\delta(u_\alpha),\alpha\bigr)$. In our particular situation, we find that the residues of the differential form $\zeta\varpi$ are $\mathrm{res}(\zeta\varpi,s)=1/\delta(u_s)(s)$, $\mathrm{res}(\zeta\varpi,\idec)=c_1^\mathcal{U}(\zeta,\idec)/\delta(u_\idec)(\idec)$, and zero elsewhere, where $0\neq\delta(u_\alpha)(\alpha)$ denotes the evaluation of $\delta(u_\alpha)\in\bk$ at $\alpha\in\ec(\Ca)$. The fact that $\delta(u_s)(s)=\delta(u_\idec)(\idec)$ follows from the facts that $\tau(u_s)=u_\idec$ and $\tau\delta=\delta\tau$. Thus the Residue Theorem \cite[Thm.~III.7.14.2]{Hartshorne} applied to $\zeta\varpi$ implies that $c_1^\mathcal{U}(\zeta,\idec)=-1$.
\end{proof}

\begin{definition}\label{def:a-zeta-differences}
    With $\za{1}{0}\in\bk$ as in Definition~\ref{deflem:a-zeta-difference}, for $\ell,m\in\Z$ we define the elements $\za{\ell}{m}\in\bk$ as follows.
    \[\za{\ell}{m}:=\begin{cases}
        \displaystyle \sum_{i=m}^{\ell-1} \tau^{-i}\bigl(\za{1}{0}\bigr) & \text{if} \ \ell > m; \\
        0 \vphantom{\dfrac{A}{A}}& \text{if} \ \ell=m;\\
        \displaystyle -\sum_{i=\ell}^{m-1} \tau^{-i}\bigl(\za{1}{0}\bigr) & \text{if} \ \ell<m.     
    \end{cases}\]
\end{definition}

\begin{remark}\label{rem:a-zeta-differences} It is clear from Definition~\ref{def:a-zeta-differences} that, for every $\ell,m,n\in\Z$, $\za{\ell}{m}=-\za{m}{\ell}$, $\tau^n\bigl(\za{\ell}{m}\bigr)=\za{\ell-n}{m-n}$, and $\za{\ell}{m}+\za{m}{n}=\za{\ell}{n}$. Moreover, the defining properties of the basic $\za{1}{0}\in\mathcal{L}\bigl([s]+[\idec]\bigr)$ in Definition~\ref{deflem:a-zeta-difference} immediately imply that each $\za{\ell}{m}\in\mathcal{L}\bigl([\ell s]+[ms]\bigr)$ and, if $\ell\neq m$, then $c_1^\mathcal{U}\bigl(\za{\ell}{m},\ell s\bigr)=1$ and $c_1^\mathcal{U}\bigl(\za{\ell}{m},m s\bigr)=-1$. These properties make the $\za{\ell}{m}\in\bk$ into algebraic avatars of the zeta differences $\zl(z-\ell s)-\zl(z-ms)$ in the Weierstrass setting and $\zq(s^{-\ell}z)-\zq(s^{-m}z)$ in the Tate setting.
\end{remark}

The following definition is the algebraic replacement of Definition~\ref{def:pinningl} in the Weierstrass setting and Definition~\ref{def:pinningq} in the Tate setting.

\begin{definition}[$\tau$-pinnings - algebraic version]\label{def:pinninga}
    A $\tau$-\emph{pinning} of $\ec$ is a choice $\mathbf{P}=(\mathcal{R},\mathcal{U})$, consisting of a set $\mathcal{R}=\{\alpha_\omega\in\ec(\Cq) \ | \ \omega\in\orbita\}$ of representatives $\alpha_\omega\in\omega$ for each orbit $\omega\in\orbita$, and a set of $\tau$-compatible local uniformizers $\mathcal{U}=\{u_\alpha \ | \ \alpha\in\ec(\Ca)\}$ as in Definition~\ref{def:uniformizers}.
\end{definition}

The choice of $\tau$-pinning $\mathbf{P}=(\mathcal{R},\mathcal{U})$ of $\ec$ as in Definition~\ref{def:pinninga} gives rise to the ancillary data described in the following result, obtained as a straightforward consequence of the Riemann-Roch Theorem~\ref{thm:rr}.

\begin{deflem} \label{deflem:phia}Let $\mathbf{P}=(\mathcal{R},\mathcal{U})$ be any $\tau$-pinning of $\ec$ as in Definition~\ref{def:pinninga}, with $\mathcal{R}=\{\alpha_\omega\}_{ \omega\in\orbita}$. 
\begin{enumerate}
    
    \item Given $\omega\in\orbita$ such that $\alpha_\omega\neq\idec$, $n\in\Z$, and $k\geq 1$, there exists a unique $\varphi^\mathbf{P}_{\omega,n,k}\in\mathcal{L}\bigl(k[\alpha_\omega\oplus ns]+[ns]\bigr)$ such that \begin{enumerate}
    \item $c_k^\mathcal{U}\bigl(\varphi^\mathbf{P}_{\omega,n,k},\alpha_\omega\oplus ns\bigr)=1$; 
    \item $c_j^\mathcal{U}\bigl(\varphi^\mathbf{P}_{\omega,n,k},\alpha_\omega\oplus ns\bigr)=0$ for $j\in\mathbb{N}-\{k\}$; and 
    \item $c_0^\mathcal{U}\bigl(\varphi^\mathbf{P}_{\omega,n,k},ns\bigr)=0$.
    \end{enumerate}
    \item Given $n\in\Z$ and $k\geq 2$, there exists a unique $\hat{\varphi}^\mathcal{U}_{\Z s,n,k}\in\mathcal{L}\bigl(k[ns]\bigr)$ such that \begin{enumerate}
    \item $c_k^\mathcal{U}\bigl(\hat{\varphi}^\mathcal{U}_{\Z s,n,k},ns\bigr)=1$; and
        \item $c_j^\mathcal{U}\bigl(\hat{\varphi}^\mathcal{U}_{\Z s,n,k},ns\bigr)=0$ for every $j\in\Z_{\geq 0}-\{1,k\}$.
    \end{enumerate}
\end{enumerate}
In case $\alpha_{\Z s}=\idec$, we write $\varphi^\mathbf{P}_{\Z s,n,1}:=0$ and $\varphi^{\mathbf{P}}_{\Z s,n,k}:=\hat{\varphi}^{\mathcal{U}}_{\Z s,n,k}$ for every $n\in\Z$ and $k\geq 2$.
\end{deflem}

\begin{proof}
   (2).~Any non-constant element $\varphi\in\mathcal{L}(2[ns])$ can be uniquely scaled to satisfy condition (2a), and then further uniquely modified by an additive constant to satisfy condition (2b). The claim for $k>2$ follows by induction: any element $\varphi\in\mathcal{L}\bigl(k[ns]\bigr)-\mathcal{L}\bigl((k-1)[ns]\bigr)$ can be uniquely scaled to satisfy condition (2a), then uniquely modified by an additive constant to satisfy $c_0^\mathcal{U}(\varphi,ns)=0$, and then uniquely modified by a $\Ca$-linear combination of $\hat{\varphi}^\mathcal{U}_{\Z s,n,j}$ for $2\leq j \leq k-1$ to satisfy the rest of condition (2b).

    (1).~Any non-constant element $\varphi\in\mathcal{L}([\alpha_\omega\oplus ns]+[ns])$ can be uniquely scaled to satisfy condition (1a), and then further uniquely modified by an additive constant to satisfy condition (1c) -- condition (1b) is vacuous in this base case. The claim for $k>1$ follows by induction: any element $\varphi\in\mathcal{L}\bigl(k[\alpha_\omega \oplus ns]+[ns]\bigr)-\mathcal{L}\bigl((k-1)[\alpha_\omega\oplus ns]+[ns]\bigr)$ can be uniquely scaled to satisfy condition (1a), then uniquely modified by an additive constant to satisfy condition (1c), and then uniquely modified by a $\Ca$-linear combination of $\varphi^\mathbf{P}_{\omega,n,j}$ for $1\leq j \leq k-1$ to satisfy condition (1b).
\end{proof}

\begin{remark}\label{rem:phia}
    The uniqueness statements in Definition~\ref{deflem:phia}, together with the fact that $c_k^\mathcal{U}\bigl(\tau^n(\varphi),\alpha\bigr)=c_k^\mathcal{U}(\varphi,\alpha\oplus ns\bigr)$ for every $\varphi\in\bk$, $\alpha\in\ec(\Ca)$, and $k\geq 0$, immediately imply that $\tau^n\bigl(\varphi^\mathbf{P}_{\omega,n,k}\bigr)=\varphi^\mathbf{P}_{\omega,0,k}$ for every $\omega\in\orbita$, $n\in \Z$, and $k\geq 1$.
\end{remark}

The ancillary elements of $\bk$, which arise uniquely from the choice of \mbox{$\tau$-pinning} $\mathbf{P}=(\mathcal{R},\mathcal{U})$ of $\mathcal{E}$ according to Lemma~\ref{deflem:phia}, in turn give rise to certain structural constants associated with $\mathbf{P}$. The next result shows that these constants in fact depend only on the choice of $\tau$-compatible local uniformizers $\mathcal{U}$, and not on the choice $\mathcal{R}$ of $\tau$-orbit representatives.

\begin{lemma}\label{lem:d-r-independence} Let $\mathcal{U}=\{u_\alpha\}_{\alpha\in\ec(\Ca)}$ be a $\tau$-compatible set of local uniformizers and suppose that $\varphi\in\mathcal{L}\bigl(k[\alpha]+[ns]\bigr)$ for some $n\in\Z$, $k\geq 1$, and $ns\neq\alpha\in\ec(\Ca)$ such that \begin{equation}\label{eq:d-r-independence}c_j^\mathcal{U}(\varphi,\alpha)=\begin{cases} 1 & \text{for} \ j=k; \ \text{and} \\0 & \text{for} \ j\in\mathbb{N}-\{ k\}.\end{cases}\end{equation} Then $d^\mathcal{U}_{\omega,k}:=-c_1^\mathcal{U}(\varphi,ns)$ depends only on the orbit $\omega:=\alpha\oplus \Z s$ and the degree~$k$, and not on the particular choices of $\alpha\in\omega$, $n\in\Z$, nor $\varphi\in\mathcal{L}\bigl(k[\alpha]+[ns]\bigr)$ satisfying \eqref{eq:d-r-independence}. Moreover, $d^\mathcal{U}_{\Z s,1}=1$, and $c_1^\mathcal{U}\bigl(\hat{\varphi}^\mathcal{U}_{\Z s,n,k},ns\bigr)=-d^\mathcal{U}_{\Z s, k}$ for every $n\in\Z$ and $k\geq 2$, where $\hat{\varphi}^\mathcal{U}_{\Z s, n, k}$ is given as in Definition~\ref{deflem:phia}(2).
\end{lemma}

\begin{proof}
    As we saw in the proof of Definition~\ref{deflem:phia}(1), the Riemann-Roch Theorem~\ref{thm:rr} implies that conditions~\ref{eq:d-r-independence} uniquely determine $\varphi$ up to an irrelevant additive constant. Given another choice of $\alpha'=\alpha\oplus ms\in\omega$ and $n'\in\Z$, we need to show that if $\varphi'\in\mathcal{L}\bigl(k[\alpha']+[n's]\bigr)$ has $c_j^\mathcal{U}(\varphi',\alpha')$ satisfying the conditions in the right-hand side of \eqref{eq:d-r-independence} then $c_1^\mathcal{U}(\varphi',n's)=c_1^\mathcal{U}(\varphi,ns)$. Now note that
    $\varphi'':=\tau^m\bigl(\varphi'\bigr)\in\mathcal{L}\bigl(k[\alpha]+[(n'-m)s]\bigr)$
    still has $c_j^\mathcal{U}\bigl(\varphi'',\alpha\bigr)$ satisfying \eqref{eq:d-r-independence}, and $c_1^\mathcal{U}(\varphi'',n''s)=c_1^\mathcal{U}(\varphi',n's)$ with $n'':=n'-m$. Next note that
    $\varphi''':=\varphi''+c_1^\mathcal{U}(\varphi'',n''s)\cdot\za{n}{n''}\in\mathcal{L}\bigl(k[\alpha]+[ns]\bigr)$
    still has $c_j^\mathcal{U}(\varphi''',\alpha)$ satisfying \eqref{eq:d-r-independence} by Remark~\ref{rem:a-zeta-differences}, whence $c_1^\mathcal{U}(\varphi,ns)=c_1^\mathcal{U}(\varphi''',ns)=c_1^\mathcal{U}(\varphi',n's)$.
    
    The fact that $d_{\Z s,1}^{\mathcal{U}}=1$ follows from the observation that if $k=1$ and $\alpha=\ell s$ for some $n\neq \ell\in\Z$ then $\varphi-\za{\ell}{n}\in\mathcal{L}([ns])=\Ca$. Finally, for $k\geq 2$ we have that
    $\varphi:=\hat{\varphi}^\mathcal{U}_{\Z s, n, k}+c_1\bigl(\hat{\varphi}^\mathcal{U}_{\Z s,n,k},ns\bigr)\cdot\za{\ell}{n}\in\mathcal{L}\bigl(k[n s]+[\ell s]\bigr)$
    has $c_j(\varphi, \alpha)$ satisfying \eqref{eq:d-r-independence} with $\alpha= ns$ and $-d^\mathcal{U}_{\Z s,k}=c_1(\varphi, \ell s)=c_1\bigl(\hat{\varphi}^\mathcal{U}_{\Z s,n,k},ns\bigr)$. \end{proof}

    \begin{corollary}\label{cor:d-r-independence}
        With notation as in Lemma~\ref{lem:d-r-independence}, $d^\mathcal{U}_{\omega,1}\neq 0$ for every $\omega\in\orbita$.
    \end{corollary}

    \begin{proof}
        If $\varphi\in\mathcal{L}([\alpha]+[ns])$ satisfies \eqref{eq:d-r-independence} then $-d^\mathcal{U}_{\omega,1}=c^\mathcal{U}_1(\varphi,ns)\neq 0$, for otherwise we would have $\varphi\in\mathcal{L}([\alpha])=\Ca$ by the Riemann-Roch Theorem~\ref{thm:rr}, which would imply that also $c_1^\mathcal{U}(\varphi,\alpha)=0$, contradicting \eqref{eq:d-r-independence}.
    \end{proof}

    \begin{remark}\label{rem:super-compatible-d-independence} In some situations it might be desirable to ask for a set $\mathcal{U}$ of local uniformizers that enjoys a stronger property than being $\tau$-compatible as in Definition~\ref{def:uniformizers}. For $\alpha\in\ec(\Ca)$, denote by $\mathcal{T}_\alpha$ the automorphism of $\bk$ induced by the translation by $\alpha$ on $\ec$ (analogously to how $\tau$ is defined relative to $s$, so that $\mathcal{T}_s=\tau$). Let us call $\mathcal{U}$ \emph{supercompatible} if $\mathcal{T}_\alpha(u_\alpha)=u_\idec$ for every $\alpha\in\ec(\Ca)$. For example, the complex analytic uniformizers $u_\alpha=z-\alpha$ and the rigid analytic uniformizers $u_\alpha=1-a^{-1}z$ that we utilized respectively throughout Section~\ref{sec:weierstrass} and Section~\ref{sec:tate} enjoy such a supercompatibility property. If $\mathcal{U}$ is supercompatible then the proof of Lemma~\ref{lem:d-r-independence} shows that in fact the $d^\mathcal{U}_{\omega,k}\in\Ca$ depend only on $k\in\mathbb{N}$, and not on $\omega\in\orbita$.  
    \end{remark}

The independence statements in the above Lemma~\ref{lem:d-r-independence} serve to justify the rampant abuses of notation in the following definition.

\begin{definition}[Structural constants of a $\tau$-pinning] \label{def:structural-constants}
    With notation as in Definition~\ref{def:a-zeta-differences} and Definition~\ref{deflem:phia}, we define the following structural constants arising from the $\tau$-pinning $\mathbf{P}=(\mathcal{R},\mathcal{U})$ of $\ec$.
     \begin{enumerate}
        \item For $\omega\in\orbita$ and $k\geq 1$, we set (for some/any $n\in\Z$) \[d_{\omega,k}^{\mathcal{U}}:=\begin{cases}
        -c_1^\mathcal{U}\bigl(\varphi^\mathbf{P}_{\omega,n,k},ns\bigr) & \text{if} \ \alpha_{\omega}\neq\idec \ \text{or} \ k\geq 2;\\
            1 & \text{if} \ \alpha_{\omega}=\idec \ \text{and} \ k=1.
        \end{cases}\]
        \item For $\omega\in\orbita$, $n\in \Z$, and $k\geq 1$, we set \[e^\mathbf{P}_{\omega,n,k}:=c_0^\mathcal{U}\Bigl(\varphi^\mathbf{P}_{\omega,n,k} + d^\mathcal{U}_{\omega,k}\cdot\za{n}{0},\ \idec\Bigr).\]
    \end{enumerate}
\end{definition}

\subsection{Panorbital residues and summability}

\begin{definition}[Panorbital residues - algebraic version]\label{def:pano-a} Relative to the $\tau$-pinning $\mathbf{P}=(\mathcal{R},\mathcal{U})$, where $\mathcal{R}=\{\alpha_{\omega}\mid \omega\in\orbita\}$ and $\mathcal{U}=\{u_\alpha\mid \alpha\in \ec(\Ca)\}$, the \emph{panorbital residues} of $f\in\bke$ of orders $0$ and $1$ relative to $\mathbf{P}$ are respectively
\begin{equation}\label{eq:pano-a-def}
    \begin{gathered}
        \pano^{\mathbf{P}}(f,0):=c_0^\mathcal{U}(f,\idec)-\sum_{\omega\in\orbita}\sum\limits_{n\in\mathbb{Z}}\sum_{k\geq 1}c_k^\mathcal{U}(f,\alpha_{\omega}\oplus ns)\cdot e^\mathbf{P}_{\omega,n,k};\quad \text{and}\\
    \pano^{\mathbf{P}}(f,1):=\sum_{\omega\in\orbita}\sum_{n\in\mathbb{Z}} \sum_{k\geq 1} n\cdot c_k^{\mathcal{U}}\bigl(f,\alpha_{\omega}\oplus ns)\cdot d_{\omega,k}^{\mathcal{U}};
    \end{gathered}
\end{equation} where the structural constants $d^\mathcal{U}_{\omega,k},e^\mathbf{P}_{\omega,n,k}\in\Ca$ associated with $\mathbf{P}$ are as in Definition~\ref{def:structural-constants}.
\end{definition}

\begin{remark}\label{rem:pano-analogue}
    The analogy between Definition~\ref{def:pano-l} in the Weierstrass setting and Definition~\ref{def:pano-q} in the Tate setting might seem closer and more direct than the corresponding analogy with the present algebraic Definition~\ref{def:pano-a}. 

    In the Weierstrass setting, the role of the $\varphi^\mathbf{P}_{\omega,n,k}$ in Definition~\ref{deflem:phia} was played by the $\varphi^\mathcal{R}_{\omega,n,k}$ in \eqref{eq:phil}, where we see that the Weierestrass analogue of the structural constants $d^\mathcal{U}_{\omega,k}$ is $0$ for $k\geq 2$ and $1$ for $k=1$, for all $\omega\in\orbitl$. Thus we see that these structural constants were already invisibly present in the order~$1$ panorbital residue in Definition~\ref{def:pano-l}, and therefore the Weierstrass analogue of the order~$1$ panorbital residue in Definition~\ref{def:pano-a} agrees with the one given in Definition~\ref{def:pano-l}. Moreover, since the Weierstrass analogue of the algebraic $\za{n}{0}\in\bk$ is $\zl(z-ns)-\zl(z)\in\bkl$ as in Lemma~\ref{lem:l-zeta-difference}, we see that the Weierstrass version of the structural constant $e^\mathbf{P}_{\omega,n,k}$ in Definition~\ref{def:structural-constants} is given by the constant terms in the local Laurent series expansions $c_0\Bigl(\frac{(-1)^{k-1}}{(k-1)!}\frac{d^{k-1}}{dz^{k-1}}\zl(z-\alpha_\omega-ns),0\Bigr)$, taking into account that $c_0\bigl(\zl(z),0\bigr)=0$ in the $k=1$ case. Thus we see that the Weierstrass analogue of the order~$0$ panorbital residue in Definition~\ref{def:pano-a} also agrees with the one given in Definition~\ref{def:pano-l}.

    In the Tate setting, the role of the $\varphi^\mathbf{P}_{\omega,n,k}$ in Definition~\ref{deflem:phia} was played by the $\varphi^\mathcal{R}_{\omega,n,k}$ in \eqref{eq:phiq}, where we see that the Tate analogue of the structural constants $d^\mathcal{U}_{\omega,k}$ is $1$ for every $k\geq 1$ and $\omega\in\orbitq$. Thus we see that these structural constants were already invisibly present in the order~$1$ panorbital residue in Definition~\ref{def:pano-q}, and therefore the Tate analogue of the order~$1$ panorbital residue in Definition~\ref{def:pano-a} agrees with the one given in Definition~\ref{def:pano-q}. Moreover, since the Tate analogue of the algebraic $\za{n}{0}\in\bk$ is $\zq(s^{-n}z)-\zq(z)\in\bkq$ as in Lemma~\ref{lem:q-zeta-difference}, we see that the Tate version of the structural constant $e^\mathbf{P}_{\omega,n,k}$ in Definition~\ref{def:structural-constants} is given by the constant terms in the local Laurent series expansions $c_0\left(\sum_{i=1}^k\binom{k-1}{i-1}\zqd{i-1}\bigl(s^{-n}\alpha_\omega^{-1}\bigr),1\right)$, taking into account that $c_0\bigl(\zq(z),1\bigr)=0$. Thus we see that the Tate analogue of the order~$0$ panorbital residue in Definition~\ref{def:pano-a} also agrees with the one given in Definition~\ref{def:pano-q}.
    
    We refer to Lemma~\ref{lem:zetaExp-a} and Proposition~\ref{prop:diff-pano} for a complementary perspective on the closeness of the analogies among all three definitions of panorbital residues, along the lines of Remark~\ref{rem:pano-wq-comparison}.
\end{remark}

The following result can be considered as an algebraic analogue of Proposition~\ref{prop:zetaExp-l} in the Weierstrass setting and Proposition~\ref{prop:zetaExp-q} in the Tate setting.

\begin{lemma}\label{lem:zetaExp-a}
    Let $\mathbf{P}=(\mathcal{R},\mathcal{U})$ be any $\tau$-pinning of $\mathcal{E}$ and let $f\in\bk$. With notation as in Definition~\ref{def:a-zeta-differences},  Definition~\ref{deflem:phia}, and Definition~\ref{def:structural-constants},
    \begin{equation}\label{eq:zetaExp-a}
    f=\pano^\mathbf{P}(f,0)+\sum_{\omega\in\orbita}\sum_{n\in\Z}\sum_{k\geq 1}c^\mathcal{U}_k(f,\alpha_\omega\oplus ns)\cdot\Bigl(\varphi^\mathbf{P}_{\omega,n,k}+d^\mathcal{U}_{\omega,k}\cdot\za{n}{0}\Bigr).
    \end{equation}
\end{lemma}

\begin{proof}
The function $f'=f-\sum_{\omega,n,k}c^\mathcal{U}_k(f,\alpha_\omega\oplus ns)\bigl(\varphi^\mathbf{P}_{\omega,n,k}+d^\mathcal{U}_{\omega,k}\cdot\za{n}{0}\bigr)\in\bk$ belongs to $\mathcal{L}([\idec])=\Ca$, and therefore has the same value for $c_0^\mathcal{U}(f',\alpha)$ at any $\alpha\in\ec(\Ca)$, say at $\alpha=\idec$ (cf.~Remark~\ref{rem:pano-analogue}).
\end{proof}

Taking into account the fact that the structural constants $d_{\omega,1}^\mathcal{U}\neq 0$ for every $\omega\in\orbita$ by Corollary~\ref{cor:d-r-independence}, the following result strictly generalizes \cite[Lem.~B.15]{Dreyfus2018}.

\begin{corollary} \label{cor:zetaExp-a}For any $\tau$-compatible set of local uniformizers $\mathcal{U}$ and for any $f\in\bk$, \[\sum_{\omega\in\orbita}\sum_{k\geq 1}\ores^\mathcal{U}(f,\omega,k)\cdot d^\mathcal{U}_{\omega,k}=0.\]
\end{corollary}
\begin{proof}
    For any $\tau$-pinning $\mathbf{P}=(\mathcal{R},\mathcal{U})$ of $\ec$, the constant function $f'$ described in the proof of Lemma~\ref{lem:zetaExp-a} has $c_1^\mathcal{U}(f',\idec)=\sum_{\omega,k}\ores^\mathcal{U}(f,\omega,k)\cdot d^\mathcal{U}_{\omega,k}$.
\end{proof}

\begin{proposition} \label{prop:diff-pano}
    Let $\mathbf{P}=(\mathcal{R},\mathcal{U})$ be a $\tau$-pinning of $\ec$ with $\mathcal{U}=\{u_\alpha\}_{\alpha\in\ec(\Ca)}$ as in Definition~\ref{def:uniformizers}. For the unique invariant differential $0\neq\varpi\in H^0(\ec,\Omega^1_{\ec/\Ca})$ such that the valuation $\nu_\idec(du_\idec-\varpi)\geq 1$, we have
    \[\pano^\mathbf{P}(f,1)=\sum_{\omega\in\orbita}\sum_{n\in\Z}n\cdot\mathrm{res}(f\varpi,\alpha_\omega\oplus ns).\]
\end{proposition}
\begin{proof}
    For any initial choice of $0\neq \tilde{\varpi}\in H^0(\ec,\Omega^1_{\ec/\Ca})$, we have a local formal expansion \mbox{$\tilde{\varpi}=\left(\sum_{i\geq 0}\tilde{h}_iu_\idec^i\right)du_\idec$} with $\tilde{h}_0\neq 0$, and after scaling we obtain the unique $\varpi:=\tilde{h}_0^{-1}\tilde{\varpi}$ with $\nu_\idec(du_\idec-\varpi)\geq 1$. By Lemma~\ref{lem:zetaExp-a}, for each $\Z s\neq\omega\in\orbita$ and $n\in\Z$ we have \[\mathrm{res}(f\varpi,\alpha_\omega\oplus ns)=\sum_{k\geq 1}c_k^\mathcal{U}(f,\alpha_\omega\oplus ns)\cdot\mathrm{res}\bigl(\varphi^\mathbf{P}_{\omega,n,k}\varpi,\alpha_\omega\oplus ns\bigr).\] By the Residue Theorem \cite[Thm.~III.7.14.2]{Hartshorne}, \[\mathrm{res}\bigl(\varphi^\mathbf{P}_{\omega,n,k}\varpi,\alpha_\omega\oplus ns\bigr)=-\mathrm{res}\bigl(\varphi^\mathbf{P}_{\omega,n,k}\varpi,ns\bigr)=-c_1^\mathcal{U}\bigl(\varphi^\mathbf{P}_{\omega,n,k},ns\bigr)=d^\mathcal{U}_{\omega,k}.\] It remains to show that, for each $m\in\Z$, \[\mathrm{res}(f\varpi, ms)=\sum_{k\geq 1}c_k^\mathcal{U}(f,ms)\cdot d^\mathcal{U}_{\Z s,k}.\] A priori we only know that \[\mathrm{res}(f\varpi,ms)=\sum_{k\geq 1}c_k^\mathcal{U}(f,ms)\cdot \rho_{k-1},\] where $\rho_i\in\Ca$ are the coefficients occurring in the formal local expansion \mbox{$\varpi=\left(1+\sum_{i\geq 1}\rho_iu_{ms}^i\right)du_{ms}$} at $ms$. These coefficients $\rho_i\in\Ca$ are independent of $m\in\Z$ due to the $\tau$-compatibility of $\mathcal{U}$ and the translation invariance of $\varpi$ \cite[Prop.~III.5.1]{SilvermanIntro}. It suffices to show that $\rho_{k-1}=d^\mathcal{U}_{\Z s,k}$ from Definition~\ref{def:structural-constants}. By Lemma~\ref{lem:d-r-independence}, this is already proved in case $k=1$, and for the remaining $k\geq 2$ we have \[0=\mathrm{res}(\hat{\varphi}^\mathcal{U}_{\Z s,n,k}\varpi,ns)=\rho_{k-1}\cdot c_k^\mathcal{U}(\hat{\varphi}^\mathcal{U}_{\Z s,n,k},ns)+\rho_0\cdot c_1^\mathcal{U}(\hat{\varphi}^\mathcal{U}_{\Z s,n,k},ns)=\rho_{k-1}-d^\mathcal{U}_{\Z s,k},\] again by the Residue Theorem \cite[Thm.~III.7.14.2]{Hartshorne}.
\end{proof}

\begin{remark}\label{rem:diff-pano} The proof of Proposition~\ref{prop:diff-pano} above gives an alternative way to compute the structural constants $d^\mathcal{U}_{\Z s,k}$: since $\varpi=\sum_{k\geq 1}d^\mathcal{U}_{\Z s,k}u_\idec^{k-1}du_\idec$ and $du_\idec=\delta(u_\idec)\varpi$, it follows that \[\sum_{k\geq 1}d^\mathcal{U}_{\Z s,k}u_\idec^{k}=\frac{u_\idec}{\delta(u_\idec)}\in u_\idec\Ca[[u_\idec]].\] If $\mathcal{U}$ happens to be supercompatible (see Remark~\ref{rem:super-compatible-d-independence}), then this gives a recipe to compute every $d^\mathcal{U}_{\omega,k}\in\Ca$ at once without having to compute every local expansion prescribed by Definition~\ref{def:structural-constants}.    
\end{remark}

\begin{theorem}\label{thm:maina} Let $\mathbf{P}=(\mathcal{R},\mathcal{U})$ be any $\tau$-pinning of $\ec$, and let $f\in\bke$. Then $f$ is elliptically summable if and only if $\pano^{\mathbf{P}}(f,0)=0$, $\pano^{\mathbf{P}}(f,1)=0$, and $\ores^{\mathcal{U}}(f,\omega,k)=0$ for every $\omega\in\orbita$ and $k\geq 1$.
\end{theorem}

\begin{proof}
   First suppose $f=\tau(g)-g$ for some $g\in\bk$. From the fact that $c_k^\mathcal{U}(\tau(g),\alpha)=c_k^\mathcal{U}(g,\alpha\oplus s)$ for every $\alpha\in\ec(\Ca)$ and $k\geq 1$ follows the relation $\ores^\mathcal{U}(\tau(g),\omega,k)=\ores^\mathcal{U}(g,\omega,k)$ for every $\omega\in\orbitq$ and $k\geq 1$, and therefore every $\ores(f,\omega,k)=0$. Moreover, since each $\tau\bigl(\varphi^\mathbf{P}_{\omega,n,k}\bigr)=\varphi^\mathbf{P}_{\omega,n-1,k}$ (see Definition~\ref{deflem:phia} and Remark~\ref{rem:phia}) and each $\tau\bigl(\za{n}{0}\bigr)=\za{n-1}{0}+\za{0}{-1}$ (see Definition~\ref{def:a-zeta-differences} and Remark~\ref{rem:a-zeta-differences}), applying $\tau$ to both sides of the expansion \eqref{eq:zetaExp-a} for $g$, we find that similarly
   \begin{align*}
       \pano^\mathbf{P}(g,0)&=\pano^\mathbf{P}(\tau(g),0)-\za{0}{-1}\cdot\sum_{\omega\in\orbita}\sum_{k\geq 1}\ores^\mathcal{U}(g,\omega,k)\cdot d^\mathcal{U}_{\omega,k}\\&=\pano^\mathbf{P}(\tau(g),0)
   \end{align*}
   by Corollary~\ref{cor:zetaExp-a}, and therefore $\pano^\mathbf{P}(f,0)=0$. Finally, for the unique invariant differential $\varpi\in H^0(\ec,\Omega^1_{\ec/\Ca})$ such that $\nu_\idec(du_\idec-\varpi)\geq 1$ we similarly have that $\mathrm{res}(\tau(g)\varpi,\alpha)=\mathrm{res}(g\varpi,\alpha\oplus s)$, and thus by Proposition~\ref{prop:diff-pano} \begin{multline*}\pano^\mathbf{P}(\tau(g),1) =\sum_{\omega\in\orbita}\sum_{n\in \Z}n\cdot\mathrm{res}\bigl(g\varpi,\alpha_\omega\oplus (n+1)s\bigr)\\
    =\pano^\mathbf{P}(g,1)-\sum_{\omega\in\orbita}\sum_{n\in\Z}\mathrm{res}\bigl(g\varpi,\alpha_\omega\oplus ns\bigr)=\pano^\mathbf{P}(g,1),\end{multline*} by the Residue Theorem \cite[Thm.~III.7.14.2]{Hartshorne}. Thus $\pano^\mathcal{R}(f,1)=0$. 

    To prove the converse, let us suppose that every orbital residue and both panorbital residues of $f$ are zero. For every $\omega\in\orbita$, $n\in\Z$, and $k\geq 1$, let $\varphi^{\mathbf{P}}_{\omega,n,k}\in\bk$ be as in Definition~\ref{deflem:phia}, and let us consider

    \begin{equation}\label{eq:ftildea-def}\tilde{f}:=f+\sum_{\omega\in\orbita}\sum_{n\in\Z}\sum_{k\geq 1}c_k^{\mathcal{U}}(f,\alpha_{\omega}\oplus ns)\cdot\Bigl(\tau^n\bigl(\varphi^\mathbf{P}_{\omega,n,k}\bigr)-\varphi^\mathbf{P}_{\omega,n,k}\Bigr).\end{equation}

    Since $\tau^n(\varphi)-\varphi$ is summable for any $\varphi\in\bke$ and $n\in\Z$, $f$ is summable if and only if $\tilde{f}\in\bke$ is summable. Moreover, since each $\tau^n\bigl(\varphi^\mathbf{P}_{\omega,n,k}\bigr)=\varphi^\mathbf{P}_{\omega,0,k}$ (see Remark~\ref{rem:phia}), writing $f$ as in \eqref{eq:zetaExp-a} yields
    \begin{align} 
        \tilde{f} &=0+\sum_{\omega\in\orbita}\sum_{k\geq 1}\left(\ores^{\mathcal{U}}(f,\omega,k)\cdot \varphi_{\omega,0,k}^{\mathbf{P}}
        + \sum_{n\in\Z}c^\mathcal{U}_k(f,\alpha_\omega\oplus ns)\cdot d^\mathcal{U}_{\omega,k}\cdot\za{n}{0}\right) \notag\\ 
        &= 0+\sum_{\omega\in\orbita}\sum_{n\in\Z}\sum_{k\geq 1}c^\mathcal{U}_k(f,\alpha_\omega\oplus ns)\cdot d^\mathcal{U}_{\omega,k}\cdot\za{n}{0}, \label{eq:ftildea-short}
    \end{align} where the leftmost $0$'s in \eqref{eq:ftildea-short} arise from our assumptions that $\pano^\mathbf{P}(f,0)$ and the $\ores^\mathcal{U}(f,\omega,k)=0$ all vanish. Next let us again simplify notation by defining 
    \begin{align*}\tilde{c}(n)&:=c_1^\mathcal{U}(\tilde{f},ns)=\sum_{\omega\in\orbita}\sum_{k\geq 1}c_k^\mathcal{U}(f,\alpha_{\omega}\oplus ns)\cdot d_{\omega,k}^\mathcal{U}\in\Ca\qquad\text{for} \ n\neq 0;\qquad \text{and} \\ \tilde{c}(0)&:=c_1^\mathcal{U}(\tilde{f},\idec)=\sum_{\omega\in\orbita}\sum_{k\geq 1}\ores^\mathcal{U}(f,\omega,k)\cdot d^\mathcal{U}_{\omega,k}=0\qquad\text{(cf.~Corollary~\ref{cor:zetaExp-a}),}\end{align*}
     so that $\tilde{f} = \sum_{n\in\mathbb{Z}} \tilde{c}(n)\za{n}{0}$ in \eqref{eq:ftildea-short} and $\sum_{n\in\Z} n\tilde{c}(n) = \pano^{\mathbf{P}}(f,1)$. Let us also define the auxiliary functions as in Definition~\ref{def:a-zeta-differences}
    \begin{equation}\label{eq:psia-def}
    \psi_j:=\begin{cases}
        \za{j}{1} & \text{for} \ j\geq 2;\\
        \za{j}{0} & \text{for} \ j\leq-1.
    \end{cases}    
    \end{equation}
    We now consider the further reduction of $f$ by summable elements of $\bke$:
    \begin{equation}\label{eq:fbara-def}\bar{f}:=\tilde{f}+\sum_{n\leq -1}\sum_{j=n}^{-1}\tilde{c}(n)\cdot\Bigl(\tau^{-1}\bigl(\psi_j\bigr)-\psi_j\Bigr)+\sum_{n\geq 2}\sum_{j=2}^n\tilde{c}(n)\cdot \Bigl(\tau\bigl(\psi_j\bigr)-\psi_j\Bigr).
    \end{equation}
    Once again we see that $f$ is summable if and only if $\bar{f}$ is summable. Moreover, we find from the Definition~\ref{def:a-zeta-differences} and the basic properties detailed in Remark~\ref{rem:a-zeta-differences} that
    \begin{equation}\label{eq:psia-comp1}
        \sum_{j=2}^n\Bigl(\tau\bigl(\psi_j\bigr)-\psi_j\Bigr)=n\za{1}{0}-\za{n}{0}
    \end{equation} for $n\geq 2$, and similarly that
    \begin{equation}\label{eq:psia-comp2}
        \sum_{j=n}^{-1}\Bigl(\tau^{-1}\bigl(\psi_j\bigr)-\psi_j\Bigr)=n\za{1}{0}-\za{n}{0}
    \end{equation} for $n\leq -1$. It follows that $\bar{f}\in\mathcal{L}\bigl([s]+[\idec]\bigr)$. Thus, writing the $\tilde{f}$ as in \eqref{eq:ftildea-short} in \eqref{eq:fbara-def}, the computations \eqref{eq:psia-comp1} and \eqref{eq:psia-comp2} yield that \vspace{-.07in}
     \[
        \bar{f}=\overbrace{\tilde{c}(0)\za{0}{0}}^{0\cdot0}+\tilde{c}(1)\za{1}{0}+ \sum_{n\in\mathbb{Z}-\{0,1\}}n\tilde{c}(n)\za{1}{0}=\pano^\mathbf{P}(f,1)\cdot\za{1}{0}=0.
    \] This concludes the proof that $f$ is summable.    \end{proof}

\begin{remark}\label{rem:reduced-form-a}
    Whether or not any orbital or panorbital residues of $f\in\bk$ vanish, defining $\tilde{f}$ in terms of $f$ as in \eqref{eq:ftildea-def}, and subsequently $\bar{f}$ in terms of $\tilde{f}$ as in \eqref{eq:fbara-def}, results in the \emph{reduced form} for $f$ (cf.~\cite[Lem.~B.14]{Dreyfus2018}):
    \begin{multline*}\bar{f}=\pano^\mathbf{P}(f,0)+\pano^\mathbf{P}(f,1)\cdot\za{1}{0}\\+\sum_{\omega\in\orbita}\sum_{k\geq 1}\ores^\mathcal{U}(f,\omega,k)\cdot\varphi^\mathbf{P}_{\omega,0,k},\end{multline*} relative to a choice of $\tau$-pinning $\mathbf{P}=(\mathcal{R},\mathcal{U})$ as in Definition~\ref{def:pinninga} and with $\varphi^\mathbf{P}_{\omega,0,k}\in\bk$ as in Definition~\ref{deflem:phia} and $\za{1}{0}\in\bk$ as in Definition~\ref{deflem:a-zeta-difference} (cf.~Remark~\ref{rem:reduced-form-l} and Remark~\ref{rem:reduced-form-q}), and such that $\bar{f}-f$ is summable.
\end{remark}

In Remark~\ref{rem:pano-pinning-l} and Remark~\ref{rem:pano-pinning-q} we explicitly tracked the effect of the choice of pinning on the panorbital residues defined respectively in the Weierstrass and Tate settings. The following results accomplish the analogous goal in the present algebraic setting. We address separately the simpler effect of modifying the set of representatives $\mathcal{R}$ in Lemma~\ref{lem:pano-pinning-ar}, and the more complicated effect of modifying the set of $\tau$-compatible local uniformizers $\mathcal{U}$ in Lemma~\ref{lem:pano-pinning-au}.

\begin{lemma}\label{lem:pano-pinning-ar} Let $\mathbf{P}=(\mathcal{R},\mathcal{U})$ be a $\tau$-pinning of $\ec$ with $\mathcal{R}=\{\alpha_\omega\}_{\omega\in\orbita}$ as in Definition~\ref{def:pinninga}. For $\mathcal{R}'=\{\alpha'_\omega\}_{\omega\in\orbita}$ any other choice of $\tau$-orbit representatives, denote by $n_\omega\in\Z$ be the unique integers such that $\alpha'_\omega=\alpha_\omega\oplus n_\omega s$ for each $\omega\in\orbita$. Then for $\mathbf{P}':=(\mathcal{R}',\mathcal{U})$ we have
\begin{align}\label{eq:pano-pinning-ar1} 
    \pano^{\mathbf{P}'}(f,1)&=\pano^\mathbf{P}(f,1)-\!\!\sum_{\omega\in\orbita}\sum_{k\geq 1}n_\omega \ores^\mathcal{U}(f,\omega,k) d_{\omega,k}^{\mathcal{U}};\\
    \pano^{\mathbf{P}'}(f,0)&=\pano^\mathbf{P}(f,0)+\!\!\sum_{\omega\in\orbita} \sum_{k\geq 1}\ores^\mathcal{U}(f,\omega,k) e^\mathbf{P}_{\omega,n_\omega,k}.\label{eq:pano-pinning-ar0}
\end{align}
In particular, the panorbital residues of $f\in\bk$ are independent of the choice of $\mathcal{R}$ whenever all the orbital residues of $f$ vanish.
\end{lemma}

\begin{proof}
We see immediately from Definition~\ref{def:pano-a} that \eqref{eq:pano-pinning-ar1} holds.
Moreover, the function $\varphi^{\mathbf{P}'}_{\omega,0,k}-\varphi^\mathbf{P}_{\omega,n_\omega,k}-d_{\omega,k}^\mathcal{U}\cdot\za{n_\omega}{0}\in\mathcal{L}(0)=\Ca$ is $\tau$-invariant and equal to its degree $0$ term at $\idec$, whence this constant
\begin{align*}
    -e^\mathbf{P}_{\omega,n_\omega,k}&=c_0^{\mathcal{U}}\Bigl(\varphi^{\mathbf{P}'}_{\omega,0,k}-\varphi^\mathbf{P}_{\omega,n_\omega,k}-d^\mathcal{U}_{\omega,k}\cdot\za{n_\omega}{0},\ \idec\Bigr) \\&=c_0^{\mathcal{U}}\Bigl(\varphi^{\mathbf{P}'}_{\omega,n,k}-\varphi^\mathbf{P}_{\omega,n_\omega+n,k}-d^\mathcal{U}_{\omega,k}\cdot\za{n_\omega+n}{n}, \ \idec\Bigr)=e^{\mathbf{P}'}_{\omega,n,k}-e^\mathbf{P}_{\omega,n_\omega+n,k}\end{align*} independently of $n$ (cf.~Remark~\ref{rem:a-zeta-differences} and Remark~\ref{rem:phia}). With this, \eqref{eq:pano-pinning-ar0} follows from the definition \eqref{eq:pano-a-def}.\end{proof}

\begin{lemma}\label{lem:pano-pinning-au} Let $\mathbf{P}=(\mathcal{R},\mathcal{U})$ be a $\tau$-pinning of $\ec$ as in Definition~\ref{def:pinninga}, with $\mathcal{R}=\{\alpha_\omega\}_{\omega\in\orbita}$ and $\mathcal{U}=\{u_\alpha\}_{\alpha\in\ec(\Ca)}$. Let $\mathcal{U}'=\{u_\alpha'\}_{\alpha\in\ec(\Ca)}$ be another choice of $\tau$-compatible set of local uniformizers, and let $h_{\Z s,0},h_{\Z s,1}\in\Ca$ be given by $h_{\Z s,\ell}:=c^{\mathcal{U}'}_{1-\ell}\bigl(u_\idec^{-1},\idec\bigr)$ for $\ell\in\{0,1\}$. Then for $\mathbf{P}':=(\mathcal{R},\mathcal{U}')$ we have
\begin{equation}\label{eq:pano-pinning-au1}
    \pano^{\mathbf{P}'}(f,1) =h_{\Z s,0}\cdot \pano^\mathbf{P}(f,1).\end{equation}Moreover, provided that $\alpha_{\Z s}\neq\idec$, we also have
    \begin{equation}
    \pano^{\mathbf{P}'}(f,0) =\pano^\mathbf{P}(f,0)-h_{\Z s,1}\cdot\pano^\mathbf{P}(f,1),\label{eq:pano-pinning-au0-nice}\end{equation}whereas in the exceptional situation where $\alpha_{\Z s}=\idec$ we have
    \begin{multline}\label{eq:pano-pinning-au0-mean}
    \pano^{\mathbf{P}'}(f,0) =\pano^\mathbf{P}(f,0)-h_{\Z s,1}\cdot\pano^\mathbf{P}(f,1)\\ +\sum_{n\in\Z}\Bigl(c^{\mathcal{U}'}_0(f,ns)-c^{\mathcal{U}}_0(f,ns)\Bigr).
\end{multline}
\end{lemma}

\begin{proof} Recall that one can compute $h^{(m)}_{\omega,j}\in\Ca$ for $m,j\geq 0$ and $\omega\in\orbita$ such that each $c_k^{\mathcal{U}'}(\varphi,\alpha)=\sum_{m\geq k}h^{(m)}_{\omega,m-k}\cdot c_m^{\mathcal{U}}(\varphi,\alpha)$ for every $\varphi\in\bk$, $\alpha\in \omega$, and $k\geq 1$ simultaneously as in \eqref{eq:u-coeff-effect} in~Remark~\ref{rem:ores-uniformizers}, where we denoted $h^{(1)}_{\omega,k}=h_{\omega,k}$ consistently with the usage of $h_{\Z s,0}$ and $h_{\Z s,1}$ in the statement of Lemma~\ref{lem:pano-pinning-au}.

Letting $\varpi\in H^0(\ec,\Omega^1_{\ec/\Ca})$ such that the valuation $\nu_\idec(du_\idec-\varpi)\geq 1$ as in Proposition~\ref{prop:diff-pano}, we find that $\varpi':=h_{\Z s,0}\cdot\varpi$ similarly satisfies $\nu_\idec(du_\idec'-\varpi')\geq 1$ (cf.~Remark~\ref{rem:ores-uniformizers}). We obtain \eqref{eq:pano-pinning-au1} by a twofold application of Proposition~\ref{prop:diff-pano}:
\begin{multline*}\pano^{\mathbf{P}'}(f,1)=\sum_{\omega\in\orbita}\sum_{n\in\Z}n\cdot\mathrm{res}(f\varpi',\alpha_\omega\oplus ns)\\=h_{\Z s,0}\sum_{\omega\in\orbita}\sum_{n\in\Z}n\cdot\mathrm{res}(f\varpi,\alpha_\omega\oplus ns)=h_{\Z s,0}\pano^\mathbf{P}(f,1).\end{multline*}
    
 By Lemma~\ref{lem:zetaExp-a}, \eqref{eq:pano-pinning-au0-nice} will follow from the forthcoming computation, valid in case $\alpha_{\Z s}\neq \idec$, that \begin{gather}\label{eq:pano-pinning-au0-proof} \sum_{\omega\in\orbita}\sum_{n\in\Z}\sum_{k\geq 1}c^{\mathcal{U}'}_k(f,\alpha_\omega\oplus ns)\cdot\Bigl(\varphi^{\mathbf{P}'}_{\omega,n,k}+d^{\mathcal{U}'}_{\omega,k}\cdot\zeta^{\mathcal{U}'}_{(n,0)}\Bigr)\\ =h_{\Z s,1}\cdot\pano^\mathbf{P}(f,1)+\sum_{\omega\in\orbita}\sum_{n\in\Z}\sum_{k\geq 1}c^{\mathcal{U}}_k(f,\alpha_\omega\oplus ns)\cdot\Bigl(\varphi^{\mathbf{P}}_{\omega,n,k}+d^{\mathcal{U}}_{\omega,k}\cdot\zeta^{\mathcal{U}}_{(n,0)}\Bigr),\notag\end{gather} where $\varphi_{\omega,n,k}^{\mathbf{P}}$ and $\varphi_{\omega,n,k}^{\mathbf{P}'}$ are as in Definition~\ref{deflem:phia}, $\za{n}{0}$ and $\zeta^{\mathcal{U}'}_{(n,0)}$ are as in Definition~\ref{def:a-zeta-differences}, and $d^\mathcal{U}_{\omega,k}$ and $d^{\mathcal{U}'}_{\omega,k}$ are as in Definition~\ref{def:structural-constants}. To begin, \mbox{we find that}
 
 \begin{equation}\label{eq:zeta-u-effect}
        h_{\Z s,0}\cdot\zeta^{\mathcal{U}'}_{(n,0)}=\za{n}{0}+h_{\Z s,1}\cdot n;
    \end{equation}
    first in the basic case $n=1$ by comparing local expansions at $\Ca ((u_\idec'))$ and using the computations that $c_1^{\mathcal{U}'}(\za{1}{0},s)=h_{\Z s,0}$ and $c_0^{\mathcal{U}'}(\za{1}{0},\idec)=h_{\Z s,1}$ by \eqref{eq:u-coeff-effect} and Definition~\ref{deflem:a-zeta-difference}, and then for arbitrary $n\in\Z$ from the Definition~\ref{def:a-zeta-differences}. Similarly, comparing local expansions in $\Ca((u_{\alpha_\omega\oplus ns}'))$, we find for every $\omega\in\orbita$, $n\in\Z$, and $k\geq 1$, that
    \begin{equation}\label{eq:phi-u-effect}
        \sum_{j=1}^k h^{(k)}_{\omega,k-j}\cdot\varphi^{\mathbf{P}'}_{\omega,n,j}=\varphi^\mathbf{P}_{\omega,n,k}-c^{\mathcal{U}'}_0\bigl(\varphi^\mathbf{P}_{\omega,n,k},\,ns\bigr).
    \end{equation} Indeed, the difference between the two sides of \eqref{eq:phi-u-effect} belongs to $\mathcal{L}([ns])=\Ca$ by \eqref{eq:u-coeff-effect}, and the the condition that each $c^{\mathcal{U}'}_0(\varphi^{\mathbf{P}'}_{\omega,n,j},ns)=0$ in Definition~\ref{deflem:phia} yields the equality.
    Further comparing local expansions in $\Ca((u_{ns}'))$, we obtain from \eqref{eq:phi-u-effect} that \begin{equation}\label{eq:d-u-effect}
    \sum_{j=1}^k h^{(k)}_{\omega,k-j}\cdot d^{\mathcal{U}'}_{\omega,j}=h_{\Z s,0}\cdot d^\mathcal{U}_{\omega,k}.
    \end{equation}

    Provided that we are in the non-exceptional case where $\alpha_\omega\neq\idec$, we compute from \eqref{eq:u-coeff-effect} that the constant correction term in \eqref{eq:phi-u-effect} is \begin{equation}\label{eq:phi-u-effect-constant}
        -c_0^{\mathcal{U}'}\bigl(\varphi^\mathbf{P}_{\omega,n,k},ns\bigr)=h_{\Z s,1}\cdot d^\mathcal{U}_{\omega,k}.
    \end{equation} Therefore, for each $\omega\in\orbita$ such that $\alpha_\omega\neq\idec$ and for each $n\in \Z$, we find that \begin{multline}\label{eq:pano-pinning-au0-long-computation-nice}
    \sum_{k\geq 1} c^{\mathcal{U}'}_k(f,\alpha_\omega\oplus ns)\cdot\Bigl(\varphi^{\mathbf{P}'}_{\omega,n,k}+d^{\mathcal{U}'}_{\omega,k}\cdot\zeta^{\mathcal{U}'}_{(n,0)}\Bigr)\\
      =\sum_{k\geq 1}\left(\sum_{m\geq k} h^{(m)}_{\omega,m-k}\cdot c^{\mathcal{U}}_m(f,\alpha_\omega\oplus ns)\right)\cdot\Bigl(\varphi^{\mathbf{P}'}_{\omega,n,k}+d^{\mathcal{U}'}_{\omega,k}\cdot\zeta^{\mathcal{U}'}_{(n,0)}\Bigr) \\
      =\sum_{k\geq 1}c_k^\mathcal{U}(f,\alpha_\omega\oplus ns)\cdot\sum_{j=1}^k h^{(k)}_{\omega,k-j}\cdot\Bigl(\varphi^{\mathbf{P}'}_{\omega,n,j}+d^{\mathcal{U}'}_{\omega,j}\cdot\zeta^{\mathcal{U}'}_{(n,0)}\Bigr)\\
      =\sum_{k\geq 1}c_k^\mathcal{U}(f,\alpha_\omega\oplus ns)\cdot\Bigl(\varphi^\mathbf{P}_{\omega,n,k}+h_{\Z s,1}d_{\omega,k}^{\mathcal{U}}+d^\mathcal{U}_{\omega,k}\za{n}{0}+d^{\mathcal{U}}_{\omega,k}h_{\Z s,1}n\Bigr),
    \end{multline} where the first equality follows from \eqref{eq:u-coeff-effect}, the second equality is a rearrangement, and the third equality is obtained from \eqref{eq:phi-u-effect}, \eqref{eq:phi-u-effect-constant},  \eqref{eq:d-u-effect}, and \eqref{eq:zeta-u-effect}. Thus, provided that $\alpha_{\Z s}\neq\idec$, summing over all $\omega\in\orbita$ and $n\in\Z$ yields \eqref{eq:pano-pinning-au0-proof}, in view of Corollary~\ref{cor:zetaExp-a}.

    It remains to prove \eqref{eq:pano-pinning-au0-mean} in the exceptional case where $\alpha_{\Z s}=\idec$. With $\hat{\varphi}^\mathcal{U}_{\Z s,n,k}=\varphi^\mathbf{P}_{\omega,n,k}$ and $\hat{\varphi}^{\mathcal{U}'}_{\omega,n,k}=\varphi^{\mathbf{P}'}_{\omega,n,k}$ as in Definition~\ref{deflem:phia}(2), we find from \eqref{eq:u-coeff-effect} that in this case the constant correction term in \eqref{eq:phi-u-effect} is \begin{equation}\label{eq:pano-pinning-au0-phihat-constant}
        -c_0^{\mathcal{U}'}\bigl(\hat{\varphi}^\mathcal{U}_{\Z s,n,k},ns\bigr)=h_{\Z s, 1}\cdot d^\mathcal{U}_{\Z s,k}-h^{(k)}_{\Z s,k},
    \end{equation} so that \eqref{eq:phi-u-effect}, which remains valid as stated in all cases and regardless of whether $\alpha_{\Z s}=\idec$, specializes in this case to
    \begin{equation}\label{eq:phihat-u-effect}
        \sum_{j=1}^k h^{(k)}_{\Z s,k-j}\cdot\hat{\varphi}^{\mathcal{U}'}_{\Z s,n,j}=\hat{\varphi}^\mathcal{U}_{\Z s,n,k}+h_{\Z s,1}\cdot d^\mathcal{U}_{\Z s,k}-h^{(k)}_{\Z s,k}.
    \end{equation} The relations \eqref{eq:zeta-u-effect} and \eqref{eq:d-u-effect} also remain valid, since they do not depend on the choice of $\mathcal{R}$ at all -- although one can also tortuously recover \eqref{eq:d-u-effect} for $\omega=\Z s$ by carefully comparing local $\Ca((u_{ns}'))$-expansions in \eqref{eq:phihat-u-effect} and using the fact that $d^{\mathcal{U}'}_{\Z s,1}=1=d^\mathcal{U}_{\Z s,1}$ by Lemma~\ref{lem:d-r-independence}. Thus the analogous computation to \eqref{eq:pano-pinning-au0-long-computation-nice} in this case yields that, for each $n\in\Z$,
\begin{gather}\label{eq:pano-pinning-au0-long-computation-mean} \sum_{k\geq 1}c^{\mathcal{U}'}_k(f,ns)\cdot\Bigl(\hat{\varphi}^{\mathcal{U}'}_{\Z s,n,k}+d^{\mathcal{U}'}_{\Z s,k}\cdot\zeta^{\mathcal{U}'}_{(n,0)}\Bigr)\\
=\sum_{k\geq 1}c^\mathcal{U}_k(f,ns)\cdot(\hat{\varphi}^\mathcal{U}_{\Z s,n,k}+d^\mathcal{U}_{\Z s,k}\cdot\za{n}{0}+h_{\Z s,1}\cdot(n+1)\cdot d^\mathcal{U}_{\Z s,k} -h^{(k)}_{\Z s,k}\Bigr).\notag
\end{gather} Thus in case $\alpha_{\Z s}=\idec$ we find that instead of \eqref{eq:pano-pinning-au0-long-computation-nice} we obtain
\begin{multline} \label{eq:pano-pinning-au0-long-computation-mean-result}
    \sum_{\omega\in\orbita}\sum_{n\in\Z}\sum_{k\geq 1}c^{\mathcal{U}'}_k(f,\alpha_\omega\oplus ns)\Bigl(\hat{\varphi}^{\mathcal{U}'}_{\omega,n,k}+d^{\mathcal{U}'}_{\omega,k}\cdot\zeta^{\mathcal{U}'}_{(n,0)}\Bigr)\\=h_{\Z s,1}\cdot\pano^\mathbf{P}(f,1)+\sum_{\omega\in\orbita}\sum_{n\in\Z}\sum_{k\geq 1}c^{\mathcal{U}}_k(f,\alpha_\omega\oplus ns)\cdot\Bigl(\hat{\varphi}^{\mathcal{U}'}_{\omega,n,k}+d^{\mathcal{U}}_{\omega,k}\cdot\zeta^{\mathcal{U}}_{(n,0)}\Bigr)\\-\sum_{n\in\Z}\sum_{k\geq 1}h^{(k)}_{\Z s,k}\cdot c^\mathcal{U}_k(f,ns).
\end{multline}The sought relation \eqref{eq:pano-pinning-au0-mean} follows from \eqref{eq:u-coeff-effect} because, for each $n\in\Z$, \[\sum_{k\geq 1}h^{(k)}_{\Z s,k}\cdot c^\mathcal{U}_k(f,ns)=c^{\mathcal{U}'}_0(f,ns)-c_0^\mathcal{U}(f,ns).\qedhere\]
\end{proof}

\begin{corollary}\label{cor:pano-pinning-au}
    With notation as in Lemma~\ref{lem:pano-pinning-au}, suppose $\alpha_{\Z s}=\idec$ and let $h^{(k)}_{\Z s,k}\in\Ca$ for $k\geq 1$ be defined as in Remark~\ref{rem:ores-uniformizers}. Then \begin{multline}\label{eq:pano-pinning-au0-alt}
        \pano^{\mathbf{P}'}(f,0)=\pano^\mathbf{P}(f,0)-h_{\Z s,1}\cdot\pano^\mathbf{P}(f,1)\\+\sum_{k\geq 1}h^{(k)}_{\Z s,k}\cdot\ores^{\mathcal{U}}(f,\Z s,k).
    \end{multline}
\end{corollary}

\begin{proof}
    The last summation in \eqref{eq:pano-pinning-au0-long-computation-mean-result} can be rearranged in two different ways:\begin{multline*}\sum_{n\in\Z}\Bigl(c^{\mathcal{U}'}_0(f,ns)-c^{\mathcal{U}}_0(f, ns)\Bigr)=\sum_{n\in\Z}\sum_{k\geq 1}h^{(k)}_{\Z s,k}\cdot c^\mathcal{U}_k(f,ns)\\=\sum_{k\geq 1}h^{(k)}_{\Z s,k}\cdot\ores^\mathcal{U}(f,\Z s,k).\qedhere\end{multline*}
\end{proof}

\begin{remark}\label{rem:pano-pinning-a}
    The delicate computations in the proof of Lemma~\ref{lem:pano-pinning-au} all rest on a kind of expansion-contraction principle: the basic change-of-variables expansion of the $c^\mathcal{U}_k(f,\alpha)$ in \eqref{eq:u-coeff-effect} gets absorbed by the ancillary functions $\varphi^\mathbf{P}_{\omega,n,k},\za{n}{0}\in\bk$ and constants $d^\mathcal{U}_{\omega,k}\in\Ca$ engendered by the choice of $\mathbf{P}$, according to \eqref{eq:phi-u-effect}, \eqref{eq:zeta-u-effect}, and \eqref{eq:d-u-effect}, respectively. These special elements privilege the special orbit $\Z s$ among all others, and the special basepoint $\idec$ in particular, in a way that their Weierstrass and Tate counterparts do not (cf.~Remark~\ref{rem:pano-wq-comparison}). This helps explain the emergence of the additional correction terms in \eqref{eq:pano-pinning-au0-mean} whenever it happens that $\alpha_{\Z s}=\idec$. One way to avoid this this would be to disallow $\alpha_{\Z s}=\idec$ in the Definition~\ref{def:pinninga} -- an artificial restriction that we prefer not to make. One could also try to use whatever $\alpha_{\Z s}$ is in place of $\idec$ in the Definition~\ref{deflem:phia} of the $\varphi^\mathbf{P}_{\omega,n,k}$, and subsequently in the corresponding analogues of the $\za{\ell}{m}$ in Definition~\ref{def:a-zeta-differences}, which is probably workable but would have the perverse effect of obscuring the latter elements' independence from the choice of $\mathcal{R}$. These undesirable consequences seem too high a price to pay for the sake of having the conclusion of Lemma~\ref{lem:pano-pinning-au} be independent of the choice of $\mathcal{R}$.
\end{remark}

\subsection{Basic examples} \label{subsec:basic-examples-q}

To conclude this section, we illustrate the above general results with some basic examples. We choose once and for all an arbitrary $\tau$-pinning $\mathbf{P}=(\mathcal{R},\mathcal{U})$ of $\ec$ as in Definition~\ref{def:pinninga}, with $\mathcal{R}=\{\alpha_\omega\}_{\omega\in\orbitq}$ and $\mathcal{U}=\{u_\alpha\}_{\alpha\in\ec(\Ca)}$. Denote by $\varpi\in H^0\bigl(\Omega_{\ec/\Ca},\ec\bigr)$ the unique invariant differential such that $\nu_\idec(du_\idec -\varpi)\geq 1$, as in Proposition~\ref{prop:diff-pano}.

 \begin{example} \label{ex:zeta-difference-a} Consider $\za{\ell}{m}\in\bk$ as in Definition~\ref{def:a-zeta-differences} for some $\ell,m\in\Z$. We see from this definition (see Remark~\ref{rem:a-zeta-differences}) that every $\ores^\mathcal{U}\bigl(\za{\ell}{m},\omega,k)=0$. Denote by $n_\idec\in\Z$ the integer such that $\alpha_{\Z s}=n_\idec s$, so that $js=\alpha_{\Z s}\oplus(j-n_\idec)s$. Then \[\pano^\mathbf{P}\bigl(\za{\ell}{m},1\bigr)=(\ell-n_\idec)\cdot 1\cdot d^\mathcal{U}_{\Z s,1}+(m-n_\idec)\cdot(-1)\cdot d^\mathcal{U}_{\Z s,1}=\ell-m,\] since by Definition~\ref{def:structural-constants} the structural constant $d^\mathcal{U}_{\Z s,1}=1$. Therefore by Theorem~\ref{thm:maina} $\za{\ell}{m}$ cannot be summable unless $\ell=m$, in which case $\za{\ell}{m}=0$. By Lemma~\ref{lem:zetaExp-a}, \begin{multline*}
     \za{\ell}{m}\\ =\pano^\mathbf{P}\bigl(\za{\ell}{m},0\bigr)+1\cdot\Bigl(\varphi^\mathbf{P}_{\Z s,\ell-n_\idec,1}+\za{\ell-n_\idec}{0}\Bigr)-1\cdot\Bigl(\varphi^\mathbf{P}_{\Z s,m-n_\idec,1}+\za{m-n_\idec}{0}\Bigr)\\
     =\pano^\mathbf{P}\bigl(\za{\ell}{m},0\bigr)+\tau^{m-\ell}\bigl(\varphi^\mathbf{P}_{\Z s,m-n_\idec,1}\bigr)-\varphi^\mathbf{P}_{\Z s,m-n_\idec,1}+\tau^{n_\idec}\bigl(\za{\ell}{m}\bigr),
 \end{multline*} where we have used the fact that $\za{\ell-n_\idec}{0}-\za{m-n_\idec}{0}=\za{\ell-n_\idec}{m-n_\idec}$ (see Remark~\ref{rem:a-zeta-differences}) and the fact that $\tau^{m-\ell}(\varphi^\mathbf{P}_{\Z s,m-n_\idec,1})=\varphi^\mathbf{P}_{\Z s,\ell-n_\idec}$ (see Remark~\ref{rem:phia}). It follows that $\pano^\mathbf{P}\bigl(\za{\ell}{m},0)\in\Ca$ is summable, so we must actually have that $\pano^\mathbf{P}\bigl(\za{\ell}{m},0)=0$. \end{example}

 \begin{example}\label{ex:phi-a}
    For $\omega\in\orbita$, $n\in\Z$, and $k\geq 1$, consider the elements $\varphi^\mathbf{P}_{\omega,n,k}\in\bk$ from Definition~\ref{deflem:phia}. Let us also assume for the moment that we do not have both $\omega=\Z s$ and $k=1$, since in this case we would have \[\varphi^\mathbf{P}_{\Z s,n,1}=\za{\ell}{m}+\pano^\mathbf{P}(\varphi^\mathbf{P}_{\Z s,n,1},0)\] for some $m,\ell\in\Z$, which we can reduce to Example~\ref{ex:zeta-difference-a}, except for the computation of $\pano^\mathbf{P}(\varphi^\mathbf{P}_{\Z s,n,1},0)$. Then we have \begin{align*}-d^\mathcal{U}_{\omega,k}=c_1^\mathcal{U}\bigl(\varphi^\mathbf{P}_{\omega,n,k},ns\bigr)&=\ores^\mathcal{U}\bigl(\varphi^\mathbf{P}_{\omega,n,k},\Z s,1);\\
1=c^\mathcal{U}_k\bigl(\varphi^\mathbf{P}_{\omega,n,k},\alpha_\omega\oplus ns\bigr)&=\ores^\mathcal{U}\bigl(\varphi^\mathbf{P}_{\omega,n,k},\omega,k\bigr);
    \end{align*} and every other orbital residue of $\varphi^\mathbf{P}_{\omega,0,k}$ vanishes. But we can already see that $\varphi^\mathbf{P}_{\omega,n,k}$ is not summable, by Theorem~\ref{thm:main} (unless $\omega=\Z s$ and $k=1$, in which case $\alpha_{\Z s}=\idec$ if and only if $\varphi^\mathbf{P}_{\Z s,n,1}=0$ is summable). Moreover, we see that \[\pano^\mathbf{P}\bigl(\varphi^\mathbf{P}_{\omega,n,k},1)=n\cdot1\cdot d^\mathcal{U}_{\omega,k}+n\cdot(-d^\mathcal{U}_{\omega,k})\cdot d^\mathcal{U}_{\Z s,1}=0.\] From now on we allow once again the possibility that $\omega=\Z s$ and $k=1$ simultaneously. By Lemma~\ref{lem:zetaExp-a}, \begin{multline*}\varphi^\mathbf{P}_{\omega,n,k}=\pano^\mathbf{P}\bigl(\varphi^\mathbf{P}_{\omega,n,k},0\bigr)\\ +1\cdot\bigl(\varphi^\mathbf{P}_{\omega,n,k}+d^\mathcal{U}_{\omega,k}\za{n}{0}\bigr)-d^\mathcal{U}_{\omega,k}\cdot\bigl(\varphi^\mathbf{P}_{\Z s,n-n_\idec,1}+d^\mathcal{U}_{\Z s,1}\za{n-n_\idec}{0}\bigr),\end{multline*} where $\alpha_{\Z s}=n_\idec s$, so that $ns=\alpha_{\Z s}\oplus (n-n_\idec)s$. Therefore \begin{multline*}\pano^\mathbf{P}\bigl(\varphi^\mathbf{P}_{\omega,n,k},0\bigr)=d^\mathcal{U}_{\omega,k}\cdot\bigl(\za{n}{n-n_\idec}-\varphi^\mathbf{P}_{\Z s,n-n_\idec,1}\bigr)\\=d^\mathcal{U}_{\omega,k}\cdot\tau^{n_\idec-n}\bigl(\za{n_\idec}{0}-\varphi^\mathbf{P}_{\Z s,0,1}\bigr)=d^\mathcal{U}_{\omega,k}\cdot c^\mathcal{U}\bigl(\za{n_\idec}{0},\idec\bigr).\end{multline*}
\end{example}

 \begin{example} \label{ex:psi-sum-a}For $n\in\Z$, consider the elements
    \[\Psi_n(z):=n\za{1}{0}-\za{n}{0}\in\bk.\] Since the orbital residues and panorbital residues are $\Ca$-linear, we find that $\ores(\Psi_n,\Z\check{s},1)=0$ and $\pano^\mathbf{P}\bigl(\Psi_n,0\bigr)=0$ from Example~\ref{ex:zeta-difference-a}, and \[\pano^\mathbf{P}\bigl(\Psi_n,1\bigr)=n\cdot(1-0)-1\cdot(n-0)=0.\] Therefore by Theorem~\ref{thm:mainq} each $\Psi_n(z)$ is summable. This agrees with the computations \eqref{eq:psia-comp1} and \eqref{eq:psia-comp2} in case $n\neq0,1$, and $\Psi_0(z)=\Psi_1(z)=0$. 
\end{example}

\section{Applications} \label{sec:applications}

\subsection{Reminders and notation}

In this section we continue using the notation and assumptions set out at the beginning of Section~\ref{sec:algebraic}: $\bk$ is the field of rational functions on an elliptic curve $(\ec,\idec)$ over the algebraically closed field $\Ca$, which is still allowed to be of arbitrary characteristic, and $\tau$ is the automorphism of $\bk$ induced by the translation on $\ec$ with respect to a non-torsion $\Ca$-rational point $s\in\ec(\Ca)$. We neither require nor forbid the existence of a Weierstrass uniformization $\ecl$ of $\ec$ identifying $\bk$ with $\bkl$ as in Section~\ref{sec:weierstrass}, nor of a Tate uniformization $\ecq$ of $\ec$ identifying $\bk$ with $\bkq$ as in Section~\ref{sec:tate}. But we shall work in the general algebraic setting, and explain at appropriate intervals how the general results of this section can be made simpler and more concrete if one has access to the complex/rigid analytic tools that are available in those special situations.

Let us choose once and for all an invariant differential $0\neq \varpi\in H^0(\ec,\Omega^1_{\ec/\Ca})$, and denote by $\delta$ the unique dual derivation on $\bk$ such that the differential $df=\delta(f)\varpi$ for each $f\in\bk$. In the Weierstrass setting we choose $\varpi=dz$, so that $\delta=\frac{d}{dz}$, whereas in the Tate setting we choose $\varpi=\frac{dz}{z}$, so that $\delta=z\frac{d}{dz}$.

\subsection{An analogue of the Riemann-Roch Theorem} \label{sec:summable-rr}

Our goal in this subsection is to prove an analogue of the Riemann-Roch Theorem~\ref{thm:rr} for summable elliptic functions in Theorem~\ref{thm:summable-rr}. We begin by defining some useful variants of usual notions, adapted to our context.

\begin{definition}\label{def:tau-divisors} Let $ D:=\displaystyle\sum_{\alpha\in\ec(\Ca)}m_\alpha[\alpha]\in\mathrm{Div}(\ec)$ as in \eqref{eq:divisor-def}.
\begin{enumerate}
    \item The \emph{summable Riemann-Roch space} of $D$ is \[\mathcal{S}(D):=\bigl\{f\in\mathcal{L}(D) \ \big| \ f \ \text{is summable}\bigr\}.\]
    \item The $\tau$\emph{-support} of $D$ is \[\mathrm{supp}_\tau(D):=\{\omega\in\orbita \ | \ m_\alpha\neq 0\ \text{for some} \ \alpha\in\omega\}.\]
    \item For $\omega\in\orbita$, the $\omega$\emph{-component} of $D$ is $D_\omega:=\sum_{\alpha\in\omega}m_\alpha[\alpha]$.
\end{enumerate} 
\end{definition}

We begin with the following basic preliminary result.

\begin{lemma}\label{lem:summable-dispersion} If $f\in\bk$ is summable, then for each $\omega\in\orbita$ the set of poles of $f$ in $\omega$ is either empty or else contains at least two elements.   
\end{lemma}

\begin{proof} The following basic argument also dates back to at least \cite{Abramov:1971}. A version adapted to the elliptic context is given in \cite[p.~196]{Dreyfus2018}, which we follow here. Let $\mathcal{R}=\{\alpha_\omega\}_{\omega\in\orbita}$ be a choice of $\tau$-orbit representatives. For $g\in\bk-\Ca$, there exist $\omega\in\orbita$; $m,M\in\Z$ with $m\leq M$; and $k\in\mathbb{N}$; such that: $g$ has poles of order $k$ at $\alpha_\omega\oplus Ms$ and at $\alpha_\omega\oplus ms$; $g$ has no poles of order greater than $k$ anywhere in $\omega$; and every pole of $g$ of order $k$ in $\omega$ is contained in $\{\alpha_\omega\oplus ns \ | \ m\leq n \leq M\}$. Then $\tau(g)$ has a pole of order $k$ at $\alpha_\omega\oplus (m-1)s$; $\tau(g)$ has no poles of order greater than $k$ anywhere in $\omega$; and every pole of $\tau(g)$ of order $k$ in $\omega$ is contained in $\{\alpha_\omega\oplus ns \ | \ m-1\leq n \leq M-1\}$. It follows that $\tau(g)-g$ has poles at $(m-1)s$ and at $Ms$.
\end{proof}

Next is our analogue of the Riemann-Roch Theorem~\ref{thm:rr}. It computes the dimension of the $\Ca$-vector space $\mathcal{S}(D)$ of Definition~\ref{def:tau-divisors} for any effective $D\in\mathrm{Div}(\ec)$.

\begin{theorem}\label{thm:summable-rr} Let $ D=\sum_{\alpha\in\ec(\Ca)}m_\alpha[\alpha]$ be an effective divisor. For $\omega\in\orbita$, let $M_\omega:=\mathrm{max}\{m_\alpha \ | \ \alpha\in\omega\}$. \begin{enumerate}
\item If $\omega\in\orbita$ is such that $|\mathrm{supp}(D_\omega)|\leq 1$ then $\mathcal{S}(D)=\mathcal{S}(D-D_\omega)$.
    \item If $D\neq 0$ and $|\mathrm{supp}(D_\omega)|\geq 2$ for each $\omega\in\mathrm{supp}_\tau(D)$, then
 \[\mathrm{dim}_\mathcal{C}\bigl(\mathcal{S}(D)\bigr)=\mathrm{deg(D)}-1-\sum_{\omega\in\orbita}M_\omega.\]
\end{enumerate}
\end{theorem}

\begin{proof} Item (1) follows immediately from Lemma~\ref{lem:summable-dispersion}. Supposing $D\neq 0$ and $|\mathrm{supp}(D_\omega)|\geq 2$ for each $\omega\in\mathrm{supp}_\tau(D)$ as in item (2), let us choose a $\tau$-pinning $\mathbf{P}=(\mathcal{R},\mathcal{U})$ of $\ec$ as in Definition~\ref{def:pinninga} such that $\mathcal{U}=\{u_\alpha\}_{\alpha\in\ec(\Ca)}$ satisfies $\nu_\idec(du_\idec-\varpi)\geq 1$ as in Proposition~\ref{prop:diff-pano}, and such that $\mathcal{R}=\{\alpha_\omega\}_{\omega\in\orbita}$ satisfies $\alpha_\omega\in\mathrm{supp}(D)$ for each $\omega\in\mathrm{supp}_\tau(D)$. 

A general $f\in\mathcal{L}(D)$ is completely described, up to an additive constant, by the $\mathrm{deg}(D)$-many parameters \[\bigl\{c^\mathcal{U}_k(f,\alpha) \ \big| \ \alpha\in\mathrm{supp}(D) \ \text{and} \ 1\leq k\leq m_\alpha\bigr\}.\] These parameters are not free, but rather are subject to a single non-trivial linear relation arising from the the Residue Theorem~\cite[Thm.~III.7.14.2]{Hartshorne}:
\begin{equation}\label{eq:summable-rr-relation}0=\sum_{\alpha\in\ec(\Ca)}\mathrm{res}(f\varpi,\alpha)=   \sum_{\omega\in\orbita}\sum_{k\geq 1}\ores^\mathcal{U}(f,\omega,k)\cdot d^\mathcal{U}_{\omega,k}
\end{equation}
(cf.~Corollary~\ref{cor:zetaExp-a}), where the second equality can also be deduced from the fact that $\mathrm{res}(f\varpi,\alpha)=\sum_{k\geq 1}c^\mathcal{U}_{k}(f,\alpha)\cdot d^\mathcal{U}_{\omega,k}$ for each $\omega\in\orbita$ and $\alpha\in\omega$, as we saw in the proof of Proposition~\ref{prop:diff-pano}. Accounting for the additional datum of the additive constant, with which we can modify any $f$ without changing its $c^\mathcal{U}_k(f,\alpha)$, results in the correct dimension $\mathrm{dim}_\mathcal{C}\bigl(\mathcal{L}(D)\bigr)=\mathrm{deg}(D)-1+1$ according to the Riemann-Roch Theorem~\ref{thm:rr}. It follows that the subspace \begin{gather*}\mathcal{S}'(D):=\bigl\{f\in\mathcal{L}(D) \ \big| \ \ores^\mathcal{U}(f,\omega,k)=0 \ \text{for every} \ \omega\in\orbita \ \text{and} \ k\geq 1\bigr\}\\ \text{has} \qquad\mathrm{dim}_\Ca\bigl(\mathcal{S}'(D)\bigr)=\mathrm{deg}(D)+1-\sum_{\omega\in\orbita}M_\omega.\end{gather*} Thus it remains to show that the $\Ca$-linear functionals $\xi_0(\bullet):=\pano^\mathbf{P}(\bullet,0)$ and $\xi_1(\bullet):=\pano^\mathbf{P}(\bullet,1)$ are linearly independent, considered as elements of the $\Ca$-dual vector space $\mathcal{S}'(D)^*$, since $\mathcal{S}(D)=\mathrm{ker}(\xi_0)\cap\mathrm{ker}(\xi_1)$. To see this, note that any $0\neq c\in\Ca\subset\mathcal{S}'(D)$ has $\xi_1(c)=0$ and $\xi_0(c)=c\neq 0$. Moreover, it follows from our assumptions that, for any given $\omega\in\mathrm{supp}_\tau(D)$, there exists $0\neq n\in\Z$ such that $\mathcal{L}([\alpha_\omega\oplus ns]+[\alpha_\omega])\subseteq \mathcal{L}(D)$. Since every $d^\mathcal{U}_{\omega,1}\neq 0$ by Corollary~\ref{cor:d-r-independence}, it follows from the relation \eqref{eq:summable-rr-relation} that $\ores^\mathcal{U}(\zeta,\omega,1)=0$ for every $\zeta\in\mathcal{L}([\alpha_\omega\oplus ns]+[\alpha_\omega])$, whence $\mathcal{L}([\alpha_\omega\oplus ns]+[\alpha_\omega])\subseteq\mathcal{S}'(D)$. Finally, note that for any $\zeta\in\mathcal{L}([\alpha_\omega\oplus ns]+[\alpha_\omega])-\Ca$ we must have $c^\mathcal{U}_1(\zeta,\alpha_\omega\oplus ns)\neq 0$, and therefore $\zeta':=\zeta-\xi_0(\zeta)$ satisfies $\xi_1(\zeta')=n\cdot c^\mathcal{U}_1(\zeta,\alpha_\omega\oplus ns)\neq 0$ and $\xi_0(\zeta')=0$. This concludes the proof.
\end{proof}

\subsection{Differential integrability and gauge equivalence to constants} \label{sec:integrability}

A linear difference system of order $N\geq 1$ over $(\bk,\tau)$ is one of the form \begin{equation}\label{eq:mat-difeq}
    \tau(Y)=AY\qquad\text{with}\quad A\in\mathrm{GL}_N(\bk),
\end{equation} where $Y$ denotes a column vector of formal unknowns and $\tau$ is applied componentwise. Selecting a cyclic vector in the $(\bk,\tau)$-module associated with~\eqref{eq:mat-difeq} (see \cite[Appendix~B]{hendriks-singer:1999}), one obtains an equivalent homogeneous scalar difference equation of order $N$:\begin{equation}\label{eq:scal-difeq}
    \tau^N(y)+\sum_{i=0}^{N-1}a_i\tau^i(y)=0 \qquad\text{with}\quad a_0,\dots,a_{N-1}\in\bk \quad \text{s.t.} \quad a_0\neq 0.  
\end{equation} On the other hand, given a scalar difference equation \eqref{eq:scal-difeq} one writes its companion system of the form \eqref{eq:mat-difeq} in the usual way (see \cite[p.~5]{vanderput-singer:1997}).

In case $\mathrm{char}(\Ca)=0$, the differential Galois theory of \cite{HardouinSinger2008} associates with \eqref{eq:mat-difeq} a differential Galois group that encodes in its algebraic structure the differential equations over $\bk$ that are satisfied by the solutions to \eqref{eq:mat-difeq}, with respect to the $\tau$-invariant derivation $\delta$ on $\bk$. A fundamental question in this theory is whether the system \eqref{eq:mat-difeq} can be extended to a compatible difference-differential system \begin{equation}\label{eq:mat-integrability}\begin{cases}
    \tau(Y)=AY & A\in\mathrm{GL}_N(\bk);\\
    \delta(Y)=BY & B\in\mathfrak{gl}_N(\bk);
\end{cases}\qquad \text{such that} \quad \delta(A)A^{-1}=\sigma(B)-B.\end{equation} The system \eqref{eq:mat-difeq} is called $\delta$-\emph{integrable} if there exists a $B\in\mathfrak{gl}_N(\bk)$ as in \eqref{eq:mat-integrability}. Similarly, we say that the scalar equation \eqref{eq:scal-difeq} is $\delta$-integrable if its associated companion system as in \eqref{eq:mat-difeq} is $\delta$-integrable, which we see is equivalent to the existence of $b_1,\dots,b_N\in\bk$ such that \begin{equation}\label{eq:scal-integrability}
    a_i\cdot\frac{\delta(a_0)}{a_0}-\delta(a_i)=\tau(b_i)-b_i\qquad\text{for}\quad i=1,\dots,N,
\end{equation} where we write $a_N:=1$. 

Whether testing for the summability of each entry of $\delta(A)A^{-1}$ in \eqref{eq:mat-integrability} or of each $a_i\frac{\delta(a_0)}{a_0}-\delta(a_i)$ in \eqref{eq:scal-integrability}, we see that Theorem~\ref{thm:maina} mediates an immediate equivalence between computing orbital and panorbital residues and deciding $\delta$-integrability of linear difference equations over $\ec$. Our results below go beyond this immediate application of Theorem~\ref{thm:maina}, to show that some of the (pan)orbital obstructions to $\delta$-integrability are redundant already for the simplest first-order (multiplicative) systems and second-order unipotent (additive) systems. This is in concordance with some upcoming interesting results of Charlotte Hardouin and Julien Roques, which they have kindly shared with us in \cite{hardouin-roques:2023} before their formal dissemination, concerning the precise nature of the meromorphic solutions over $\C$ to compatible difference-differential systems \eqref{eq:mat-integrability} of arbitrary order over complex elliptic curves $\ecl$.

Although the differential Galois theory for difference equations of \cite{HardouinSinger2008} is restricted to characteristic zero, we nonetheless choose to address the integrability of difference equations as discussed above in arbitrary characteristic, even though the ``correct'' version of differential integrability of a difference system \eqref{eq:mat-difeq} in arbitrary characteristic $p\geq 0$ should presumably be made relative to a system of (iterative) Hasse-Schmidt derivations $\{\delta^{(k)}\}_{k\in\mathbb{N}}$ (cf.~\cite{iterative}) that is suitably compatible with $\tau$. Even though we do not pursue this path here explicitly, we consider our results below as initial steps in that direction.

\subsubsection{Additive systems}

We begin by addressing the simplest case of a difference system as in \eqref{eq:mat-difeq} of the special form \begin{equation}\label{eq:additive-mat}\tau(Y)=\begin{pmatrix} 1 & f \\ 0 & 1\end{pmatrix}Y,\end{equation} for some $0\neq f\in\bk$, which is equivalent to the companion system for the homogenization of the inhomogeneous first-order difference equation \begin{equation}\label{eq:additive-scal}\tau(y)-y=f.\end{equation} For $A:=\left(\begin{smallmatrix} 1 & f \\ 0 & 1\end{smallmatrix}\right)$ the coefficient matrix in \eqref{eq:additive-mat}, we compute $\delta(A)A^{-1}=\left(\begin{smallmatrix} 0 & \delta(f) \\ 0 & 0\end{smallmatrix}\right)$. Hence the $\delta$-integrability of \eqref{eq:additive-mat} is equivalent to the summability of $\delta(f)$. By Theorem~\ref{thm:maina}, the system \eqref{eq:additive-mat} is $\delta$-integrable if and only if the orbital and panorbital residues of $\delta(f)$ all vanish. Our goal in this subsection is to prove a new characterization of the $\delta$-integrability of \eqref{eq:additive-mat} in Theorem~\ref{thm:additive-integrability} below. 

\begin{theorem}\label{thm:additive-integrability} Let $\mathrm{char}(\Ca)=:p\geq 0$. Let $f\in\bk$, let $\mathcal{U}$ be a $\tau$-compatible set of local uniformizers as in Definition~\ref{def:uniformizers}, and let $d^\mathcal{U}_{\omega,k}\in\Ca$ be the structural constants from Definition~\ref{def:structural-constants}. Then $\delta(f)$ is summable if and only if the following conditions are satisfied.
\begin{enumerate}
    \item $\ores^\mathcal{U}(f,\omega,k)=0$ for every $\omega\in\orbita$ and $k\geq 1$ such that $k\notin p\Z$; and
    \item either $p=0$ or else $\displaystyle \sum_{\omega\in\orbita}\sum_{k\geq 1}\bigl(d^\mathcal{U}_{\omega,k}\bigr)^p\cdot\ores^\mathcal{U}(f,\omega,pk)=0$.
\end{enumerate}
In particular, if $\ores^\mathcal{U}(f,\omega,k)=0$ for every $\omega\in\orbita$ and $k\geq 1$, then $\delta(f)$ is summable, regardless of the characteristic of $\Ca$.
\end{theorem}

The proof of this result will rely on some earlier results and ideas developed in \cite[Appendix~B]{Dreyfus2018}, which was the first systematic study of the differential properties of solutions to difference equations over elliptic curves, but is restricted to characteristic zero. Below we give a compact summary of these results and methods, which of them do or do not extend to arbitrary characteristic, and how they compare with our results and methods.

\begin{theorem}[\protect{\cite[Prop.~B.5 \& Prop.~B.8]{Dreyfus2018}}]\label{thm:dhrs-additive}
    Suppose $\mathrm{char}(\Ca)=0$. Let $f\in\bk$ and let $\mathcal{U}$ be a $\tau$-compatible set of local uniformizers as in Definition~\ref{def:uniformizers}. The following statements are equivalent.
    \begin{enumerate}
        \item There exists $0\neq \mathcal{D}\in\Ca[\delta]$ such that $\mathcal{D}(f)$ is summable.
        \item There exist $\alpha\in\ec(\Ca)$ and $f^*\in\mathcal{L}([\alpha\oplus s]+[\alpha])$ such that $f^*-f$ is summable.
        \item $\ores^\mathcal{U}(f,\omega,k)=0$ for every $\omega\in\orbita$ and $k\geq 1$.
    \end{enumerate}
\end{theorem}

\begin{remark}\label{rem:dhrs-additive} The implications (3) $\Leftrightarrow$ (2) $\Rightarrow$ (1) in Theorem~\ref{thm:dhrs-additive} remain valid in arbitrary characteristic, with the same proofs as those given in \cite{Dreyfus2018}.

In fact, the reduced form $\bar{f}\in\mathcal{L}([s]+[\idec])$ for $f$ that we compute explicitly in the course of our proof of Theorem~\ref{thm:maina} (cf.~Remark~\ref{rem:reduced-form-a}) is produced by a kind of distillation of the reduction procedure developed in \cite{Dreyfus2018} to establish the equivalence (3) $\Leftrightarrow$ (2) in the proof of \cite[Prop.~B.8]{Dreyfus2018}. Neither reduction procedure relies on their standing assumption in \cite[Appendix~B]{Dreyfus2018} that $\mathrm{char}(\Ca)=0$.

Moreover, the argument given in \cite[p.~198]{Dreyfus2018} for the implication \mbox{(2) $\Rightarrow$ (1)} proves even more than it states: namely, that in (1) one can take either $\mathcal{D}=\delta$ or $\mathcal{D}=\delta^2-c\delta$ for some $c\in\Ca$. They show this as follows. The vanishing of all the orbital residues of $f$ implies the same property for their $\delta(f^*)$ (this is easy to see in characteristic zero, and not much more difficult to see in positive characteristic). Further reducing $\delta(f^*)^*\in\mathcal{L}([\alpha\oplus s]+[\alpha])$, they obtain a $\Ca$-linear combination of $f^*$ and $1$ (by the Riemann-Roch Theorem~\ref{thm:rr}), whence differentiating this linear relation results in the desired $\mathcal{D}=\delta^2-c\delta$ such that $\mathcal{D}(f)$ is summable. In the particular situation when their $f^*\in\Ca$ (i.e., if all the orbital residues of $f$ vanish as well as its first-order panorbital residue all vanish -- see again Remark~\ref{rem:reduced-form-a}), then they also obtain the better $\mathcal{D}=\delta$ in our Theorem~\ref{thm:additive-integrability}, since in that case $\delta(f^*)=0$ is already summable. Thus the implication (2) $\Rightarrow$ (1) proved in \cite[Prop.~B.5]{Dreyfus2018} indeed remains valid in arbitrary characteristic. Of course, the implication (1) $\Rightarrow$ (2) proved in \cite[Prop.~B.5]{Dreyfus2018} when $\mathrm{char}(\Ca)=0$ would only be weakened by adjoining to (1) any additional specifications for $\mathcal{D}$.

The implication (1) $\Rightarrow$ (2) is not valid in positive characteristic. The most immediate counterexample is over a Tate curve $\ecq$ as in Section~\ref{sec:tate} with $\mathrm{char}(\Ca)=2$: for $\wp_q\in\bkq$ defined as in \eqref{eq:q-wp-def}, $\delta(\wp_q)+\wp_q=0$ but $\ores(\wp_q,\Z\check{s},2)=1$. More generally, still over a Tate curve $\ecq$ and with $\mathrm{char}(\Ca)=p\geq 2$, the element $\zqd{p-1}\in\bkq$ specified in Definition~\ref{def:q-zeta-fake-derivatives} satisfies $\delta\bigl(\zqd{p-1}\bigr)+\zqd{p-1}=0$ and $\ores\bigl(\zqd{p-1},\Z\check{s},p\bigr)=1$, as we see from \eqref{eq:q-zeta-coeffs}.

Thus our Theorem~\ref{thm:additive-integrability} represents a best-possible extension of (1) $\Leftrightarrow$ (3) in Theorem~\ref{thm:dhrs-additive} to arbitrary characteristic, as well as a sharpening of the implication (3) $\Rightarrow$ (1) already in characteristic zero.
\end{remark}

The proof of Theorem~\ref{thm:additive-integrability} below will be streamlined with the aid of several basic computational results. The first, Lemma~\ref{lem:delta-zeta-summable} below, represents a special but crucial case of Theorem~\ref{thm:additive-integrability} (cf.~Example~\ref{ex:zeta-difference-l}, Example~\ref{ex:zeta-difference-q}, and Example~\ref{ex:zeta-difference-a}). It is the source of our sharpening of the implication (3)~$\Rightarrow$~(1) in Theorem~\ref{thm:dhrs-additive} when $\mathrm{char}(\Ca)=0$ (see Remark~\ref{rem:dhrs-additive}).

\begin{lemma} \label{lem:delta-zeta-summable}
    For any $\alpha\in\ec(\Ca)$ and for any $\zeta\in\mathcal{L}([\alpha\oplus s]+[\alpha])$, $\delta(\zeta)$ is summable.
\end{lemma}

\begin{proof}  It suffices to prove this for any one single $\zeta\in\mathcal{L}([s]+[\idec])-\Ca$. To see that the general case follows from this special one, denote by $\mathcal{T}_\alpha$ the automorphism of $\bk$ induced by the translation by $\alpha$ on $\ec$ (analogously to how $\tau$ is defined relative to $s$, so that $\mathcal{T}_s=\tau$). Note that for any $\zeta'\in\mathcal{L}([\alpha\oplus s]+[\alpha])$ as in the statement, $\mathcal{T}_\alpha(\zeta')=a\zeta+b$ for some $a,b\in\Ca$, with respect to any given choice of non-constant $\zeta\in\mathcal{L}([s]+[\idec])$, by the Riemann-Roch Theorem~\ref{thm:rr}. Therefore if $\delta(\zeta)$ is summable then so is the arbitrary $\delta(\zeta')$, since the automorphism $\mathcal{T}_\alpha$ commutes with $\delta$ and with $\tau$.

The special case results from an elementary computation, which we abridge below for the reader's benefit. For $\wp\in\mathcal{L}(2[\idec])-\Ca$ and $\wp'\in\mathcal{L}(3[\idec])-\mathcal{L}(2[\idec])$ satisfying \eqref{eq:alg-weierstrass-eq}:\begin{equation}\label{eq:delta-zeta-weierstrass}W(\wp,\wp'):=(\wp')^2+a_1\wp\wp'+a_3\wp'-\wp^3-a_2\wp^2-a_4\wp-a_6=0,\end{equation} let us write $\wp(\ominus s)=\wp( s)=:r\in\Ca$ and $-\wp'(s)-a_1\wp(s)-a_3=\wp'(\ominus s)=:t\in\Ca$, so that by \cite[\S III.2]{SilvermanIntro}: \begin{equation}\label{eq:delta-zeta-summable-tau}\tau^{-1}(\wp)=\left(\frac{\wp'-t}{\wp-r}\right)^2 +a_1\left(\frac{\wp'-t}{\wp-r}\right)-a_2-\wp-r.\end{equation}
Moreover, by \cite[Thm.~A1.1]{LangElliptic} there exists $0\neq c\in\Ca$ such that \begin{equation}\label{eq:delta-zeta-summable-derivatives}c\cdot\delta(\wp)=2\wp'+a_1\wp+a_3\quad\text{and}\quad c\cdot\delta(\wp')=3\wp^2+2a_2\wp+a_4-a_1\wp'.\end{equation}
We see that $\mathrm{div}(\wp-r)=[s]+[\ominus s]-2[\idec]$, and we compute similarly that $\mathrm{div}(\wp'-t)=[\ominus s]+[\beta]+[s\ominus\beta]-3[\idec]$ for some irrelevant $\beta\in\ec(\Ca)-\{\idec,s\}$ (because $\nu_\idec(\wp'-t)=-3$ for any $t\in\Ca$). It follows that \begin{equation}\label{eq:delta-zeta-summable-zdef}\zeta:=c^{-1}\cdot\frac{\wp'-t}{\wp-r}\in\mathcal{L}([s]+[\idec])-\Ca.\end{equation} Finally, computing $\delta(\zeta)$ given as in \eqref{eq:delta-zeta-summable-zdef} by means of the action of $\delta$ specified in \eqref{eq:delta-zeta-summable-derivatives}, and using the formula for $\tau^{-1}(\wp)$ in \eqref{eq:delta-zeta-summable-tau}, we find that \[\delta(\zeta)-\wp+\tau^{-1}(\wp)=\frac{W(r,t)-W(\wp,\wp')}{(\wp-r)^2}=0.\]
Or equivalently, $\delta(\zeta)=\wp-\tau^{-1}(\wp)$, which is indeed summable.\end{proof}

\begin{remark}\label{rem:delta-zeta-summable} The computational result of Lemma~\ref{lem:delta-zeta-summable} is much easier to see if one is able and willing to appeal to a complex analytic uniformization $\ecl$ as in Section~\ref{sec:weierstrass} or a rigid analytic uniformization $\ecq$ as in Section~\ref{sec:tate}. This is due to the more straightforward computations in these cases: \begin{align*}
\frac{d}{dz}\bigl(\zl(z-\alpha-s)-\zl(z-\alpha)\bigr) &=\wpl(z-\alpha)-\wpl(z-\alpha-s);\\
z\frac{d}{dz}\Bigl(\zq\bigl(s^{-1}\alpha^{-1}z\bigr)-\zq\bigl(\alpha^{-1}z\bigr)\Bigr)&=\wp_q\bigl(s^{-1}\alpha^{-1}z\bigr)-\wp_q(\alpha^{-1}z\bigr);\end{align*}
(see \eqref{eq:zeta-wp-defs} and Lemma~\ref{lem:l-zeta-difference}, and \eqref{eq:q-wp-def} and Lemma~\ref{lem:q-zeta-difference}, respectively).

It would seem desirable to have a more conceptual proof of Lemma~\ref{lem:delta-zeta-summable}, say by applying Theorem~\ref{thm:maina} to $f=\delta\bigl(\za{1}{0}\bigr)$, for the $\za{1}{0}\in\mathcal{L}([s]+[\idec])$ in Definition~\ref{deflem:a-zeta-difference}, instead of the $\zeta$ defined in \eqref{eq:delta-zeta-summable-zdef}. This approach encounters some difficulties that we do not know how to overcome. We may choose a $\tau$-pinning $\mathbf{P}=(\mathcal{R},\mathcal{U})$ of $\ec$ as in Definition~\ref{def:pinninga} such that $\alpha_{\Z s}=\idec$ and $\nu_\idec(\varpi-du_\idec)\geq 1$ (cf.~Proposition~\ref{prop:diff-pano}). The $\tau$-compatibility $\tau(u_s)=u_\idec$ easily implies that $c^\mathcal{U}_k\Bigl(\delta\bigl(\za{1}{0}\bigr),s\Bigr)=-c^\mathcal{U}_k\Bigl(\delta\bigl(\za{1}{0}\bigr),\idec\Bigr)$ for $k=1,2$. Thus, since $\delta\bigl(\za{1}{0}\bigr)\in\mathcal{L}(2[s]+2[\idec])$, every orbital residue of $\delta\bigl(\za{1}{0}\bigr)$ vanishes. Moreover, Proposition~\ref{prop:diff-pano} implies that $\pano^\mathbf{P}\Bigl(\delta\bigl(\za{1}{0}\bigr),1\Bigr)=0$ also, since every residue of the exact differential $d\za{1}{0}=\delta\bigl(\za{1}{0}\bigr)\varpi$ vanishes. In any one example of $\ec$ and $\tau$-pinning $\mathbf{P}$ as above, one could verify the (necessarily true) fact that $\pano^\mathbf{P}\Bigl(\delta\bigl(\za{1}{0}\bigr),0\Bigr)=0$ also. But we do not see how to prove this in general directly from our Definition~\ref{def:pano-a}. What makes results like Theorem~\ref{thm:additive-integrability} useful, both in theory and in practice, is that they allow one to bypass the consideration of the order $0$ panorbital residue, which is by far the most complicated of all the (pan)orbital obstructions to summability.\end{remark}

The following Lemma~\ref{lem:p-power-basic-algebraic} and its consequence Lemma~\ref{lem:phi-p} are only needed to prove Theorem~\ref{thm:additive-integrability} in case $\mathrm{char}(\Ca)>0$.

    \begin{lemma}\label{lem:p-power-basic-algebraic} Suppose $\mathrm{char}(\Ca)=:p>0$. Let $f\in\bke$, $\alpha\in\Ca^\times$, and $\omega\in\orbita$.
    \begin{enumerate}
        \item $c_{pk}(f^p,\alpha)=c_k(f,\alpha)^p$ for every $k\geq 1$.
        \item $c_{k}(f^p,\alpha)=0$ whenever $p\nmid k$.
        \item $\ores(f^p,\omega,pk)=\ores(f,\omega,k)^p$ for every $k\geq 1$.
        \item $\ores(f^p,\omega,k)=0$ whenever $p\nmid k$.
    \end{enumerate}
    \end{lemma}

    \begin{proof}
        Facts (1) and (2) are immediate from the definition \eqref{eq:laurent-coeffs-a}, and they immediately imply facts (3) and (4) by the Definition~\ref{def:oresa}.
    \end{proof}

    \begin{lemma}\label{lem:phi-p} Suppose $\mathrm{char}(\Ca)=:p>0$. Let $\mathbf{P}=(\mathcal{R},\mathcal{U})$ be a $\tau$-pinning of $\ec$ as in Definition~\ref{def:pinninga}, and consider the elements $\varphi^\mathbf{P}_{\omega,0,k}\in\bk$ from Definition~\ref{deflem:phia} and the structural constants $d^\mathcal{U}_{\omega,k}\in\Ca$ from Definition~\ref{def:structural-constants} for $\omega\in\orbita$ and $k\geq 1$. Then we have the relation  
    \begin{equation}\label{eq:phi-p}\varphi^\mathbf{P}_{\omega,0,pk}=\bigl(\varphi^\mathbf{P}_{\omega,0,k}\bigr)^p+\bigl(d^\mathcal{U}_{\omega,k}\bigr)^p\cdot\hat{\varphi}^\mathcal{U}_{\Z s,0,p},\end{equation}
where the $\hat{\varphi^\mathcal{U}_{\Z s,0,k}}\in\bk$ are as in Definition~\ref{deflem:phia} for $k\geq 2$.\end{lemma}
\begin{proof}
   This follows from the uniqueness properties of $\varphi^\mathbf{P}_{\omega,0,pk}$ in Lemma~\ref{deflem:phia}, together with Lemma~\ref{lem:p-power-basic-algebraic}(1).
\end{proof}

The following interesting consequence of Lemma~\ref{lem:phi-p} is not needed to prove Theorem~\ref{thm:additive-integrability}, but we record it here for future reference.

\begin{corollary}\label{cor:phi-p} Suppose $\mathrm{char}(\Ca)=p>0$. Let $\mathcal{U}$ be a $\tau$-compatible set of local uniformizers as in Definition~\ref{def:uniformizers}. Then the structural constants $d^\mathcal{U}_{\omega,k}\in\Ca$ of Definition~\ref{def:structural-constants} satisfy \[d^\mathcal{U}_{\omega,pk}=\bigl(d^\mathcal{U}_{\omega,k}\bigr)^pd^\mathcal{U}_{\Z s,p}\] for every $\omega\in\orbita$ and $k\geq 1$.
\end{corollary}

\begin{proof}
    This follows immediately from \eqref{eq:phi-p} and the definitions.
\end{proof}

Finally, the proof of Theorem~\ref{thm:additive-integrability} is given below.

\begin{proof}[Proof of Theorem~\ref{thm:additive-integrability}] Let us extend the choice of $\mathcal{U}$ in the statement of Theorem~\ref{thm:additive-integrability} to that of a $\tau$-pinning $\mathbf{P}=(\mathcal{R},\mathcal{U})$ of $\ec$ as in Definition~\ref{def:pinninga}, with $\mathcal{R}=\{\alpha_\omega\}_{\omega\in\orbita}$ such that $\alpha_{\Z s}=\idec$. Regardless of whether $\delta(f)$ is summable or not and whether conditions (1) and (2) are satisfied or not, we may replace $f$ with the reduced form $\bar{f}$ as in Remark~\ref{rem:reduced-form-a} without loss of generality. Then the ``reduction'' \begin{equation}\label{eq:df-reduced}\overline{\delta(f)}:=\delta\left(\sum_{\omega\in\orbita}\sum_{k\geq 1}\ores^\mathcal{U}(f,\omega,k)\cdot\varphi^\mathbf{P}_{\omega,0,k}\right)\end{equation}  has the property that \[\overline{\delta(f)}-\delta(f)=\delta\Bigl(\pano^\mathbf{P}(f,0)+\pano^\mathbf{P}(f,1)\za{1}{0}\Bigr)\] is summable by Lemma~\ref{lem:delta-zeta-summable}. Thus $\overline{\delta(f)}$ and $\delta(f)$ are either both summable or both not summable. Moreover, since $\overline{\delta(f)}$ has no poles outside of $\mathcal{R}$, we see that $\overline{\delta(f)}$ is summable if and only if $\overline{\delta(f)}=0$, by Lemma~\ref{lem:summable-dispersion}. It remains to show that $\overline{\delta(f)}=0$ if and only if conditions (1) and (2) are satisfied.

In case $\mathrm{char}(\Ca)=p=0$, condition (2) is automatic. By Theorem~\ref{thm:dhrs-additive}, if $\delta(f)$ is summable then condition (1) is satisfied. On the other hand, if condition (1) is satisfied then $\overline{\delta(f)}=0$, whence $\delta(f)$ is summable.

In case $\mathrm{char}(\Ca)=p>0$, $\overline{\delta(f)}=0$ if and only if
\begin{equation}\label{eq:df-reduced-pth-power}
    \sum_{\omega\in\orbita}\sum_{k\geq 1}\ores^\mathcal{U}(f,\omega,k)\cdot\varphi^\mathbf{P}_{\omega,0,k}=\phi^p\qquad\text{for some}\qquad \phi\in\bk.
\end{equation} This is because $\Ca$ is algebraically closed and therefore $\mathrm{ker}(\delta)=\bk^p$, the subfield of $p$-th powers in $\bk$ (note that this would also hold if $\Ca$ were only perfect). It follows from the Definition~\ref{deflem:phia} of the $\varphi^\mathbf{P}_{\omega,n,k}\in\bk$ that if \eqref{eq:df-reduced-pth-power} holds (i.e., if $\delta(f)$ is summable) then condition (1) is satisfied, because $\nu_\alpha(\phi^p)\in p\Z$ for every $\alpha\in\ec(\Ca)$ and $\phi\in\bk$. Thus it remains to show that, under the assumption that condition (1) holds, \eqref{eq:df-reduced-pth-power} and condition (2) are equivalent. By Lemma~\ref{lem:phi-p} we can write \begin{multline}\label{eq:f-reduced-pth-power}\sum_{\omega\in\orbita}\sum_{k\geq 1}\ores^\mathcal{U}(f,\omega,pk)\cdot\varphi^\mathbf{P}_{\omega,0,pk}=\left(\sum_{\omega\in\orbita}\sum_{k\geq 1}\ores^\mathcal{U}(f,\omega,pk)^{\frac{1}{p}}\cdot\varphi^\mathbf{P}_{\omega,0,k}\right)^p\\+\left(\sum_{\omega\in\orbita}\sum_{k\geq 1}\bigl(d^\mathcal{U}_{\omega,k}\bigr)^p\ores^\mathcal{U}(f,\omega,pk)\right)\hat{\varphi}_{\Z s,0,p}.\end{multline} Since $\hat{\varphi}_{\Z s,0,p}\notin \bk^p$, it follows that indeed whenever condition (1) is satisfied, \eqref{eq:df-reduced-pth-power} holds if and only if condition (2) holds.
\end{proof}

\subsubsection{Multiplicative systems}

Next we address the case of a first-order difference system as in \eqref{eq:mat-difeq}:
\begin{equation}\label{eq:multiplicative-mat}
    \tau(y)=ay
\end{equation} for some $0\neq a\in\bk$. Although seemingly simpler than the higher-order additive systems of the form~\eqref{eq:additive-mat}, but it turns out that these first-order systems present some new challenges. We are able to provide a satisfactory characterization of the $\delta$-integrability of \eqref{eq:multiplicative-mat} only in case $\mathrm{char}(\Ca)=0$ (see Theorem~\ref{thm:dlog-integrability-0} below). In case $\mathrm{char}(\Ca)$ our results are less definitive.

The $\delta$-integrability condition \eqref{eq:scal-integrability} for \eqref{eq:multiplicative-mat} is equivalent to the summability of $f:=\delta(a)/a$, which by Theorem~\ref{thm:maina} is in turn equivalent to the vanishing of the orbital and panorbital residues of this $f$. Since $f$ depends only on the divisor $\mathrm{div}(a):=\sum_{\alpha\in\ec(\Ca)}\nu_\alpha(a)$, one would like to have a more meaningful characterization of the $\delta$-integrability of \eqref{eq:multiplicative-mat} in terms of $\mathrm{div}(a)$ itself. The following result represents partial progress in this direction. We adopt the usual convention that $m\equiv n\pmod{0}$ if and only if $m=n$. 

\begin{proposition}\label{prop:dlog-integrability} Let $\mathrm{char}(\Ca)=:p\geq 0$. Let $\mathbf{P}=(\mathcal{R},\mathcal{U})$ be a $\tau$-pinning of $\ec$ as in Definition~\ref{def:pinninga}, with $\mathcal{R}=\{\alpha_\omega\}_{\omega\in\orbita}$, and let $a\in\bk^\times$. Then $\delta(a)/a$ is summable if and only if the following conditions hold.
    \begin{enumerate}
        \item $\displaystyle \sum_{\alpha\in\omega}\nu_\alpha(a)\equiv 0\pmod{p}$ for each $\omega\in\orbita$;
        \item $\displaystyle \sum_{\omega\in\orbita}\sum_{n\in\Z}n\cdot\nu_{\alpha_\omega\oplus ns}(a)\equiv 0\pmod{p}$; and
        \item $\pano^\mathbf{P}\left(\dfrac{\delta(a)}{a},0\right)=0$.
    \end{enumerate}
\end{proposition}

\begin{proof}
     Let us write $f:=\frac{\delta(a)}{a}$. We first prove this in the special case where $\mathcal{U}=\{u_\alpha\}_{\alpha\in\ec(\Ca)}$ is supercompatible, in the sense of Remark~\ref{rem:super-compatible-d-independence}, which we can choose such that $\nu_\alpha(du_\alpha-\varpi)\geq 1$ for \emph{every} $\alpha\in\ec(\Ca)$, rather than just $\alpha=\idec$ as in Proposition~\ref{prop:diff-pano}. Then \[c^\mathcal{U}_1(f,\alpha)=\mathrm{res}\left(\frac{da}{a},\alpha\right)\equiv\nu_\alpha(a) \pmod{p}\] for each $\alpha\in\ec(\Ca)$. Since $f$ has at worst first-order poles, condition~(1) is equivalent to the vanishing of every $\ores^\mathcal{U}(f,\omega,k)$. Similarly, it follows from Proposition~\ref{prop:diff-pano} that condition~(2) is equivalent to the vanishing of $\pano^\mathbf{P}(f,1)$. Thus by Theorem~\ref{thm:maina}, this concludes the proof of the Proposition~\ref{prop:dlog-integrability} in the special case, and also proves that the summability of $f$ implies all the conditions in the general case.
     
      It follows from the triangularity of the linear system \eqref{eq:u-ores-effect} that the property that all the orbital residues of $f$ vanish is independent of the choice of $\mathcal{U}$. Thus if (1) holds then every $\ores^\mathcal{U}(f,\omega,k)$ vanishes with respect to any $\mathcal{U}$, since they would vanish for the special choice. Similarly, by Lemma~\ref{lem:pano-pinning-au}, the vanishing of $\pano^\mathbf{P}(f,1)$ is independent of the choice of $\mathcal{U}$. So once again, if condition~(2) holds then $\pano^\mathbf{P}(f,1)=0$ regardless of the choice of $\mathcal{U}$, since it would vanish for the special choice. Finally, although the vanishing of $\pano^\mathbf{P}(f,0)$ is \emph{not} independent of the choice of $\mathcal{U}$ in general, this vanishing is independent of the choice of $\mathcal{U}$ for those $f\in\bk$ such that its orbital residues as well as its first-order panorbital residue all vanish, as we see from Lemma~\ref{lem:pano-pinning-au} and Corollary~\ref{cor:pano-pinning-au}. Thus if conditions (1), (2), and (3) hold for our special choice of $\mathcal{U}$ if and only if they all hold for any other choice of $\mathcal{U}$, thus reducing the general case to the special case already proved.
    \end{proof}
    
    The statement of Proposition~\ref{prop:dlog-integrability} above is somewhat unsatisfactory, due to the occurrence of the condition (3) that $\pano^\mathbf{P}(\delta(a)/a,0)=0$. One expects that this condition should be replaced with another simpler equivalent condition whose dependence on the data in $\mathrm{div}(a)$ is less mysterious. In fact, in characteristic zero the condition is entirely superfluous.

    \begin{theorem}\label{thm:dlog-integrability-0} Suppose $\mathrm{char}(\Ca)=0$. Let $\mathcal{R}=\{\alpha_\omega\}_{\omega\in\orbita}$ be a set of \mbox{$\tau$-orbit} representatives in $\ec(\Ca)$, as in Definition~\ref{def:pinninga}, and let $a\in\bk^\times$. Then $\delta(a)/a$ is summable if and only if the following conditions hold.
        \begin{enumerate}
            \item $\displaystyle \sum_{\alpha\in\omega}\nu_\alpha(a)=0$ for each $\omega\in\orbita$; and
            \item $\displaystyle \sum_{\omega\in\orbita}\sum_{n\in\Z}n\cdot\nu_{\alpha_\omega\oplus ns}(a)=0$.
        \end{enumerate}
        \end{theorem}

        \begin{proof} That conditions (1) and (2) hold if $\delta(a)/a$ is summable is already proved in Proposition~\ref{prop:dlog-integrability}. We prove the opposite implication first in the special case where $\Ca=\C$. In this case we may choose a Tate parameterization $\ecq$ of $\ec$ (see \cite[\S V.1]{silvermanAdvanced}). Choose a lift $\hat{s}\in\C^\times$ such that $\check{\hat{s}}=s\in\ec(\C)$. Let us also upgrade the $\mathcal{R}$ in the statement to a $\tau$-pinning $\hat{\mathcal{R}}=\{\hat{\alpha}_\omega\}_{\omega\in\orbitq}\subset\C^\times$ of $\ecq$ in the sense of Definition~\ref{def:pinningq}, such that $\check{\hat{\alpha}}_\omega=\alpha_\omega\in\ec(\C)$ for each orbit $\omega\in\orbita=\orbitq$. It is well-known (see e.g.~\cite[p.~14]{Roquette}) that we may then express \begin{equation}\label{eq:theta-expansion}
        \begin{gathered}
            a=\gamma
            z^{-m}\prod_{\omega\in\orbitq}\prod_{n\in\Z}\theta_q(\hat{s}^{-n}\hat{\alpha}_\omega^{-1}z)^{\nu_{\alpha_\omega\oplus ns}(a)};\qquad \text{where} \\
             m\in\Z \qquad\text{such that}\quad
            \prod_{\omega\in\orbitq}\prod_{n\in\Z}(\hat{s}^n\hat{\alpha}_\omega)^{\nu_{\alpha_\omega\oplus ns}(a)}=q^m;
            \end{gathered}
        \end{equation} $\gamma\in\C^\times$ is some irrelevant constant; and the theta function $\theta_q(z)$ in \eqref{eq:q-theta-def} satisfies the functional equation $\theta_q(qz)=-z^{-1}\theta_q(z)$. Then for $\delta=z\frac{d}{dz}$ the Euler derivation, we have that \[\frac{\delta(a)}{a}=m-\sum_{\omega\in\orbitq}\sum_{n\in\Z}\nu_{\alpha_\omega\oplus ns}(a)\cdot\zq(\hat{s}^{-n}\hat{\alpha}_\omega^{-1}z).\] We again confirm in this setting that conditions (1) and (2) respectively correspond to the vanishing of all the (Tate) orbital residues and the (Tate) first-order panorbital residue. But now we also obtain that $\pano^{\hat{\mathcal{R}}}\bigl(\frac{\delta(a)}{a},0\bigr)=m$. It remains to show that conditions (1) and (2) imply that $m=0$ already. To see this, let us write $f:=\frac{\delta(a)}{a}$ as before, and simply note that \[q^m=\left(\prod_{\omega,n}\hat{s}^{n\nu_{\alpha_\omega\oplus ns}(a)}\right)\cdot\left(\prod_{\omega,n}\hat{\alpha}_\omega^{\nu_{\alpha_\omega\oplus ns}(a)}\right)=\hat{s}^{\pano^{\hat{\mathcal{R}}}(f,1)}\cdot\prod_{\omega}\hat{\alpha}_\omega^{\ores(f,\omega,1)}=1.\]
        
        The general statement for our arbitrary algebraically closed field $\Ca$ of characteristic zero follows from the usual ``Lefschetz principle''-type argument: one can replace $\Ca$ with a countable algebraically closed subfield containing all the necessary data (the coordinates of $s$, the coefficients of the Weierstrass equation \eqref{eq:alg-weierstrass-eq}, the coefficients of $a$ as a rational function in $\Ca(\wp,\wp')$), and embed this field in $\C$. The same equation \eqref{eq:alg-weierstrass-eq} defines now an algebraic elliptic curve $\ec_\C$ over $\C$, which one can analytify to obtain the Riemann surface $\ec_\C^\mathrm{an}$, which admits a Tate parameterization as above. We can write $f=\frac{\delta(a)}{a}$ uniquely as \[f=F_1(\wp)+F_2(\wp)\cdot\wp',\] where $F_1(X),F_2(X)\in\Ca(X)$ are rational functions. The analytic proof of the existence of $g(z)\in\C(\wp_q(z),\wp_q'(z))=\C(\wp,\wp')$ such that $\tau(g)-g=f$ implies that there exist rational functions $G_1(X),G_2(X)\in\C(X)$ such that \begin{equation}\label{eq:formally-summable}\bigl(F_1(X),F_2(X)\bigr)=T\bigl(G_1(X),G_2(X)\bigr)-\bigl(G_1(X),G_2(X)\bigr),\end{equation} where $T:\Ca(X)\times\Ca(X)$ is the $\Ca$-rational transformation formally deduced from the algebraic addition-by-$s$ formulas on $\ec$ \cite[\S III.2]{SilvermanIntro} (see also \eqref{eq:delta-zeta-summable-tau}). After replacing the $\C$-coefficients in the rational expressions for $g=G_1(\wp)+G_2(\wp)\wp'$ with formal unknowns, the relation \eqref{eq:formally-summable} defines an affine (not necessarily irreducible) algebraic variety over $\Ca$. Having proved that this variety has a $\C$-point, it must also have a $\Ca$-point, since $\Ca$ is algebraically closed. Therefore if $f$ is summable in $\C(\wp,\wp')$, then $f$ is also summable in $\Ca(\wp,\wp')=\bk$.
        \end{proof}

        In fact, the same argument above also works in arbitrary characteristic, under some supplementary assumptions on $\Ca$ and on $\ec$, to show a related one-directional result.

        \begin{corollary} \label{cor:dlog-integrability-0}
            Let $\mathcal{R}=\{\alpha_\omega\}_{\omega\in\orbita}$ be a set of \mbox{$\tau$-orbit} representatives in $\ec(\Ca)$, as in Definition~\ref{def:pinninga}, and let $a\in\bk^\times$. Suppose moreover that $\Ca$ admits a non-Archimedean absolute value $| \ \ |$ with respect to which the $j$-invariant $|j(\ec)|>1$, and that the following conditions hold:
        \begin{enumerate}
            \item $\displaystyle \sum_{\alpha\in\omega}\nu_\alpha(a)=0$ for each $\omega\in\orbita$; and
            \item $\displaystyle \sum_{\omega\in\orbita}\sum_{n\in\Z}n\cdot\nu_{\alpha_\omega\oplus ns}(a)=0$.
        \end{enumerate}
        Then $\delta(a)/a$ is summable.
        \end{corollary}

        \begin{proof}
            As in the proof of Theorem~\ref{thm:dlog-integrability-0}, we first prove the result in the special case where $\Ca$ is already complete with respect to $| \ \ |$. The non-integrality condition on the $j$-invariant of $\ec$ is precisely what is needed in order for $\ec$ to admit a (unique!) Tate parameterization as $\ecq$ \cite[Thm.~VII]{Roquette}. After replacing every instance of $\C$ with $\Ca$, the same argument in the first part of the proof of Theorem~\ref{thm:dlog-integrability-0} proves our result in this special case. The proof of the general case follows from the special case by the same argument as in the second part of the proof of Theorem~\ref{thm:dlog-integrability-0}, after replacing every instance of $\C$ with the completion $\boldsymbol{\Ca}$ of $\Ca$ with respect to $| \ \ |$.
        \end{proof}
        
    Two first-order systems $\tau(y)=ay$ and $\tau(y)=a'y$ as in \eqref{eq:multiplicative-mat} are said to be \emph{gauge equivalent} is there exists $r\in\bk^\times$ such that $a=a'\frac{\tau(r)}{r}$. The following concluding result, which is again valid in arbitrary characteristic, completely characterizes when the first-order system \eqref{eq:multiplicative-mat} is gauge equivalent to a constant system $\tau(y)=cy$ with $c\in\Ca^\times$, which is a strictly stronger property than that of being only $\delta$-integrable (see Remark~\ref{rem:constant-gauge-equivalent} below).

 \begin{theorem}\label{thm:constant-gauge-equivalent}
     Let $\mathcal{R}=\{\alpha_\omega\}_{\omega\in\orbita}$ be a set of $\tau$-orbit representatives in $\ec(\Ca)$, as in Definition~\ref{def:pinninga}, and let $a\in\bk^\times$. There exist $r\in\bk^\times$ and $c\in\Ca^\times$ such that $a=c\cdot\tau(r)/r$ if and only if the following conditions hold:
     \begin{enumerate}
         \item $\displaystyle\sum_{\alpha\in\omega}\nu_\alpha(a)=0$ for each $\omega\in\orbita$;
         \item $\displaystyle \sum_{\omega\in\orbita}\sum_{n\in\Z}n\cdot\nu_{\alpha_\omega\oplus ns}(a)=0$; and
         \item $\displaystyle\sum_{\omega\in\orbita}\sum_{n\in\Z}n\cdot\nu_{\alpha_\omega\oplus ns}(a)\cdot\alpha_{\omega} \oplus \displaystyle\sum_{\omega\in\orbita}\sum_{n\in\Z}\dfrac{n(n+1)}{2}\cdot \nu_{\alpha_{\omega}\oplus ns}(a)\cdot s=\idec$.
     \end{enumerate}
 \end{theorem}

 \begin{proof} Let us denote by $\tau^*$ the additive endomorphism of $\mathrm{Div}(\ec)$ defined by $\tau^*([\alpha])=[\alpha\ominus s]$, so that $\tau^*(\mathrm{div}(a))=\mathrm{div}(\tau(a))$. Consider the divisor
     \begin{equation}\label{eq:gauge-divisor-def}D:=\sum_{\omega\in\orbita}\left(\sum_{n\geq 1}\nu_{\alpha_\omega\oplus ns}(a)\sum_{i=1}^n\bigl[\alpha_\omega\oplus is\bigr]-\sum_{n\leq -1}\nu_{\alpha_\omega\oplus ns}(a)\sum_{i=n+1}^{0}\bigl[\alpha_\omega\oplus is\bigr]\right)\end{equation} 
     
    We claim that $D-\tau^*(D)=\mathrm{div}(a)$ if and only if condition $(1)$ holds. It suffices to prove that $D_\omega-\tau^*(D_\omega)=\mathrm{div}_\omega(a)$, where the $\omega$-component $D_\omega$ of $D$ is as in Definition~\ref{def:tau-divisors}(3), and $\mathrm{div}_\omega(a)$ similarly denotes the $\omega$-component of $\mathrm{div}(a)$. Indeed, since we can also write  
     \begin{align*}
         D_\omega &= \sum\limits_{i\geq 1}[\alpha_{\omega}\oplus is]\sum\limits_{n\geq i} \nu_{\alpha_{\omega}\oplus ns}(a) - \sum\limits_{i\leq 0}[\alpha_{\omega}\oplus is]\sum\limits_{n\leq i-1} \nu_{\alpha_{\omega}\oplus ns}(a);\\
         \tau^*(D_\omega) &= \sum\limits_{i\geq 0}[\alpha_{\omega}\oplus is]\sum\limits_{n\geq i+1} \nu_{\alpha_{\omega}\oplus ns}(a) - \sum\limits_{i\leq -1}[\alpha_{\omega}\oplus is]\sum\limits_{n\leq i} \nu_{\alpha_{\omega}\oplus ns}(a);
     \end{align*}
     we find that the difference 
     \[D_\omega-\tau^*(D_\omega) = [\alpha_{\omega}]\sum\limits_{i\neq 0} (-\nu_{\alpha_{\omega}\oplus is}(a)) + \sum\limits_{i\neq 0} [\alpha_{\omega}\oplus is]\nu_{\alpha_{\omega}\oplus is}(a).\]
     This is precisely the divisor of $a$ if and only if $\sum\limits_{i\neq 0}(-\nu_{\alpha_{\omega}\oplus is}(a)) = \nu_{\alpha_{\omega}}(a)$ for each orbit $\omega\in\orbita$, which is equivalent to condition $(1)$. We also claim that $\deg(D)=0$ if and only if condition (2) holds. Indeed, in any case we have 
     \[\mathrm{deg}(D)=\sum\limits_{\omega\in\orbita}\sum\limits_{n\in\mathbb{Z}} n\cdot \nu_{\alpha_{\omega}\oplus ns}(a).\]
    Finally, we see that condition (3) holds if and only if the evaluation $\mathrm{ev}(D)=\idec$. Indeed, in any case we have  \[\mathrm{ev}(D)=\displaystyle\sum_{\omega\in\orbita}\sum_{n\in\Z}n\cdot\nu_{\alpha_\omega\oplus ns}(a)\cdot\alpha_{\omega} \oplus \displaystyle\sum_{\omega\in\orbita}\sum_{n\in\Z}\dfrac{n(n+1)}{2}\cdot \nu_{\alpha_{\omega}\oplus ns}(a)\cdot s.\]
    
    Thus by the Abel-Jacobi Theorem \cite[Cor.~III.3.5]{SilvermanIntro}, if conditions (2) and (3) hold then there exists $r\in\bk^\times$ such that $\mathrm{div}(r)=-D$ as in \eqref{eq:gauge-divisor-def}, and if we further have condition (1) then $\mathrm{div}\bigl(a\cdot\frac{r}{\tau(r)}\bigr)=\mathrm{div}(r)-D+\tau^*(D)=0$, i.e., $c:=a\cdot\frac{r}{\tau(r)}\in\Ca^\times$.

    Suppose on the other hand that there exist $r\in\bk$ and $c\in\Ca^\times$ such that $a=c\cdot\frac{\tau(r)}{r}$. Then for each $\omega\in\orbita$, we must have \[\sum_{\alpha\in\omega}\nu_\alpha(a)=\sum_{\alpha\in\omega}(\nu_\alpha(\tau(r))-\nu_\alpha(r))=\mathrm{deg}(\mathrm{div}_\omega(r))-\mathrm{deg}(\mathrm{div}_\omega(r))=0.\] Hence condition (1) is satisfied, and therefore the divisor $D$ in \eqref{eq:gauge-divisor-def} satisfies \[D-\tau^*(D)=\mathrm{div}(a)=\mathrm{div}(\tau(r))-\mathrm{div}(r)=(-\mathrm{div}(r))-\tau^*(-\mathrm{div}(r)).\] It follows that $D+\mathrm{div}(r)$ is fixed by $\tau^*$. But the only fixed point of $\tau^*$ on $\mathrm{Div}(\ec)$ is the trivial divisor $0$, and therefore $D=-\mathrm{div}(r)$. Then $\mathrm{deg}(D)=-\mathrm{deg}(r)=0$ and $\mathrm{ev}(D)=\ominus\mathrm{ev}(\mathrm{div}(r))=\idec$, again by the Abel-Jacobi Theorem \cite[Cor.~III.3.5]{SilvermanIntro}, which as we saw above implies that conditions (2) and (3) are also satisfied. 
 \end{proof}

\begin{remark}\label{rem:constant-gauge-equivalent}
    It is easy to see that if \eqref{eq:multiplicative-mat} is equivalent to a constant system then it is $\delta$-integrable. Indeed, if $a=c\cdot\frac{\tau(r)}{r}$ as in Theorem~\ref{thm:constant-gauge-equivalent} then $b=\frac{\delta(r)}{r}$ witnesses the $\delta$-integrability condition \eqref{eq:scal-integrability} for $a_0=a$. 
    
    For difference systems as in \eqref{eq:mat-difeq} but over the projective line over $\C$, with $\bk$ replaced with $\C(x)$, and with $\tau$ replaced with one of the ``standard'' difference operators $x\mapsto x+1$ (shift); or $x\mapsto qx$ for some $q\in\Ca^\times$ of infinite multiplicative order; or $x\mapsto x^m$ for some $m\in\Z_{\geq 2}$ (Mahler); it is proved in the landmark result of \cite{schaefke-singer:2019} that a suitable notion of differential integrability of the difference system is equivalent to its gauge equivalence to a constant system. Based on this, one might have expected a similar equivalence to hold also for difference equations over elliptic curves.
    
    But as we first learned from Hardouin and Roques \cite{hardouin-roques:2023}, for a difference system \eqref{eq:mat-difeq} over an elliptic curve $\ec$ such as we consider here, its differential integrability no longer implies that it is gauge equivalent to a constant system. In \cite{hardouin-roques:2023} they showed us the fundamental counterexample to this: the additive system of the form \eqref{eq:additive-mat} with $f$ the elliptic function in $\bkl$ described in Lemma~\ref{lem:l-zeta-difference}: \begin{equation}\label{eq:hr-example}\tau(Y)=\begin{pmatrix}
        1 & \zl(z-s)-\zl(z) \\ 0 & 1
    \end{pmatrix}Y,\end{equation} which is $\frac{d}{dz}$-integrable (as we now see more generally from Theorem~\ref{thm:additive-integrability}, since all of its orbital residues vanish), but one can show that it is not gauge equivalent to a constant system. In fact, an additive system of the form $\tau(Y)=\left(\begin{smallmatrix}1 & f \\ 0 & 1\end{smallmatrix}\right)Y$ as in \eqref{eq:additive-mat} if and only if there exists $g\in\bk$ such that $f=\tau(g)-g+c$ for some $c\in\Ca$. As we now understand from our Theorem~\ref{thm:maina} (cf.~Remark~\ref{rem:reduced-form-a}) this occurs precisely when every (pan)orbital residue of $f$ vanishes except for $\pano(f,0)=c$. The computation in Example~\ref{ex:zeta-difference-l}, that the order $1$ panorbital residue of $\zl(z-s)-\zl(z)$, ``explains'' why the system \eqref{eq:hr-example} cannot possibly be gauge equivalent to a constant system.
\end{remark}

\section*{Acknowledgements}

The example \eqref{eq:hr-example} was the original source of inspiration for the development of panorbital residues in terms of the expansion \eqref{eq:zetaExp-l} in the second author's doctoral dissertation \cite{Babbitt2025} in the Weierstrass case (as in Section~\ref{sec:weierstrass}), and subsequently for the development in \cite{Babbitt2025} of a prototype of the panorbital residues in the algebraic case that we develop in Section~\ref{sec:algebraic}, and subsequently for the development of the panorbital residues in the Tate case that we carry out in Section~\ref{sec:tate} for the first time. Needless to say, this work would not have been possible without the generosity of Hardouin and Roques in sharing their insights with us in \cite{hardouin-roques:2023}. We are very thankful to them. We are also grateful for illuminating discussions with Eric Rains and with Michael Singer, which deepened our knowledge and helped us sharpen our thinking on these matters.

\end{document}